\documentclass[12pt,psamsfonts]{amsart}
\usepackage{fullpage,amssymb,amsfonts,amsmath,amsthm}
\usepackage[all,cmtip]{xy}
\usepackage{enumerate}
\usepackage{mathrsfs}
\usepackage{comment}
\usepackage{hyperref} 
\usepackage[capitalise]{cleveref}
\usepackage{floatrow}
\usepackage{url}
\usepackage{caption}
\usepackage{varwidth}
\DeclareRobustCommand{\SkipTocEntry}[5]{}

\newtheorem{thm}{Theorem}[section]
\newtheorem{cor}[thm]{Corollary}
\newtheorem{prop}[thm]{Proposition}
\newtheorem{lem}[thm]{Lemma}

\theoremstyle{definition}
\newtheorem{defn}[thm]{Definition}

\theoremstyle{remark}

\newtheorem*{rem*}{Remark}

\numberwithin{equation}{section}
\newcounter{notation}

\DeclareUrlCommand\DOI{}
\newcommand{\doi}[1]{\href{http://dx.doi.org/#1}{\DOI{doi:#1}}}



\crefname{figure}{Figure}{Figures}
\crefname{thm}{Theorem}{Theorems}
\crefname{cor}{Corollary}{Corollarys}
\crefname{cor*}{Corollary}{Corollarys}
\crefname{lem}{Lemma}{Lemmas}
\crefname{prop}{Proposition}{Propositions}
\crefname{defn}{Definition}{Definitions}

\crefname{rem}{Remark}{Remarks}

\def\addsymbol #1: #2#3{$#1$ \> \parbox{5.4in}{#2 \dotfill \pageref{#3}}\\} 
\def\addsymbolEND #1: #2#3{$#1$ \> \parbox{5.4in}{#2 \dotfill \pageref{#3}}}

\newcommand{\ds}{\displaystyle}
\newcommand{\fk}[1]{\mathfrak{#1}}

\renewcommand{\bar}{\overline}

\newcommand{\C}{\mathbb{C}}
\newcommand{\cC}{\mathcal{C}}
\newcommand{\Cl}{\mathrm{Cl}}

\newcommand{\kf}{\mathfrak{f}}

\newcommand{\cK}{\mathcal{K}}

\newcommand{\cL}{\mathcal{L}}

\renewcommand{\Im}{\mathrm{Im}}

\newcommand{\smod}[1]{\, (\mathrm{mod}^*{\, #1} )}
\renewcommand{\pmod}[1]{\, (\mathrm{mod} {\, #1})}
\newcommand{\kn}{\mathfrak{n}}
\newcommand{\N}{\mathrm{N}}

\newcommand{\kp}{\mathfrak{p}}
\newcommand{\cP}{\mathcal{P}}

\newcommand{\cO}{\mathcal{O}}
\newcommand{\ord}{\mathrm{ord}\,}

\newcommand{\kq}{\mathfrak{q}}
\newcommand{\Q}{\mathbb{Q}}

\newcommand{\cR}{\mathcal{R}}
\newcommand{\R}{\mathbb{R}}

\renewcommand{\Re}{\mathrm{Re}}

\newcommand{\cS}{\mathcal{S}}

\newcommand{\cT}{\mathcal{T}}

\newcommand{\cZ}{\mathcal{Z}}

\begin{document}

\title{Explicit estimates for the zeros of Hecke $L$-functions}

\author{Asif Zaman}
\address{
Department of Mathematics, University of Toronto \\
Room 6290, 40 St. George St., M5S2E4, Toronto, ON, Canada}
\email{asif@math.toronto.edu}
\thanks{The author was supported in part by an NSERC PGS-D scholarship.}

\date{}



\begin{abstract} Let $K$ be a number field and, for an integral ideal $\mathfrak{q}$ of $K$, let $\chi$ be a character of the narrow ray class group modulo $\kq$. We establish various new and improved explicit results, with effective dependence on $K, \kq$ and $\chi$, regarding the zeros of the Hecke $L$-function $L(s,\chi)$, such as zero-free regions, Deuring-Heilbronn phenomenon, and zero density estimates.
\end{abstract}

\maketitle

\tableofcontents

\section{Introduction}
Let $K$ be a number field of degree $n_K = [K:\Q]$ with absolute discriminant $d_K = |\mathrm{disc}(K/\Q)|$ and ring of integers $\cO_K$. For an integral ideal $\kq \subseteq \cO_K$, the (narrow) ray class group modulo $\kq$ is defined to be $\Cl(\kq) := I(\kq)/P_{\kq}$
where $I(\kq)$ is the group of fractional ideals of $K$ relatively prime to $\kq$, and $P_{\kq}$ is the subgroup of principal ideals $(\alpha)$ of $K$ such that $\alpha \equiv 1 \smod{\kq}$. Recall $\alpha \equiv 1 \smod{\kq}$ if and only if $\alpha$ is totally positive and $v(\alpha-1) \geq  v(\kq)$ for all discrete valuations $v$ of $K/\Q$. Characters $\chi$ of the ray class group will be referred to as Hecke characters, which we will often denote $\chi \pmod{\kq}$. Let $\ord \chi$ be the multiplicative order of $\chi$ in $\Cl(\kq)$. 

A Hecke character $\chi \pmod{\kq}$ possesses an associated $L$-function defined by
\[
L(s,\chi) := \prod_{\substack{ \kp \nmid \kq} } \Big(1 - \frac{\chi(\kp)}{(\N\kp)^{s}} \Big)^{-1} \quad \text{ for $\sigma > 1$,}
\]
where $s = \sigma+it$, $\N = \N^K_{\Q}$ is the absolute norm on integral ideals of $K$, and the product is over prime ideals $\kp \subseteq \cO_K$. In the special case $\kq = \cO_K$ and $\chi = \chi_0$ the principal character, the associated $L$-function is the Dedekind zeta function of $K$ given by
\[
\zeta_K(s) = \prod_{\kp} \Big(1 - \frac{1}{(\N\kp)^{s}} \Big)^{-1} \quad \text{ for $\sigma > 1$}.
\]
It is well known that the zeros of Hecke $L$-functions are intimately related with the distribution of prime ideals of $K$ amongst the equivalence classes of $\Cl(\kq)$ and, by class field theory, with prime ideal decompositions in abelian extensions of $K$. Indeed, the analytic properties of Hecke $L$-functions have been widely studied by many authors such as Fogels \cite{Fogels}. 

However, the known results on zeros of Hecke $L$-functions do not typically have explicit dependence on the field $K$ with explicit absolute constants.  In the last few years, there has been some progress in this direction with explicit results on the zero-free regions of $\zeta_K(s)$ by Kadiri \cite{Kadiri} and zero-free regions for Hecke $L$-functions by Ahn and Kwon \cite{Ahn-Kwon}. Kadiri and Ng \cite{KadNg} have also proved a form of quantitative Deuring-Heilbronn phenomenon  and explicit zero density estimates for $\zeta_K(s)$, the latter of which was subsequently improved by Trudgian \cite{Trudgian_2015}. The aim of this paper is to provide improved and several new explicit results on the zeros of Hecke $L$-functions such as improved zero-free regions, the Deuring-Heilbronn phenomenon for real Hecke characters, and zero density estimates. 

In the classical case $K = \Q$ and $\kq = (q)$, Hecke $L$-functions are the familiar Dirichlet $L$-functions modulo $q$ and prime ideals in equivalence classes of $\Cl(\kq)$ naturally correspond to primes in arithmetic progressions modulo $q$. There is a vast literature on the zeros of Dirichlet $L$-functions, including many with explicit constants, and of particular importance to us is the landmark paper of Heath-Brown \cite{HBLinnik}.

For the statement of the main theorems, note that $\nu(x)$ is \emph{any fixed} increasing real-variable function $\geq 4$ such that $\nu(x) \gg \log(x+4)$ for $x \geq 1$.

\begin{thm} \label{MainTheorem-ZFR}  Suppose $d_K (\N\kq) n_K^{n_K}$ is sufficiently large and let $r \geq 1$ be an integer. 
Then the function $\ds \prod_{\substack{ \chi \pmod{\kq} \\ \ord \chi \geq r} } L(s,\chi)$ has at most 1 zero, counting with multiplicity, in the rectangle
\[
\sigma \geq 1 - \frac{c}{\log d_K + \tfrac{3}{4}\log \N\kq + n_K \cdot \nu(n_K)}, \qquad |t| \leq 1
\]
where $s = \sigma+it$ and
\[
c = 
\begin{cases}
0.1764 & \text{if $r \geq 6$,} \\
0.1489 & \text{if $r = 5$,}  \\
0.1227 & \text{if $r = 2, 3, 4$,} \\
0.0875 & \text{if $r= 1$.}
\end{cases}
\]
Moreover, if this exceptional zero $\rho_1$ exists, then it and its associated character $\chi_1$ are both real. 
\end{thm}

\begin{cor} \label{Corollary-ZFR}
The Dedekind zeta function $\zeta_K(s)$ has at most 1 zero, counting with multiplicity, in the rectangle
\[
\sigma \geq 1 - \frac{0.0875}{\log d_K + n_K \cdot \nu(n_K)}, \qquad |t| \leq 1,
\]
where $s = \sigma+it$ and provided $d_K n_K^{n_K}$ is sufficiently large. If this exceptional zero exists, it is real. 
\end{cor}

\begin{rem*} For general number fields $K$, it is possible that the exceptional character $\chi_1$ is principal. That is, the Dedekind zeta function $\zeta_K(s)$ may have a real zero exceptionally close to $s=1$.
\end{rem*}
As already mentioned, some explicit results have been shown by Kadiri \cite{Kadiri} and Ahn and Kwon \cite{Ahn-Kwon} for zero-free regions of $L(s,\chi)$ of the form
\begin{equation}
\sigma \geq 1 - \frac{c_0}{\log d_K + \log \N\kq},  \qquad |t| \leq 0.13.
\label{AlternateZFR}
\end{equation}
Note that the dependence on the degree $n_K$ has been ``absorbed" into $\log d_K$.  It has been shown that $L(s,\chi)$ is zero-free (except possibly for one real zero when $\chi$ is real) in the rectangle \eqref{AlternateZFR} for
\[
c_0 = \begin{cases}
0.1149 & \text{if $\ord \chi \geq 5$ \cite{Ahn-Kwon}}, \\
0.1004 & \text{if $\ord \chi = 4$ \cite{Ahn-Kwon}}, \\
0.0662 & \text{if $\ord \chi = 3$ \cite{Ahn-Kwon}},  \\
0.0392 & \text{if $\ord \chi = 2$ and $d_K$ is sufficiently large \cite{Kadiri}}\footnotemark, \\
0.0784 & \text{if $\ord \chi = 1$ and $d_K$ is sufficiently large \cite{Kadiri}.} 
\end{cases}
\]
Note that the results of \cite{Kadiri} also allow for $|t| \leq 1$ to be used in \eqref{AlternateZFR}.\footnotetext{This case is not explicitly written in the cited paper but is directly implied by the case $\ord \chi = 1$.} Comparing the above known values for $c_0$ with \cref{MainTheorem-ZFR}, if a given family of number fields $K$ satisfies
\begin{equation}
n_K = O\Big( \frac{\log(d_K \N\kq)}{\log\log(d_K\N\kq)} \Big),
\label{SmallDegreeAssumption}
\end{equation}
then, for a suitable choice of $\nu(x)$, \cref{MainTheorem-ZFR} is superior to all previously known cases, especially in the $\N\kq$-aspect. A classical theorem of Minkowski states, for any number field $K$, 
\[
n_K = O(\log d_K)
\]
so, unless $n_K$ is unusually large, one would expect that \eqref{SmallDegreeAssumption} typically holds. For example, given a fixed rational prime $p$, one can verify that the family of $p$-power cyclotomic fields $K = \Q(e^{2\pi i/p^m})$ satisfies \eqref{SmallDegreeAssumption}.

We also establish a result,  similar to those of \cite{Graham} and \cite{HBLinnik} for Dirichlet $L$-functions, giving a larger zero-free region but allowing more zeros. 

\begin{thm} \label{MainTheorem-ZFR_2} Suppose $d_K (\N\kq) n_K^{n_K}$ is sufficiently large. 
\noindent
Then $\ds \prod_{\chi \pmod{\kq} } L(s,\chi)$ has at most 2 zeros, counting with multiplicity, in the rectangle
\[
\sigma \geq 1 - \frac{0.2866}{\log d_K+ \tfrac{3}{4} \log \N\kq + n_K \cdot \nu(n_K)} \qquad |t| \leq 1.
\]
Moreover, the Dedekind zeta function $\zeta_K(s)$ has at most 2 zeros, counting with multiplicity, in the rectangle
\[
\sigma \geq 1 - \frac{0.2909}{\log d_K+ n_K \cdot \nu(n_K)} \qquad |t| \leq 1.
\]
\end{thm}

When an exceptional zero $\rho_1$ from \cref{MainTheorem-ZFR} exists, we prove an explicit version of the well-known Deuring-Heilbronn phenomenon.  

\begin{thm} \label{MainTheorem-DH} Let $\chi_1 \pmod{\kq}$ be a real character. Suppose 
\[
\beta_1 = 1-\frac{\lambda_1}{\log d_K + \tfrac{3}{4}\log \N\kq + n_K \cdot \nu(n_K)}
\]
is a real zero of $L(s,\chi_1)$ with $\lambda_1 > 0$ . Then, provided $d_K(\N\kq)n_K^{n_K}$ is sufficiently large (depending on $R \geq 1$ and possibly $\epsilon > 0$), the function $\ds \prod_{\chi \pmod{\kq} } L(s,\chi)$ has only the one zero $\beta_1$, counting with multiplicity, in the rectangle
\[
\sigma \geq 1 - \frac{\min\{ c_1 \log(1/\lambda_1), R\} }{\log d_K + \tfrac{3}{4}\log \N\kq + n_K \cdot \nu(n_K)}  , \qquad |t| \leq 1,
\]
where $s=\sigma+it$ and
\[
c_1 = \begin{cases} 
\tfrac{1}{2}-\epsilon & \text{if $\chi_1$ is quadratic and $\lambda_1 \leq 10^{-10}$}, \\
0.2103 & \text{if $\chi_1$ is quadratic and $\lambda_1 \leq 0.1227$}, \\ 
1-\epsilon & \text{if $\chi_1$ is principal and $\lambda_1 \leq 10^{-5}$}, \\
0.7399 & \text{if $\chi_1$ is principal and $\lambda_1 \leq 0.0875$}. \\ 
\end{cases}
\] 
\end{thm}

In the classical case $K = \Q$ and $\kq = (q)$, Linnik \cite{Linnik2} was the pioneer of the Deuring-Heilbronn phenomenon. For other number fields,
a non-explicit $K$-uniform variant of \cref{MainTheorem-DH} is due to Lagarias-Montgomery-Odlyzko \cite{LMO} in the case of the Dedekind zeta function $\zeta_K(s)$ and to Weiss \cite{Weiss} for general Hecke $L$-functions. Kadiri and Ng \cite{KadNg} have recently established an explicit version of the Deuring-Heilbronn phenomenon for zeros of the Dedekind zeta function $\zeta_K(s)$ with 
\[
c_1 = \begin{cases}  
0.9045 & \text{if $\lambda_1 \leq 10^{-6}$}, \\
0.6546 & \text{if $\lambda_1 \leq 0.0784$}. \\ 
\end{cases}
\] 
Hence, \cref{MainTheorem-DH} improves upon their result when \eqref{SmallDegreeAssumption} holds and when the primary term $c_1 \log(1/\lambda_1)$ dominates, as normally is the case. 

We also establish explicit bounds related to the zero density of Hecke $L$-functions. For $\lambda > 0$, define $N(\lambda)$ to be the number of non-principal characters $\chi \pmod{\kq}$ with a zero in the region
\begin{equation}
\sigma \geq 1-\frac{\lambda}{\log d_K + \tfrac{3}{4}\log \N\kq + n_K \cdot \nu(n_K)}, \qquad |t| \leq 1. 
\label{MainTheorem-ZD-Region}
\end{equation}
In the classical case $K= \Q$ and $\kq = (q)$, this quantity has been analyzed by \cite{Graham,HBLinnik} for a slowly growing range ($\lambda \ll \log\log\log q$) and by \cite{HBLinnik} for a bounded range $(\lambda \leq 2)$. We establish a result in the same vein as the latter. To do so, we require some  technical assumptions.

Let $0 < \lambda \leq 2$ be given. Let $f \in C_c^2([0,\infty))$ have Laplace transform $F(z) = \int_0^{\infty} f(t) e^{-zt} dt$. Suppose $f$ satisfies all of the following:
\begin{equation}
\begin{aligned}
& f(t) \geq 0 \text{ for $t \geq 0$;} \qquad  \Re\{ F(z) \} \geq 0 \text{ for $\Re\{z\} \geq 0;$} \\
& F(\lambda) > \tfrac{1}{3}f(0); \qquad \Big( F(\lambda) - \tfrac{1}{3}f(0) \Big)^2 > \tfrac{1}{3}f(0)\Big(\tfrac{1}{4}f(0) + F(0) \Big). \\
\end{aligned}
\label{MainTheorem-ZD-Condition}
\end{equation}
Then we have the following result.

\begin{thm} \label{MainTheorem-ZD} Let $\epsilon > 0$ and $0 < \lambda \leq 2$. Suppose $f \in C_c^2([0,\infty))$ satisfies \eqref{MainTheorem-ZD-Condition}. Then unconditionally, 
\[
N(\lambda)\leq 
\frac{\Big(\tfrac{1}{4}f(0) + F(0) \Big)\Big( F(0) - \tfrac{1}{12} f(0) \Big)}
{\Big( F(\lambda) - \tfrac{1}{3}f(0) \Big)^2  - \tfrac{1}{3}f(0) \Big(\tfrac{1}{4}f(0)+ F(0) \Big)} + \epsilon
\]
for $d_K(\N\kq)n_K^{n_K}$ sufficiently large depending on $\epsilon$ and $f$. 
\end{thm}
\begin{rem*} Let $\rho_1$ be a certain zero of a Hecke $L$-function $L(s,\chi_1)$ with the property that $\Re\{\rho_1\} \geq \Re\{\rho_{\chi}\}$ for any zero $\rho_{\chi}$ in the rectangle \eqref{MainTheorem-ZD-Region} of any $\chi \pmod{\kq}$. By introducing dependence on $\rho_1$, the bound on $N(\lambda)$ in \cref{MainTheorem-ZD} can be improved. See \cref{sec:ZeroFreeGap_and_BadZeros} for the choice of $\rho_1$ and \cref{NewZDE} for further details. 
\end{rem*}
\cref{MainTheorem-ZD} and its proof are inspired by \cite[Section 12]{HBLinnik} and so similarly, the obtained bounds are non-trivial only for a narrow range of $\lambda$. By choosing\footnote{See the discussion following the proof of \cref{NewZDE} for details.} $f$ roughly optimally, we exhibit a table of bounds derived from \cref{MainTheorem-ZD} below.
\vspace{5mm}

\hspace*{-0.2in}
\begin{tabular}{c|c|c|c|c|c|c|c|c|c|c|c|c|c|c|c|c|c|c|c|c|c} 
$\lambda $ & $.100$ & $.125$ & $.150$ & $.175$ & $.200$ & $.225$ & $.250$ & $.275$ & $.300$ & $.325$ & $.350$ & $.375$ & $.400$ & $.425$ \\ 
 \hline 
$N(\lambda)$ & \vspace*{0.1in}  2 &  2 &  3 &  3 &  4 &  4 &  5 &  6 &  7 &  9 &  11 &  15 &  22 &  46   \\ 
\end{tabular} \\\\
\noindent
One can see that the estimates obtained are comparable to \cref{MainTheorem-ZFR,MainTheorem-ZFR_2} which respectively imply that $N(0.1227) \leq 1$ and $N(0.2866) \leq 2$ . 

In the classical case $K = \Q$, Heath-Brown substantially improved upon all preceding work for zeros of Dirichlet $L$-functions and so, for general number fields $K$, we have taken advantage of the innovations founded in \cite{HBLinnik} to improve on the existing aforementioned results and also to establish new explicit estimates. As such, the general structure of this paper is reminiscent of his work and is subject to small improvements similar to those suggested in \cite[Section 16]{HBLinnik}. Xylouris implemented a number of those suggestions in \cite{Xylouris} so in principle one could refine the results here  by the same methods. 

Finally, we describe the organization of the paper. Section 2 covers well-known facts about Hecke $L$-functions and some elementary estimates. In Section 3, we specify some frequently-used notation and identify zeros of Hecke $L$-functions which will play a key role throughout the paper. Sections 4, 5 and 6 establish several different ``explicit inequalities" related to $-\tfrac{L'}{L}(s,\chi)$ by involving classical arguments, higher derivatives of $-\tfrac{L'}{L}(s,\chi)$, and smooth weights. The results therein form the technical crux of all subsequent proofs and applications. Section 7 provides bounds for the zero density quantity $N(\lambda)$. Section 8 quantifies Deuring-Heilbronn phenomenon for the exceptional case. Section 9 deals with the milder zero repulsion in the non-exceptional case. Section 10 establishes a zero-free region for Hecke $L$-functions.

For the reader who wishes to proceed quickly to the proofs of the theorems:

\begin{itemize}
	\item \cref{MainTheorem-ZFR} and \cref{Corollary-ZFR} are proved in \cref{sec:ZeroFreeRegion}. 
	\item \cref{MainTheorem-ZFR_2} is an immediate corollary of \cref{SZ-L1Lp-Summary,SZ-L1L2-Summary,Bounds-Lp-CC,Bounds-L2-CC}.
	\item \cref{MainTheorem-DH} is an immediate corollary of \cref{SZ-L1Lp-Summary,SZ-L1L2-Summary}. 
	\item \cref{MainTheorem-ZD} is a special case of \cref{NewZDE}. 
\end{itemize}

\addtocontents{toc}{\SkipTocEntry}
\subsection*{Acknowledgements} 
\noindent
It is my pleasure to gratefully acknowledge the support and guidance of my advisor, Prof. John Friedlander, who initially suggested this problem to me and with whom I have had many helpful discussions. I would also like to thank the referee for a careful reading of the paper and many valuable suggestions.

\section{Preliminaries}
\label{sec:Preliminaries}


\subsection{Hecke $L$-functions}
\label{HeckeL-functions}
Recall \emph{Hecke characters} are characters $\chi$ of the ray class group $\Cl(\kq) = I(\kq)/P_{\kq}$. We often write $\chi \pmod{\kq}$ to indicate this relationship. For notational convenience, we pullback the domain of $\chi$ to  $I(\kq)$ and then extend it to all of $I(\cO)$ by zero; that is, $\chi(\kn)$ is defined for all integral ideals $\kn \subseteq \cO$ and $\chi(\kn) = 0$ for $(\kn,\kq) \neq 1$. The \emph{conductor} $\fk{f}_{\chi}$ of a Hecke character $\chi \pmod{\kq}$ is the maximal integral ideal such that $\chi$ is the push-forward of a Hecke character modulo $\fk{f}_{\chi}$. It follows that $\fk{f}_{\chi}$ divides $\kq$. We say $\chi$ is \emph{primitive modulo $\kq$} if $\kf_{\chi} = \kq$. 

Thus, the Hecke $L$-function associated to $\chi \pmod{\kq}$ may be written as
\[
L(s,\chi) = \sum_{\kn \subseteq \cO} \chi(\kn) (\N{\kn})^{-s} = \prod_{\kp} \Big(1-\frac{\chi(\kp)}{(\N{\kp})^s} \Big)^{-1} \qquad \text{for $\sigma > 1$}
\]
where $s = \sigma + it \in \C$. Unless otherwise specified, we shall henceforth refer to Hecke characters as characters.
\subsubsection*{Functional Equation}
\label{subsec:FuncEqn}
Let $\chi \pmod{\kf_{\chi}}$ be a primitive character. Recall that the \emph{L-function of $\chi$ at infinity} is given by
\begin{equation}
L_{\infty}(s,\chi) := \Big[ \pi^{-s/2} \Gamma\Big( \frac{s}{2}\Big) \Big]^{a(\chi)} \cdot \Big[ \pi^{-\frac{s+1}{2}} \Gamma\Big( \frac{s+1}{2} \Big) \Big]^{b(\chi)}
\label{L-Infinite}
\end{equation}
where $\Gamma(s)$ is the Gamma function and $a(\chi), b(\chi)$ are certain non-negative integers satisfying 
\[
a(\chi) + b(\chi) = [K:\Q] = n_K.
\]
Then the \emph{completed $L$-function of $L(s,\chi)$} is defined to be
\begin{equation}
\xi(s,\chi) := \begin{cases}  (d_K\N\kf_{\chi})^{s/2} L(s,\chi) L_{\infty}(s,\chi) & \text{if } \chi \neq \chi_0, \\\\
 (d_K\N\kf_{\chi})^{s/2} L(s,\chi) L_{\infty}(s,\chi)  \cdot  s(1-s) & \text{if } \chi = \chi_0.
 \end{cases}
 \label{L-completed}
\end{equation}
With an appropriate choice of $a(\chi)$ and $b(\chi)$, it is well-known that $\xi(s,\chi)$ is an entire function satisfying the functional equation
\begin{equation}
\xi(s,\chi) = \varepsilon(\chi)  \cdot \xi(1-s,\overline{\chi})
\label{FunctionalEquation}
\end{equation}
where $\varepsilon(\chi) \in \C$ is the global root number having absolute value 1. See \cite[Section 5]{LO} for details. 

\subsubsection*{Convexity Bound} 

\begin{lem} \label{ConvexityBd} Let $\delta \in (0,\tfrac{1}{2})$ be given. Suppose  $\chi$ is a primitive non-principal Hecke character modulo $\kf_{\chi}$. Then
\[
|L(s,\chi)| \ll  \zeta_{\Q}(1+\delta)^{n_K} \Big( \frac{d_K \N\kf_{\chi}}{ (2\pi)^{n_K}} (1+|s|)^{n_K} \Big)^{(1-\sigma+\delta)/2}
\]
and
\[
|(s-1) \cdot \zeta_K(s)| \ll  \zeta_{\Q}(1+\delta)^{n_K} \Big( \frac{d_K}{ (2\pi)^{n_K}} (1+|s|)^{n_K} \Big)^{(1-\sigma+\delta)/2}
\]
uniformly in the region
\[
-\delta \leq \sigma \leq 1+\delta. 
\]  
\end{lem} 
\begin{proof} This is a version of \cite[Theorem 5]{Rademacher} which has been simplified for our purposes. In his notation, the constants $v_q, a_p, a_{p+r_2}, v_p$ are all zero for characters of $\Cl(\kq)$. Recall that $\zeta_{\Q}(\, \cdot \,)$ is the classical Riemann zeta function. 
\end{proof}
\subsubsection*{Explicit Formula} 
Using the Hadamard product for $\xi(s,\chi)$, one may derive an explicit formula for the logarithmic derivative of $L(s,\chi)$. Before recording this classical result, we introduce an additional piece of notation which will be used throughout the paper: 
\begin{equation}
E_0(\chi) := \begin{cases} 1 & \text{if $\chi$ is principal,} \\ 0 & \text{otherwise.} \end{cases}
\label{def:E_0}
\end{equation}
\begin{lem}
\label{LogDiffCorollary}
Let $\chi$ be a primitive Hecke character modulo $\kf_{\chi}$. Then for all $s \in \C$ away from zeros of $\xi(s,\chi)$, 
\[
-\frac{L'}{L}(s,\chi) =  \frac{E_0(\chi)}{s-1} + \frac{E_0(\chi)}{s} + \frac{1}{2} \log (d_K \N\kq) + \frac{L_{\infty}'}{L_{\infty}}(s, \chi)  - B(\chi) - \sum_{\rho} \Big(\frac{1}{s-\rho} + \frac{1}{\rho} \Big),
\]
where $B(\chi) \in \C$ is a constant depending on $\chi$ and the conditionally convergent sum is over all zeros $\rho$ of $\xi(s,\chi)$. Moreover, 
\[
\Re\{ B(\chi) \}= -\frac{1}{2}  \sum_{\rho} \Big(\frac{1}{1-\rho} + \frac{1}{\rho} \Big) = -\sum_{\rho} \Re \frac{1}{\rho} < 0.
\label{BChi}
\]
\end{lem}
\begin{proof} See \cite[Section 5]{LO} for a proof. Note $\ds ``\sum_{\rho}"$ denotes $``\ds \lim_{T \rightarrow \infty} \sum_{|\Im \rho| \leq T}"$. 
\end{proof}
\noindent
\cref{LogDiffCorollary}  gives the desired formula for $-\frac{L'}{L}(s,\chi)$ with only $\frac{L_{\infty}'}{L_{\infty}}(s,\chi)$ to be estimated.
\begin{lem} \label{LogDiff-Infinite} Let $\chi$ be a primitive Hecke character.  If $\Re\{s\} \geq 1/8$, then
\[
\frac{L_{\infty}'}{L_{\infty}}(s,\chi) \ll n_K \log(2 + |s|). 
\]
\end{lem}
\begin{proof} See \cite[Lemma 5.3]{LO}. 
\end{proof}

\subsection{Elementary Estimates}
\label{Elementary}
\begin{lem} \label{ImprimitiveQ} Let $\kq$ be an integral ideal. Then, for $\epsilon > 0$, 
\[
\sum_{\kp \mid \kq} \frac{\log \N\kp}{\N\kp} \leq \sqrt{ n_K \log \N\kq } \leq \frac{1}{2}\Big( \frac{n_K}{\epsilon} + \epsilon \log \N\kq \Big)
\]
where is sum is over prime ideals $\kp$ dividing $\kq$.
\end{lem}
\begin{proof} The second inequality follows from $(x+y)/2 \geq \sqrt{xy}$ for $x,y \geq 0$. It suffices to prove the first estimate. 
Write $\kq = \prod_{i=1}^r \kp_i^{e_i}$ in its unique ideal factorization where $\kp_i$ are  distinct prime ideals and $e_i \geq 1$. Denote $q_i = \N\kp_i$ and $a_m = \#\{ i : q_i = m\}$. Observe that $a_m = 0$ unless $m$ is a power of a rational prime $p$. Since the principal ideal $(p)$ factors into at most $n_K$ prime ideals in $K$, it follows $a_m \leq n_K$ for $m \geq 1$. Thus, by Cauchy-Schwarz,
\begin{align*}
\sum_{\kp \mid \kq} \frac{\log \N\kp}{\N\kp} 
= \sum_{i=1}^r \frac{\log q_i}{q_i} 
& \leq \Big(\sum_{i=1}^r \frac{\log q_i}{q_i^2} \Big)^{1/2} \Big( \sum_{i=1}^r \log q_i\Big)^{1/2}  \\
& = \Big(\sum_{m \geq 1} a_m \frac{\log m}{m^2} \Big)^{1/2} \Big( \sum_{i=1}^r \log q_i\Big)^{1/2}  \\
& \leq n_K^{1/2} \Big(\sum_{m \geq 1} \frac{\log m}{m^2} \Big)^{1/2} \Big( \sum_{i=1}^r e_i \log q_i\Big)^{1/2}  \\
& =  \Big(\sum_{m \geq 1} \frac{\log m}{m^2} \Big)^{1/2}\sqrt{n_K \log \N\kq} 
\end{align*}
Since $\sum_{m \geq 1} \frac{\log m}{m^2} < 1$, the result follows. 
\end{proof}

\begin{lem} For $\sigma > 1$, 
\begin{align*}
\zeta_K(\sigma) & \leq \zeta(\sigma)^{n_K} \leq \Big(\frac{\sigma}{\sigma-1}\Big)^{n_K}, \\
\log \zeta_K(\sigma) &  \leq n_K \log\Big(\frac{\sigma}{\sigma-1} \Big), \\
-\frac{\zeta_K'}{\zeta_K}(\sigma) & \leq -n_K \frac{\zeta'}{\zeta}(\sigma) \leq  \frac{n_K}{\sigma-1}. 
\end{align*}
where $\zeta(s) = \zeta_{\Q}(s)$ is the classical Riemann zeta function.
\label{LogDedekind}
\end{lem}
\begin{proof} For the first inequality, observe
\[
\zeta_K(\sigma) = \prod_{\kp} (1- (\N\kp)^{-\sigma})^{-1} = \prod_{p} \prod_{(p) \subseteq \kp} (1- (\N\kp)^{-\sigma})^{-1} \leq \prod_p (1-p^{-\sigma})^{-n_K} = \zeta(\sigma)^{n_K}
\]
and note $\zeta(\sigma) \leq \big( \frac{\sigma}{\sigma-1}\big)$ from \cite[Corollary 1.14]{MV}. The second inequality follows easily from the first. The third inequality follows by an argument similar to that of the first and additionally noting $-\frac{\zeta'}{\zeta}(\sigma) < \frac{1}{\sigma-1}$ by \cite[Lemma (a)]{Louboutin_1992} for example.
\end{proof}

\begin{lem} Let $k \geq 1$ and $\chi \pmod{\kq}$ be a Hecke character. Then
\[
\frac{1}{k!} \frac{d^k}{ds^k} \frac{L_{\infty}'}{L_{\infty}}(s,\chi)  \ll n_K
\]
provided $\Re\{s\} > 1$. 
\label{LInfinity-HigherDerivatives}
\end{lem}
\begin{proof} Denote $\psi^{(k)}(\, \cdot \, ) = \frac{d^k}{ds^k} \frac{\Gamma'}{\Gamma}(\, \cdot \,)$. From \eqref{L-Infinite}, we have that
\begin{equation}
\frac{d^k}{ds^k} \frac{L_{\infty}'}{L_{\infty}}(s,\chi) =\frac{a(\chi)}{2^{k+1}}   \cdot \psi^{(k)}\Big(\frac{s}{2} \Big)  +  \frac{b(\chi)}{2^{k+1}} \cdot \psi^{(k)}\Big( \frac{s+1}{2} \Big).
\label{LInfinity-HigherDerivatives-ETS}
\end{equation}
Since $a(\chi) + b(\chi) = n_K$, it suffices to bound $\psi^{(k)}(z)$ for $\Re\{z\} > 1/2$. From the well-known logarithmic derivative of the Gamma function (see \cite[(C.10)]{MV} for example), observe
\[
\Big| \frac{\psi^{(k)}(z)}{k!} \Big| 
= \Big| (-1)^k \sum_{n=1}^{\infty} \frac{1}{(n+z)^{k+1}}\Big| \leq 2^{k+1} \sum_{n \text{ odd} } \frac{1}{n^{k+1}} = (2^{k+1}-1)\zeta_{\Q}(k+1)
\]
for $\Re\{z\} > 1/2$, which yields the desired result when combined with \eqref{LInfinity-HigherDerivatives-ETS} as $\zeta_{\Q}(k+1) \leq \zeta_{\Q}(2) = \pi^2/6$. 
\end{proof}
\section{Zero-Free Gap and Labelling of Zeros}
\label{sec:ZeroFreeGap_and_BadZeros}
The main goal of this section is to show that there is a thin rectangle inside the critical strip above which there is a zero-free gap for 
\[
Z(s) := \prod_{\chi \pmod{\kq}} L(s,\chi). 
\]
Afterwards, we label important zeros of $Z(s)$ which we will refer to throughout the paper.  This zero-free gap is necessary for the proof of \cref{ExplicitNP-Apply} -- a crucial component for later sections.  \\

Let $\vartheta \in [\tfrac{3}{4}, 1]$ be fixed\footnote{For the purposes of this paper, setting $\vartheta = \tfrac{3}{4}$ would be sufficient but we wish to maintain flexibility for possible future investigations of Hecke $L$-functions.} and let $\nu(x), \eta(x)$ be any fixed increasing functions for $x \in [1,\infty)$ such that
\begin{equation}
\begin{aligned}
& \nu(x) \in [4,\infty), \qquad \nu(x) \gg \log(x+4), \\
& \eta(x) \in [2,\infty), \qquad \eta(x) \rightarrow \infty \text{ as } x \rightarrow \infty, \qquad \text{ and } \tfrac{x}{\eta(x) \log(x+1)} \text{ is increasing.} \\
\end{aligned}
\label{NuFunction}
\end{equation}
One could take $\eta(x) = \tfrac{1}{2}\log x + 2$, for example. Denote
\begin{equation}
\begin{aligned}
\cL & := \log d_K + \vartheta \cdot \log \N\kq + n_K \cdot \nu(n_K) , \\
\cL^* & := \log d_K + \vartheta \cdot \log \N\kq,  \\
\cT & := (\cL^*)^{1/ \eta(n_K)\log(n_K+1)} +  \nu( n_K). 
\end{aligned}
\label{cL-definition}
\end{equation}
Similarly, for a Hecke character $\chi \pmod{\kq}$ with conductor $\kf_{\chi}$, define
\begin{equation}
\label{cL-definition-1}
\begin{aligned}
\cL_{\chi} & :=  \log d_K + \log \N\kf_{\chi} +  n_K \cdot \nu(n_K), \\
\cL_{\chi}^* & :=  \log d_K + \log \N\kf_{\chi} \\
\cL_0 & := \cL_{\chi_0} = \log d_K + n_K \cdot \nu(n_K)
\end{aligned}
\end{equation}
For the remainder of the paper, we shall maintain this notation because these quantities will be ubiquitous in all of our estimates. \emph{Henceforth, all implicit constants will be absolute (in particular, independent of $K, \kq$ and all Hecke characters $\chi$ modulo $\kq$) and will only implicitly depend on the fixed $\vartheta, \nu$ and $\eta$.}\\

First, we record some simple relationships between the quantities defined in \eqref{cL-definition} and \eqref{cL-definition-1}.
\begin{lem} \label{QuantityRelations} For the quantities defined in \eqref{cL-definition} and \eqref{cL-definition-1}, all of the following hold: 
\begin{enumerate}[(i)]
	\item $4 \leq \cT \leq \cL$.
	\item $n_K\log \cT = o(\cL)$.
	\item $\cL^* + n_K \log \cT \leq \cL + o(\cL) \quad $ and $ \quad \cL^*_{\chi} + n_K \log \cT \leq \cL_{\chi} + o(\cL)$.
	\item $\cT \rightarrow \infty$ if and only if $\cL \rightarrow \infty$. 
	\item $a \cL_0 + b \cL_{\chi} \leq (a+b) \cL$ provided $0 \leq b \leq 3a$.
\end{enumerate}
\end{lem}
\begin{proof} Statements (i) and (iii) follow easily from (ii) and/or the definitions  of $\cT, \cL$ and $\cL^*$. For (ii), observe that
\[
n_K\log \cT \leq \frac{n_K  \log \cL^*}{\eta(n_K) \log (n_K+1)} + n_K \log \nu(n_K) + n_K \log 2 
\]
The second and third terms are $o(\cL)$ as $\nu(x)$ is increasing. For the first term, note that  $\tfrac{n_K}{\eta(n_K) \log (n_K+1)}$ is increasing as a function $n_K$ by \eqref{NuFunction} and so plugging in the upper bound $n_K = O(\log d_K) = O(\cL^*)$ from Minkowski's theorem we deduce
\[
n_K \log \cT \ll \frac{\cL^*}{\eta(\cL^*)} + o(\cL) = o(\cL)
\]
since $\eta(x) \rightarrow \infty$ as $x \rightarrow \infty$ and $\cL^* \leq \cL$. For (iv), if $n_K$ is bounded, then necessarily $\cL^* \rightarrow \infty$ in which case both $\cT$ and $\cL$ approach infinity. Otherwise, if $n_K \rightarrow \infty$, then both $\cT$ and $\cL$ approach infinity since $\nu(x) \rightarrow \infty$ as $x \rightarrow \infty$. For (v), the claim follows from the fact that $\vartheta \geq \tfrac{3}{4}$. 
\end{proof}

Next, we establish the desired zero-free gap which motivates the choice of $\cL$ and its related quantities.
\begin{lem} \label{ZeroGap} Let $C_0 >0 $ be a sufficiently large absolute constant and let $\cT$ be defined as in \eqref{cL-definition}. For $\cL$ sufficiently large, there exists a positive integer  $T_0 = T_0(\kq) \leq \frac{ \cT }{10}$ such that $\prod_{\chi \pmod{\kq}}L(s,\chi)$ has no zeros in the region
\[
1-\frac{\log\log\cT}{C_0\cL} \leq \sigma \leq 1, \quad T_0 \leq |t| \leq 10T_0.
\]
\end{lem}
\begin{proof} For $0 \leq \alpha \leq 1$ and $T \geq 0$, denote
\[
N_{\kq}(\alpha, T) =  \sum_{\chi \pmod{\kq}} \#\{ \rho \in \C \mid L(\rho, \chi) = 0, \quad \alpha \leq \beta \leq 1, \quad 0 \leq |\gamma| \leq T \}
\]
where we count zeros with multiplicity. 
We shall apply a simplified version of a result of Weiss \cite[Theorem 4.3]{Weiss}; in his notation, we restrict to the case $H = P_{\kq}$ and $Q = n_K^{n_K}d_K\N\kq$ from which it follows $h_H = \#I(\kq)/H \leq e^{O(\cL)}$ by \cite[Lemma 1.16]{Weiss} and also $Q \leq e^{O(\cL)}$ since $\nu(x) \gg \log(x+4)$ by \eqref{NuFunction}. Therefore, by \cite[Theorem 4.3]{Weiss}, for $c_6 \leq \alpha \leq 1- \frac{c_7}{\cL}$, we have
\begin{equation}
N_{\kq}(\alpha,T) \ll  ( e^{\cL} \cdot T^{n_K})^{c_8(1-\alpha)}
\label{WeissLemma}
\end{equation}
 for some absolute constants $0 < c_6 < 1, c_7 > 0, c_8 >0$ and provided $T$ and $\cL$ are sufficiently large. Now suppose, for a contradiction, that no such $T_0$ exists. Setting $\alpha = 1- \frac{\log\log \cT}{C_0\cL}$, it follows that every region
\[
\alpha \leq \sigma \leq 1, \quad 10^j \leq |t| \leq 10^{j+1},
\]
for $0 \leq j < J$ where $J := \big[ \dfrac{\log \cT}{\log 10} \big]$, contains at least one zero of $\prod_{\chi \pmod{\kq}} L(s,\chi)$. Hence,
\[
N_{\kq}(\alpha, \cT) \geq J \gg \log \cT. 
\]
On the other hand, by \eqref{WeissLemma} with $T = \cT$ sufficiently large, our choice of $\alpha$ implies
\[
N_{\kq}(\alpha, \cT)  \ll \big( e^{4\cL/3} \cT^{n_K} \big)^{c_8(1-\alpha)} \ll \exp\Big( \frac{c_8}{C_0} \Big( \tfrac{4}{3} \log\log \cT + \frac{n_K \log \cT \log\log \cT}{\cL}\Big) \Big).
\]
From \cref{QuantityRelations}, $n_K \log \cT = o(\cL)$ so for some absolute constant $c_9 > 0$ 
\begin{align*}
N_{\kq}(\alpha, \cT) & \ll \exp\Big( \frac{c_9}{C_0} \log\log \cT \Big) \ll (\log \cT)^{ \tfrac{c_9}{C_0}}.
\end{align*}
Upon taking $C_0 = 2c_9$, we obtain a contradiction for $\cT$ sufficiently large. From \cref{QuantityRelations}, we may equivalently ask that $\cL$ is sufficiently large. 
\end{proof}

Using the zero-free gap from \cref{ZeroGap}, we label important ``bad" zeros of $Z(s) = \prod_{\chi \pmod{\kq}} L(s,\chi)$.  
These zeros will be referred to throughout the paper. A typical zero of $L(s,\chi)$ will be denoted $\rho = \beta + i\gamma$ or $\rho_{\chi} = \beta_{\chi} + i\gamma_{\chi}$ when necessary. \\

\noindent
	 \emph{Worst Zero of each Character} \\
	Consider the rectangle
	\[
	\cR = \cR(\kq) := \{ s \in \C : 1-\frac{\log \log \cT}{C_0\cL} \leq \sigma \leq 1, \quad |t| \leq T_0 \}
	\]
	for $T_0 = T_0(\kq) \in [1, \tfrac{\cT}{10}]$ and $C_0 > 0$ defined by \cref{ZeroGap}. Denote $\mathcal{Z}$ to be the multiset of zeros of $Z(s)$ contained in $\cR$. Choose finitely many zeros $\rho_1, \rho_2,\dots$ from $\mathcal{Z}$ as follows:
	\begin{enumerate}
		\item Pick $\rho_1$ such that $\beta_1$ is maximal, and let $\chi_1$ be the corresponding character. Remove all zeros of $L(s,\chi_1)$ and $L(s,\overline{\chi_1})$ from $\mathcal{Z}$. 
		\item Pick $\rho_2$ such that $\beta_2$ is maximal, and let $\chi_2$ be the corresponding character. Remove all zeros of $L(s,\chi_2)$ and $L(s,\overline{\chi_2})$ from $\mathcal{Z}$.
		\item[\vdots]
	\end{enumerate}
	Continue in this fashion until $\cR$ has no more zeros to choose. Then if $\chi \neq \chi_i, \overline{\chi_i}$ for $1 \leq i < k$, then by \cref{ZeroGap} every zero $\rho$ of $L(s,\chi)$ satisfies:
	\begin{equation}
	\Re(\rho) \leq \Re(\rho_k) \quad \text{ or } \quad |\Im(\rho)| \geq 10T_0.
	\label{ZeroGapProperty}
	\end{equation}
	For convenience of notation, denote
	\[
	\rho_k = \beta_k + i\gamma_k, \quad \beta_k = 1 - \frac{\lambda_k}{\cL}, \quad \gamma_k = \frac{\mu_k}{\cL}. 
	\]
	$ $ \\
	\emph{Second Worst Zero of the Worst Character}\\
	Suppose $L(s,\chi_1)$ has a zero $\rho' \neq \rho_1, \overline{\rho_1}$ in the rectangle $\cR$, or possibly a repeated real zero $\rho' = \rho_1$. Choose $\rho'$ with $\Re(\rho')$ maximal and write
	\[
	\rho' = \beta' + i\gamma', \quad \beta' = 1 - \frac{\lambda'}{\cL}, \quad \gamma' = \frac{\mu'}{\cL}. 
	\]


\section{Classical Explicit Inequality}
\label{ClassicalExpIneq}
In this section, we prove an inequality for $-\Re\{\frac{L'}{L}(s,\chi)\}$ based on a bound for $L(s,\chi)$ in the critical strip and a type of Jensen's formula employed by Heath-Brown in \cite[Section 3]{HBLinnik}. First, we deal with non-primitive characters. 
\begin{lem}
\label{ImprimitiveChar}
Suppose $\chi \pmod{\kq}$ is induced by the character $\chi^* \pmod{\kf}$. Then for $\epsilon > 0$, we have
\[
 \frac{L'}{L}(s,\chi) = \frac{L'}{L}(s,\chi^*) + O\Big(\frac{n_K}{\epsilon} + \epsilon \cL^*\Big)
\]
uniformly in the range $\sigma > 1$. 
\end{lem}
\begin{proof} Observe 
\[
\begin{aligned}
& \Big| \frac{L'}{L}(s,\chi) - \frac{L'}{L}(s,\chi^*)\Big|  \leq  \sum_{ \kp \mid \kq} \sum_{j \geq 1} \frac{\log \N\kp}{(\N\kp)^{j}}  \leq 2\sum_{\kp \mid \kq} \frac{\log \N\kp}{\N\kp}. 
 \end{aligned}
\]
The desired result then follows from \cref{ImprimitiveQ}. 
\end{proof}
Next, we give a bound for $L(s,\chi)$ in a relevant region of the critical strip.  
\begin{lem} \label{CriticalStripConstant} Let $\chi \pmod{\kq}$ be non-principal and induced by the primitive character $\chi^* \pmod{\kf_{\chi}}$. There exists an absolute constant $\phi > 0$ such that for $\epsilon > 0$,
we have
\[
|L(s,\chi^*)| \leq \exp\Big\{ 2\phi  \cL_{\chi} (1-\sigma+\epsilon)  + o_{\epsilon}(\cL) \Big\}
\]
and
\[
|(s-1) \cdot \zeta_K(s)| \leq  \exp\Big\{ 2\phi  \cL_0 (1-\sigma+\epsilon)  + o_{\epsilon}(\cL) \Big\}
\]
uniformly in the region
\[
\frac{1}{2} \leq \sigma \leq 1 + \frac{\log \cL}{\cL}, \qquad |t| \leq \cT.
\]
In particular, we may take $\phi = \tfrac{1}{4}$. 
\end{lem}
\begin{rem*}
The term $o_{\epsilon}(\cL)/\cL$ goes to zero as $\cL \rightarrow \infty$ but the rate of convergence depends on $\epsilon > 0$.   
\end{rem*}

\begin{proof}  Take $\phi = \tfrac{1}{4}$. From \cref{ConvexityBd}, we have
\[
L(s,\chi^*) \ll  \zeta(1+\epsilon)^{n_K} \Big( \frac{d_K \N\kf_{\chi}}{ (2\pi)^{n_K}} (1+|s|)^{n_K} \Big)	^{2\phi(1-\sigma+\epsilon)}
\] 
uniformly in the region $-\epsilon \leq \sigma \leq 1+\epsilon$. Noting $1+|s| \leq 3+ \cT$ and using \cref{LogDedekind}, the above is therefore
\[
\ll  \exp\Big\{ (\cL_{\chi}^* + n_K\log( \cT+3) ) \cdot 2\phi (1-\sigma+\epsilon) + O_{\epsilon}(n_K) \Big\}
\]
From \cref{QuantityRelations},  it follows $\cL_{\chi}^* + n_K \log( \cT+3) \leq \cL_{\chi} + o(\cL)$. Substituting this into the above, the desired result then follows upon noting $n_K = o(\cL)$. The proof for $\zeta_K(s)$ is similar.
\end{proof}

Any improvement on the constant $\phi$ will have a crucial effect on the final result. For example, the Lindel\"{o}f hypothesis for Hecke $L$-functions gives $\phi = \epsilon$. \emph{For the remainder of this paper, we set }
\[
\phi := \frac{1}{4}. 
\]
We may now establish the main result of this section.
\begin{lem} \label{ClassicalEI} Let $\chi \pmod{\kq}$ be arbitrary. For any $\epsilon > 0$, there exists a $\delta = \delta(\epsilon) > 0$ such that
\[
-\Re \frac{L'}{L}(s,\chi) \leq ( \phi + \epsilon )\cL_{\chi} +  \Re\Big\{ \frac{E_0(\chi)}{s-1}  \Big\} - \sum_{|1+it-\rho| \leq \delta} \Re\Big\{ \frac{1}{s-\rho} \Big\}  + o_{\epsilon}(\cL)
\]
uniformly for
\[
1 + \frac{1}{\cL \log \cL} \leq \sigma \leq 1 + \frac{\log\cL}{\cL}, \qquad |t| \leq \cT.
\]
\end{lem}
\begin{rem*} When $K=\Q$,  Heath-Brown showed the same inequality in \cite[Lemma 3.1]{HBLinnik} with $\phi = \tfrac{1}{6}$ instead of $\phi = \tfrac{1}{4}$ by leveraging Burgess' estimate for character sums. Note that Heath-Brown's notation for $\phi$ differs by a factor of 2 with our notation. 
\end{rem*}
\begin{proof} We closely follow the arguments leading to the proof of \cite[Lemma 3.1]{HBLinnik}. Assume without loss that $\epsilon \in (0,1/2)$. Suppose $\chi \pmod{\kq}$ is induced from a non-principal primitive character $\chi^* \pmod{\kf}$. Apply \cite[Lemma 3.2]{HBLinnik} with $f( \, \cdot \,) = L(\, \cdot \, ,\chi^*)$ with $a=s$ and $R = \tfrac{1}{2}$. Then
\begin{equation}
-\Re \frac{L'}{L}(s,\chi^*) = -\sum_{|s-\rho| < \frac{1}{2}} \Re\Big\{ \frac{1}{s-\rho} - 4(s-\rho) \Big\} - J
\label{JensenUse}
\end{equation}
where
\[
J := \frac{2}{\pi} \int_0^{2\pi} (\cos \theta) \cdot \log| L(s+\tfrac{1}{2} e^{i\theta},\chi^*)|d\theta.
\]
We require a lower bound for $J$ so we divide the contribution of the integral into three separate intervals depending on the sign of $\cos \theta$. 
\begin{itemize}
	\item For $\theta \in [0,\pi/2]$, by \cref{LogDedekind} it follows that
\begin{align*}
\log|L(s+\tfrac{1}{2}e^{i\theta},\chi^*)| 
&  \leq \log \zeta_K(\sigma+\tfrac{1}{2}\cos\theta)  \leq n_K \log \Big( \frac{2}{\sigma-1+\tfrac{1}{2}\cos\theta}\Big) 
\end{align*}
On the interval $I_1 := [0, \tfrac{\pi}{2} - (\sigma-1) ]$, as $\sigma-1 \geq 0$, the contribution of the integral $J$ is
\[
\ll  n_K  \int_0^{\pi/2} (\cos \theta) \log(4/\cos \theta) d\theta \ll n_K. 
\]
On the interval $I_2 := [\tfrac{\pi}{2} - (\sigma-1), \tfrac{\pi}{2}]$, as $\cos \theta \geq 0$, the contribution of the integral $J$ is
\[
\ll  n_K \log(2/(\sigma-1))  \int_{I_2} (\cos \theta) d\theta \ll n_K (\sigma-1) \log(2/(\sigma-1)) \ll \frac{n_K (\log \cL)^2}{\cL}  \ll (\log \cL)^2
\]
because $(\cL\log\cL)^{-1} \leq \sigma -1 \leq \cL^{-1} \log \cL $ and $n_K \ll \cL$. 

\item For $\theta \in [\pi/2, 3\pi/2]$, notice
\[
\frac{1}{2} \leq \sigma-\frac{1}{2} \leq \sigma + \frac{1}{2}\cos\theta \leq \sigma \leq 1+\frac{\log \cL}{\cL}.
\]
Hence, we may use \cref{CriticalStripConstant} to see that
\begin{align*}
 \log|L(s+\tfrac{1}{2}e^{i\theta},\chi^*)| 
 & \leq 2\phi  \cL_{\chi} (1-\sigma-\tfrac{1}{2}\cos\theta + \epsilon) + o_{\epsilon}(\cL)    \\
 & \leq 2\phi  \cL_{\chi}(-\tfrac{1}{2}\cos\theta + \epsilon) + o_{\epsilon}(\cL). 
\end{align*}
This implies that
\begin{align*}
  \int_{\pi/2}^{3\pi/2} (\cos \theta) \cdot \log|L(s+\tfrac{1}{2}e^{i\theta},\chi)|d\theta 
&  \geq 2\phi  \cL_{\chi}   \int_{\pi/2}^{3\pi/2} \big(- \tfrac{1}{2}\cos^2 \theta + \epsilon \cos \theta \big) d\theta  + o_{\epsilon}(\cL) \\
&  = \phi \cL_{\chi} \big( -\tfrac{\pi}{2} -  4\epsilon \big) + o_{\epsilon}(\cL).
\end{align*}

\item For $\theta \in [3\pi/2, 2\pi]$, we obtain the same contribution as $\theta \in [0,\pi/2]$ by a similar argument. 
\end{itemize} 
Combining all contributions, we have
\begin{equation}
J \geq  \Big(- \phi - \frac{2\epsilon}{\pi} \Big)\cL_{\chi}  + o_{\epsilon}(\cL)
\label{Jintegral}
\end{equation}
since $n_K = o(\cL)$. For the sum over zeros in \eqref{JensenUse}, notice that we may arbitrarily discard zeros from the sum since for $|s-\rho| < \tfrac{1}{2}$, 
\[
\Re\Big\{ \frac{1}{s-\rho} - 4(s-\rho) \Big\} = (\sigma-\beta) \Big( \frac{1}{|s-\rho|^2} - \frac{1}{(\tfrac{1}{2})^2} \Big) \geq 0. 
\]
Thus, for any $0 < \delta < \tfrac{1}{2}- \tfrac{\log \cL}{\cL}$, we may restrict our sum over zeros from $|s-\rho| < \tfrac{1}{2}$ to a smaller circle within it: $|1+it-\rho| \leq \delta$. From our previous observation, we may discard zeros outside of this smaller circle. As $\cL \geq 4$ by \cref{QuantityRelations}, we may instead impose that $0 < \delta < \tfrac{1}{8}$.

Now, from \cite[Lemma 2.1]{LMO} and \cref{QuantityRelations}, we have that 
\[
\#\{ \rho : |1+it-\rho| \leq \delta\} \ll  \cL_{\chi}^* + n_K \log(\cT +3)  \ll \cL_{\chi} + o(\cL).
\]
Further, for such zeros $\rho$ satisfying $|1+it-\rho| \leq \delta$, notice
\[
\Re\{ s-\rho \} = \sigma-\beta \leq   \frac{\log \cL}{\cL} + \delta 
\]
implying, for some absolute constant $c_0 \geq 1$, 
\begin{equation}
\sum_{|1+it-\rho| \leq \delta} \Re\{ 4(s-\rho) \Big\} \leq 4c_0(\cL_{\chi} + o(\cL) ) \Big(\frac{\log \cL}{\cL} + \delta \Big) = 4 c_0 \delta \cL_{\chi} + O(\log \cL). 
\label{SmallerCircle}
\end{equation}
Combining these observations, from \eqref{JensenUse}, \eqref{Jintegral}, and \eqref{SmallerCircle} and \cref{ImprimitiveChar} we see that
\begin{align*}
-\Re \frac{L'}{L}(s,\chi) 
&  \leq -\sum_{|1+it-\rho| \leq \delta} \Re\Big\{ \frac{1}{s-\rho} - 4(s-\rho) \Big\}  + \big(  \phi + \frac{2 \epsilon}{\pi} \big)\cL_{\chi}  + o_{\epsilon}(\cL)\\
&   \leq -\sum_{|1+it-\rho| \leq \delta} \Re\Big\{ \frac{1}{s-\rho} \Big\}  +  \big(  \phi + \frac{2\epsilon}{\pi} + 4c_0 \delta \big)\cL_{\chi}  + o_{\epsilon}(\cL). 
\end{align*}
Taking $\delta = \epsilon/4c_0$, note $\delta \in (0,1/8)$ as $\epsilon < 1/2$ by assumption and $c_0 \geq 1$. Rescaling $\epsilon$ appropriately completes the proof for $\chi$ non-principal.

If $\chi = \chi_0$ is principal, then we  proceed in the same manner, applying Lemma 3.2 of \cite{HBLinnik} with $f(\,\cdot\,) = (s-1)\zeta_K(\, \cdot \,)$. This choice gives rise to the additional term $\frac{1}{s-1}$, but otherwise we continue with the same argument. The only other difference occurs in the analysis of the integral $J$ for $\theta \in [0,\pi/2] \cup [3\pi/2, 2]$ where we must instead estimate
\[
\log\big( |s-1 +\tfrac{1}{2}e^{i\theta}| \cdot |\zeta_K(s+\tfrac{1}{2}e^{i\theta})| \big). 
\]
As  $|s-1+\tfrac{1}{2}e^{i\theta}| \geq \tfrac{1}{2}$, the contribution to the integral $J$ will be bounded and hence ignored. 
\end{proof}
\section{Polynomial Explicit Inequality}
\label{PolynomialEXP}
By including higher derivatives of $-\tfrac{L'}{L}(s,\chi)$, the goal of this section is establish a generalization of the ``classical explicit inequality"  based on techniques in \cite[Section 4]{HBLinnik}.   Let $\chi \pmod{\kq}$ be a Hecke character. For a polynomial $P(X) = \sum_{k=1}^d a_k X^k \in \R[X]$ of degree $d \geq 1$,  define a real-valued function
\begin{equation}
\cP(s,\chi) = \cP(s,\chi; P)  := \sum_{\kn \subseteq \cO} \frac{\Lambda_K(\kn)}{(\N\kn)^\sigma} \Big( \sum_{k=1}^d a_k \frac{\big( (\sigma-1) \log \N\kn\big)^{k-1}}{(k-1)!} \Big) \cdot \Re\Big\{ \frac{\chi(\kn)}{(\N\kn)^{it} } \Big\}
\label{PolyAction-Trig}
\end{equation}
for $\sigma > 1$ and where $\Lambda_K(\,\cdot\,)$ is the von Mangoldt $\Lambda$-function on integral ideals of $\cO_K$ defined by
\[
\Lambda_K(\kn) = \begin{cases} \log \N\kp & \text{if $\kn$ is a power of a prime ideal $\kp$,} \\ 0 & \text{otherwise}. \end{cases}
\]
From the classical formula
\[
-\frac{L'}{L}(s,\chi) = \sum_{\kn \subseteq \cO} \Lambda_K(\kn) \chi(\kn) (\N\kn)^{-s} \qquad \text{ for } \sigma > 1,
\]
it is straightforward to deduce that
\begin{equation}
\begin{aligned}
\cP(s,\chi) 
& =\sum_{k=1}^d a_k (\sigma-1)^{k-1} \cdot \Re\Big\{  \frac{ (-1)^k}{(k-1)!} \frac{d^{k-1}}{ds^{k-1}} \frac{L'}{L}(s,\chi)  \Big\}  & \text{ for $\sigma > 1$}.\\
\end{aligned}
\label{PolyAction}
\end{equation}
To prove an explicit inequality using $\cP(s,\chi)$, we first reduce the problem to primitive characters.
\begin{lem} \label{PolyEI-Imprimitive} Let $\chi \pmod{\kq}$ be induced from the primitive character $\chi^* \pmod{\kf_{\chi}}$. Let $P(X) \in \R[X]$ be a polynomial with $P(0) = 0$. Then, for $\epsilon > 0$, 
\[
\cP(s,\chi) = \cP(s, \chi^*) + O_P( \epsilon^{-1} n_K + \epsilon \cL)
\]
uniformly in the region 
\[
1 < \sigma \leq 1 + \frac{100}{\cL}. 
\]
\end{lem}
\begin{proof} Denote $d = \deg P$. Observe
\[
 \Big| \cP(s,\chi) -  \cP(s,\chi^*)\Big| 
 	\ll_P \sum_{ (\kn,\kq) \neq 1}\frac{\Lambda_K(\kn)}{\N\kn} \Big( \frac{\log \N\kn}{\cL} \Big)^{d-1}  
	\ll_{P} \sum_{ \kp \mid \kq} \sum_{j \geq 1} \frac{\log \N\kp}{(\N\kp)^{j}} \cdot \Big( \frac{j \log \N\kp}{\cL} \Big)^{d-1}.  
\]
For $\kp \mid \kq$, note $\log \N\kp \ll \log \N\kq \ll \cL$ and so the above is
\[
\ll_{P} \sum_{ \kp \mid \kq} \sum_{j \geq 1}  \frac{j^{d-1}\log \N\kp}{(\N\kp)^{j}}  \ll_{P} \sum_{\kp \mid \kq} \frac{\log \N\kp}{\N\kp}. 
\]
The desired result then follows from \cref{ImprimitiveQ}. 
\end{proof}
\begin{prop} \label{PolyEI} Let $\chi \pmod{\kq}$ and $\epsilon > 0$ be arbitrary. Suppose the polynomial $P(X) = \sum_{k=1}^d a_k X^k$
of degree $d \geq 1$ has non-negative real coefficients. Then there exists $\delta = \delta(\epsilon,P) > 0$, such that 
\begin{equation}
\tfrac{1}{\cL}  \cdot \cP(s, \chi)  \leq  \Re\bigg\{ \frac{ P\big( \frac{\sigma-1}{s-1}\big)}{\sigma-1} E_0(\chi) -  \sum_{|1+it-\rho_{\chi}| \leq \delta} \frac{P\big( \frac{\sigma-1}{s-\rho_{\chi}} \big)}{\sigma-1} \bigg\} \cdot \frac{1}{\cL} + a_1 \phi \tfrac{\cL_{\chi}}{\cL}  + \epsilon 
\label{PolyEI-Inequality}
\end{equation}
uniformly in the region
\[
1 + \frac{1}{\cL \log \cL} \leq \sigma \leq 1+\frac{100}{\cL}, \qquad |t| \leq \cT
\]
provided $\cL$ is sufficiently large depending on $\epsilon$ and $P$. 
\end{prop}
\begin{proof} Let $\chi^* \pmod{\kf_{\chi}}$ be the primitive character inducing $\chi \pmod{\kq}$. From \cref{PolyEI-Imprimitive} and the observation $n_K = o(\cL)$, it follows
\[
\tfrac{1}{\cL} \cP(s,\chi) = \tfrac{1}{\cL} \cP(s,\chi^*) + \epsilon
\]
for $\cL$ sufficiently large depending on $\epsilon$ and $P$. Thus, it suffices to show \eqref{PolyEI-Inequality} with $\cP(s,\chi^*)$ instead of $\cP(s,\chi)$. 
Define
\[
P_2(X) := \sum_{k=2}^d a_k X^k = P(X) - a_1X.
\]
Using \cref{LogDiffCorollary} and \cref{LInfinity-HigherDerivatives}, we see for $k \geq 2$ and $\sigma > 1$ that
\begin{align*}
\frac{(-1)^k}{(k-1)!} \frac{d^{k-1}}{ds^{k-1}} \frac{L'}{L}(s,\chi^*)  
& =   \frac{E_0(\chi)}{(s-1)^k} -\sum_{\rho_{\chi}} \frac{1}{(s-\rho_{\chi})^k } + \frac{E_0(\chi)}{s^k} - \frac{(-1)^k}{(k-1)!} \frac{d^{k-1}}{ds^{k-1}} \frac{L_{\infty}'}{L_{\infty}}(s, \chi^*)  \\
& =  \frac{E_0(\chi)}{(s-1)^k} -\sum_{\rho_{\chi}} \frac{1}{(s-\rho_{\chi})^k }  + O(n_K).
\end{align*}
Substituting these formulae into $\cP(s,\chi^*; P_2)$ defined via \eqref{PolyAction}, it follows for $\sigma > 1$ that
\begin{equation}
\label{PolyDeg2}
\begin{aligned}
\cP(s,\chi^*; P_2)
& =  \frac{1}{\sigma-1} \sum_{k=2}^d a_k  \Re\bigg\{ \Big( \frac{\sigma-1}{s-1} \Big)^k E_0(\chi) -  \sum_{\rho_{\chi}} \Big( \frac{\sigma-1}{s-\rho_{\chi} } \Big)^k \bigg\}    + O_P(n_K).
\end{aligned}
\end{equation}
Obtain $\delta = \delta(\epsilon)$ from \cref{ClassicalEI}. Since $ \sigma < 1+\tfrac{100}{\cL}$, we see by the zero density estimate \cite[Lemma 2.1]{LMO} and \cref{QuantityRelations} that
\begin{align*}
\sum_{|1+it-\rho_{\chi}| \geq \delta} \Big| \frac{\sigma-1}{s-\rho_{\chi} } \Big|^k 
& \ll \Big( \frac{100}{\cL} \Big)^k \sum_{|1+it-\rho_{\chi}| \geq \delta} \frac{1}{|s-\rho_{\chi}|^k } 
 \ll_{\delta} \Big( \frac{100}{\cL} \Big)^k  \sum_{\rho_{\chi}}  \frac{1}{1+ |t-\gamma_{\chi}|^2}  \\
&\qquad  \ll_{\delta} \Big( \frac{100}{\cL} \Big)^k  \cdot(\cL^*+ n_K \log \cT )\\
& \qquad  \ll_{\delta} \frac{ (100)^{k} }{\cL^{k-1}} \ll_{\delta, P} \frac{1}{\cL}.
\end{align*}
Hence, 
\[
\frac{1}{\sigma-1} \sum_{k=2}^d a_k \Re\bigg\{ \sum_{|1+it-\rho_{\chi}| \geq \delta} \Big( \frac{\sigma-1}{s-\rho_{\chi} } \Big)^k \bigg\} \ll_{\epsilon, P} \log \cL 
\]
since $\sigma > 1 + \frac{1}{\cL \log \cL}$ and $\delta$ depends only on $\epsilon$. Removing this contribution in \eqref{PolyDeg2} implies that
\begin{align*}
\cP(s,\chi^*; P_2)
& =  \frac{1}{\sigma-1} \sum_{k=2}^d a_k  \Re\bigg\{ \Big( \frac{\sigma-1}{s-1} \Big)^k E_0(\chi) -  \sum_{|1+it-\rho_{\chi}| \leq \delta} \Big( \frac{\sigma-1}{s-\rho_{\chi} } \Big)^k \bigg\}   + O_{\epsilon, P}( n_K + \log \cL ) \\
& =  \Re\bigg\{ \frac{ P_2( \frac{\sigma-1}{s-1} ) \} }{\sigma-1} E_0(\chi) -   \sum_{|1+it-\rho_{\chi}| \leq \delta} \frac{P_2( \frac{\sigma-1}{s-\rho_{\chi} }) }{\sigma-1}  \bigg\}   + O_{\epsilon, P}( n_K + \log \cL ). 
\end{align*}
For the linear polynomial $P_1(X) := a_1X$, we apply \cref{ClassicalEI} directly to find that
\[
\cP(s,\chi^*; P_1) \leq a_1\big( \phi + \epsilon \big) \cL_{\chi}  +  \Re\bigg\{ \frac{ P_1\big( \frac{\sigma-1}{s-1}\big)}{\sigma-1} E_0(\chi) -  \sum_{|1+it-\rho_{\chi}| \leq \delta} \frac{P_1\big( \frac{\sigma-1}{s-\rho_{\chi}} \big)}{\sigma-1} \bigg\}  + o_{\epsilon, P}(\cL) 
\]
for $\cL$ sufficiently large depending on $\epsilon$. 

Finally, from \eqref{PolyAction}, we see that $\cP(s, \chi^*; P) = \cP(s, \chi^* ; P_1) + \cP(s,\chi^*; P_2)$ since $P = P_1 + P_2$, so combining the above inequality with the previous equation we conclude
\begin{align*}
\cP(s,\chi^*) 
& \leq a_1\big( \phi + \epsilon \big) \cL_{\chi}  +  \Re\bigg\{ \frac{ P\big( \frac{\sigma-1}{s-1}\big)}{\sigma-1} E_0(\chi) -  \sum_{|1+it-\rho_{\chi}| \leq \delta} \frac{P\big( \frac{\sigma-1}{s-\rho_{\chi}} \big)}{\sigma-1} \bigg\}  \\
& \qquad + O_{\epsilon,P}(n_K+ \log \cL ) + o_{\epsilon, P}(\cL). 
\end{align*}
Dividing both sides by $\cL$ and taking $\cL$ sufficiently large depending on $\epsilon$ and $P$, the errors may be made arbitrarily small. Choosing a new $\epsilon$ yields the desired result.  
\end{proof}

We wish to use \cref{PolyEI} in many contexts but typically we want to restrict the sum over zeros $\rho$ to just a few specified zeros. To do so, we impose an additional condition on  $P(X)$. 

\begin{defn} \label{Admissible} A polynomial $P(X) \in \R_{\geq 0}[X]$ is \emph{admissible} if $P(0) = 0$ and
\[
\Re\Big\{ P\Big(\frac{1}{z}\Big) \Big\} \geq 0 \qquad \text{when $\Re\{z\} \geq 1$}. 
\]
\end{defn}

Now we establish a general lemma which we will repeatedly apply in varying circumstances. 

\begin{lem} \label{PolyEI-Apply} 
Let $\epsilon > 0$ and $0 < \lambda < 100$ be arbitrary, and let $s = \sigma + it$ with
\[
\sigma = 1 + \frac{\lambda}{\cL}, \qquad  |t| \leq \cT. 
\]
Let $\chi \pmod{\kq}$ be an arbitrary Hecke character and let $\mathcal{Z} := \{ \tilde{\rho_1}, \tilde{\rho_2}, \dots, \tilde{\rho_J}\}$
 be a finite multiset of zeros of $L(s,\chi)$ (called the \underline{extracted zeros}) where
 \[
 \tilde{\rho_j} = \tilde{\beta_j} + i\tilde{\gamma_j} =  \Big(1-\frac{\tilde{\lambda_j}}{\cL}\Big) + i\cdot \frac{\tilde{\mu_j}}{\cL}, \qquad 1 \leq j \leq J.
 \]
Suppose $P(X) = \sum_{k=1}^d a_k X^k$ is an admissible polynomial. Then
 \begin{equation*}
\begin{aligned}
\frac{\lambda}{\cL} \cdot \cP(s,\chi) & \leq    \Re\bigg\{ E_0(\chi) P\Big( \frac{\lambda}{\lambda + i\mu}\Big) - \sum_{j=1}^J  P\Big( \frac{\lambda}{\lambda + \tilde{\lambda_j} + i(\mu - \tilde{\mu_j}) } \Big) \bigg\} + a_1 \lambda \phi \tfrac{\cL_{\chi}}{\cL} + \epsilon
\end{aligned}
\end{equation*}
for $\cL$ sufficiently large depending only on $\epsilon$, $\lambda$,  the polynomial $P$, and the number of extracted zeros $J$. 

\end{lem}
\begin{proof} From \cref{PolyEI} and the admissibility of $p$, it follows that
\begin{equation}
\begin{aligned}
\tfrac{\lambda}{\cL} \cP(s, \chi)
& \leq a_1 \lambda \phi  \tfrac{\cL_{\chi}}{\cL} + \epsilon + \Re\bigg\{ E_0(\chi) P\Big( \frac{\lambda}{\lambda + i\mu}\Big)  -  \sum_{\substack{ |1+it-\rho_{\chi}| \leq \delta \\ \rho_{\chi} \in \mathcal{Z}}} P\Big( \frac{\lambda}{\lambda + \lambda_{\chi} + i(\mu - \mu_{\chi} ) } \Big)   \bigg\}  
\end{aligned}
\label{PolyEI-Application-0}
\end{equation}
for some $\delta = \delta(\epsilon,p)$ and $\cL$ sufficiently large depending on $\epsilon, P$ and $\lambda$.  Note the admissibility of $P$ was used to restrict the sum over zeros further by throwing out $\rho_{\chi} \not\in \cZ$ satisfying $|1+it-\rho_{\chi}| \leq \delta$.  For the remaining sum, consider $\tilde{\rho}_j \in \mathcal{Z}$. If $|1+it-\tilde{\rho}_j| \geq \delta$, then $|\tilde{\mu}_j-\mu| \gg_{\delta} \cL$  or $\tilde{\lambda}_j \gg_{\delta} \cL$. As $P(0) = 0$, it follows 
\[
 \Re\bigg\{ P\Big( \frac{\lambda}{\lambda + \tilde{\lambda_j} + i(\mu - \tilde{\mu_j} ) } \Big) \bigg\} \ll_{\epsilon, p, \lambda} \cL^{-1}.
\]
Hence, in the sum over zeros in \eqref{PolyEI-Application-0}, we may include each extracted zero $\tilde{\rho}_j$ with error $O_{\epsilon,P,\lambda}(\cL^{-1})$ implying
\begin{equation*}
\begin{aligned}
\sum_{ \substack{ |1+it-\rho_{\chi}| \leq \delta \\ \rho_{\chi} \in \mathcal{Z}} } \Re\bigg\{ P\Big( \frac{\lambda}{\lambda + \lambda_{\chi} + i(\mu - \mu_{\chi} ) } \Big) \bigg\} 
& = \sum_{j=1}^J  \Re\bigg\{ P\Big( \frac{\lambda}{\lambda + \tilde{\lambda_j} + i(\mu - \tilde{\mu_j}) } \Big) \bigg\} + O_{\epsilon, P, \lambda}(\cL^{-1} J).
\end{aligned}
 \label{PolyEI-Application-1}
\end{equation*}
Using this estimate in \eqref{PolyEI-Application-0} and taking $\cL$ sufficiently large depending on $\epsilon, \lambda, p$ and $J$, we have the desired result upon choosing a new $\epsilon$. 
\end{proof}

During computations, we will employ \cref{PolyEI-Apply} with $P(X) = P_4(X)$ as given in \cite{HBLinnik}. That is, \textit{for the remainder of this paper, denote}
\begin{equation}
P_4(X) := X + X^2 + \tfrac{4}{5}X^3 + \tfrac{2}{5}X^4.
\label{def:P_4}
\end{equation}
We establish a key property of $P_4(X)$ in \cref{PmLemma} using the following observation. 

\begin{lem} \label{GmLemma}
Let $V,W \geq 0$ be arbitrary and $m \geq 1$ be a positive integer. Define
\[
G_m(x,y,z) := V \cdot \frac{x^m}{(x^2+z^2)^m} + W \cdot \frac{y^m}{(y^2+z^2)^m} - \frac{1}{(1+z^2)^m}.
\]
for $x,y,z \in \R$. If $x,y \geq 1$ then
\[
G_m(x,y,z) \geq 0 \quad \text{ provided } \quad \frac{V}{x^m} + \frac{W}{y^m} \geq 1.
\]
\end{lem}
\begin{proof} Notice
\[
G_m(x,y,z) = \frac{V/x^m}{(1+(z/x)^2)^m} + \frac{W/y^m}{(1+(z/y)^2)^m} - \frac{1}{(1+z^2)^m} \geq \Big( \frac{V}{x^m} + \frac{W}{y^m} - 1\Big) \frac{1}{(1+z^2)^m}. 
\]
\end{proof}
\begin{lem}
\label{PmLemma}
The polynomial $P_4(X)$ is admissible. Additionally, if $0 < a \leq b \leq c$, $A > 0$, and $B,C \geq 0$, then
\begin{equation}
\Re\{ C \cdot P_4\big( \frac{a}{c+it} \big) + B \cdot P_4\big( \frac{a}{b+it} \big) -  A \cdot P_4\big( \frac{a}{a+it} \big)  \} \geq 0
\label{RealPart-Pm-Positivity}
\end{equation}
provided 
\[
 \frac{C}{c^4} + \frac{B}{b^4} \geq  \frac{A}{a^4}. 
\]
\end{lem}
\begin{proof} The proof that  $P_4(X)$ is admissible is given in \cite[Section 4]{HBLinnik}. It remains to prove \eqref{RealPart-Pm-Positivity}. By direct computation, one can verify that
\begin{equation}
\Re\{P_4\big( {a\over b+it}\big)\}
={16\over 5}{(ab)^4\over (b^2+t^2)^4}
+{a(b-a)\over 5(b^2+t^2)^3}Q(a,b,t),
\label{RealPart-P4-Identity}
\end{equation}
where $Q(a,b,t)
=5t^4+2(5b^2+5ab-a^2)t^2+b^2(5b^2+10ab+14a^2)$ is clearly positive for $0<a\leq b$ and $t \in \R$. Thus, for $0 <  a \leq b$ and $t \in \R$, we have
\begin{equation}
\Re\{ P_4\big( \frac{a}{b+it} \big) \} \geq \frac{16}{5} \frac{ (a b)^4}{(b^2+t^2)^4}.
\label{RealPart-P4}
\end{equation}
Now, consider the LHS of \eqref{RealPart-Pm-Positivity}. Apply \eqref{RealPart-P4} to the first and second term and \eqref{RealPart-P4-Identity}  to the third term deducing that the LHS of \eqref{RealPart-Pm-Positivity} is
\begin{equation}
 \geq  \frac{16a^4}{5} \cdot \Big( C \cdot \frac{ c^4}{(c^2+t^2)^4} + B \cdot \frac{ b^4}{(b^2+t^2)^4} -  A \cdot \frac{ a^4}{(a^2+t^2)^4} \Big) \geq \frac{16A}{5} \cdot G_4\big(\tfrac{c}{a}, \tfrac{b}{a}, \tfrac{t}{a}\big) 
 \label{RealPart-P4-Reduction}
\end{equation}
where $G_4(x,y,z)$ is defined in \cref{GmLemma} with $V = C/A, W = B/A$. Applying \cref{GmLemma} to $G_4\big(\tfrac{c}{a}, \tfrac{b}{a}, \tfrac{t}{a}\big)$ immediately implies \eqref{RealPart-Pm-Positivity} with the desired condition.
\end{proof}

\section{Smoothed Explicit Inequality}
\label{GeneralizedExpIneq}
We further generalize the ``classical explicit inequality" to smoothly weighted versions of $-\tfrac{L'}{L}(s,\chi)$, similar to the well-known Weil's explicit formula. For any Hecke character $\chi \pmod{\kq}$ and function $f : [0,\infty) \rightarrow \R$ with compact support, define
\begin{align*}
\mathcal{W}(s,\chi; f) & := \sum_{\kn \subseteq \cO} \Lambda_K(\kn) \chi(\kn) (\N\kn)^{-s} f\Big(\frac{\log (\N\kn)}{\cL} \Big) \qquad \text{for $\sigma > 1$,} \\
\cK(s,\chi;f) & := \Re\{ \mathcal{W}(s,\chi;f) \}.
\end{align*}
We begin with the same setup as \cite[Section 5]{HBLinnik}. Assume $f$ satisfies the following condition:\\

\noindent
\textbf{Condition 1} \emph{Let $f$ be a continuous function from $[0,\infty)$ to $\R$, supported in $[0,x_0)$ and bounded absolutely by $M$, and let $f$ be twice differentiable on $(0,x_0)$, with $f''$ being continuous and bounded by $B$.}\\

Recall that the \emph{Laplace transform of $f$} is given by
\begin{equation}
F(z) :=  \int_0^{\infty} e^{-zt} f(t) dt, \qquad z \in \C. 
\label{Laplace}
\end{equation}
Note $F(z)$ is entire since $f$ has compact support. For $\Re(z) > 0$, we have
\begin{equation}
F(z) = \frac{1}{z} f(0) + F_0(z),
\label{F0Laplace}
\end{equation}
where 
\begin{equation}
|F_0(z)| \leq |z|^{-2} A(f)
\label{BoundF0}
\end{equation}
with
\[
A(f) = 3Bx_0 + \frac{2|f(0)|}{x_0}. 
\]
Define the \emph{content of $f$} to be 
\begin{equation}
\cC = \cC(f) := (x_0, M, B, f(0)).
\label{ContentofF}
\end{equation}
For the purposes of generality, estimates in this section will depend only on the content of $f$. For all subsequent sections, we will ignore this distinction and allow dependence on $f$ in general.  We first reduce our analysis to primitive characters and then prove the main result.
\begin{lem} Suppose $\chi \pmod{\kq}$ is induced from $\chi^* \pmod{\kf_{\chi}}$. For $\epsilon > 0$ and $f$ satisfying Condition 1, 
\[
\mathcal{W}(s, \chi; f) = \mathcal{W}(s,\chi^*; f) + O_{\cC}\Big( \frac{n_K}{\sqrt{\epsilon}} + \epsilon \cL^* \Big)
\]
uniformly in the region $\sigma > 1$. 
\label{Reduce2Prim}
\end{lem}
\begin{proof} Observe 
\[
\begin{aligned}
& \Big| \mathcal{W}(s,\chi; f) -  \mathcal{W}(s,\chi^*; f)\Big| \leq \sum_{ (\kn,\kq) \neq 1}\frac{\Lambda_K(\kn)}{\N\kn} |f( \cL^{-1}(\log \N\kn))|  \\
& \qquad  \leq M\sum_{ (\kn,\kq) \neq 1} \frac{\Lambda_K(\kn)}{\N\kn} 
= M\sum_{ \kp \mid \kq} \sum_{j \geq 1} \frac{\log \N\kp}{(\N\kp)^{j}}  \leq 2M \sum_{\kp \mid \kq} \frac{\log \N\kp}{\N\kp}. 
 \end{aligned}
\]
The desired result then follows from \cref{ImprimitiveQ}. 
\end{proof}

\begin{prop} \label{ExplicitNP} Let $\chi \pmod{\kq}$ and $\epsilon > 0$ be arbitrary, and suppose $s = \sigma + it$ satisfies
\[
|\sigma-1| \leq \frac{(\log \cL)^{1/2}}{\cL}, \quad |t| \leq \cT.
\]
 Suppose $f$ satisfies Condition 1 and that $f(0) \geq 0 $. Then there exists $\delta = \delta(\cC, \epsilon) \in (0,1)$ depending only on $\epsilon$ and the content of $f$ (and independent of $\chi, \kq, K$ and $s$) such that 
 \begin{equation}
 \label{ExplicitNP1}
\begin{aligned}
\tfrac{1}{\cL}\cdot \cK(s,\chi; f) & \leq  E_0(\chi) \cdot \Re\{  F((s-1)\cL) \}  - \sum_{|1+it-\rho| \leq \delta} \Re\{ F( (s-\rho) \cL) \} \\
& \qquad + f(0)\phi \tfrac{\cL_{\chi}}{\cL}  + \epsilon
\end{aligned}
\end{equation}
provided $\cL$ is sufficiently large depending on $\epsilon$ and the content of $f$. 
\end{prop}
\begin{proof} The proof will closely follow the arguments of \cite[Lemma 5.2]{HBLinnik}. Let $\chi^* \pmod{\kf_{\chi}}$ be the primitive character inducing $\chi$. From \cref{Reduce2Prim}, 
\[
\cK(s,\chi; f) = \cK(s, \chi^*; f) + O_{\cC}( \epsilon^{-1} n_K  + \epsilon \cL^*).
\]
Dividing both sides by $\cL$ and recalling $n_K = o(\cL)$, it follows that
\[
\cL^{-1} \cK(s,\chi; f) \leq \cL^{-1} \cK(s, \chi^*; f) + \epsilon
\]
for $\cL$ sufficiently large depending on $\epsilon$ and the content of $f$.  Thus, we may prove \eqref{ExplicitNP1} with $\cK(s, \chi^*; f)$ instead of $\cK(s, \chi; f)$.

Let $\sigma \geq 1+2\cL^{-1}$ and set $\sigma_0 := 1+\cL^{-1}$ so $\sigma_0 < \sigma$. Consider
\begin{equation}
I := \frac{1}{2\pi i} \int_{\sigma_0 -i\infty}^{\sigma_0+i\infty} \Big( -\frac{L'}{L}(w, \chi^*) \Big) F_0((s-w)\cL) dw. 
\label{LapInvNP}
\end{equation}
Since $F_0$ satisfies \eqref{BoundF0} and 
\[
-\frac{L'}{L}(w,\chi^*) \ll |\frac{\zeta_K'}{\zeta_K}(\sigma_0)| \ll n_K (\sigma_0-1)^{-1}
\]
by \cref{LogDedekind},  the integral converges absolutely. Hence, we may compute $I$ by interchanging the summation and integration, and calculating the integral against $(\N\kn)^{-w}$ term-wise. That is to say,
\begin{equation}
I = \sum_{\kn \subseteq \cO} \Lambda(\kn) \chi^*(\kn) \Big( \frac{1}{2\pi i} \int_{\sigma_0-i\infty}^{\sigma_0+i\infty} (\N \kn)^{-w} F_0( (s-w)\cL) dw \Big). 
\label{LIswap}
\end{equation}
Arguing as in \cite[Section 5, p.21]{HBLinnik} and using Lebesgue's Dominated Convergence Theorem, one can verify
\[
\frac{1}{2\pi i} \int_{\sigma_0-i\infty}^{\sigma_0+i\infty} (\N \kn)^{-w} F_0( (s-w)\cL) dw = \frac{(\N \kn)^{-s}}{\cL} \cdot \big( f(\cL^{-1} \log \N\kn) - f(0) \big)
\]
since $f$ satisfies Condition 1.   Substituting this result into \eqref{LIswap}, we see that
\begin{equation}
I = \frac{1}{\cL} \Big( \mathcal{W}(s, \chi^*; f)  + \frac{L'}{L}(s,\chi^*) f(0) \Big). 
\label{LapInv1}
\end{equation}
Returning to \eqref{LapInvNP}, we shift the line of integration from $(\sigma_0 \pm \infty)$ to $(-\frac{1}{2} \pm \infty)$ yielding
\begin{equation}
\begin{aligned}
I & =  E_0(\chi) F_0((s-1)\cL) - \sum_{\rho} F_0((s-\rho) \cL)  \\
& \quad - r(\chi) F_0(s\cL) + \frac{1}{2\pi i} \int_{-\frac{1}{2} -i\infty}^{-\frac{1}{2}+i\infty} \Big( -\frac{L'}{L}(w, \chi^*) \Big) F_0((s-w)\cL) dw 
\end{aligned}
\label{LapInv2a}
\end{equation}
where the sum is over the non-trivial zeros of $L(w,\chi)$ and $r(\chi) \geq 0$ is the order of the trivial zero $w=0$ of $L(w,\chi^*)$. From \eqref{L-Infinite}, notice $r(\chi) \leq n_K$ so by \eqref{BoundF0},
\[
r(\chi) |F_0(s\cL)| \ll \frac{n_K A(f)}{|s \cL|^2} \ll  \frac{n_K A(f)}{\cL^2} \ll \frac{A(f)}{\cL}.
\]
To bound the remaining integral in \eqref{LapInv2a}, we apply the functional equation \eqref{FunctionalEquation} of $L(w,\chi^*)$ and \cref{LogDiff-Infinite}; namely, we note for $\Re\{w\} = -1/2$ that
\begin{align*}
-\frac{L'}{L}(w,\chi^*) & =  \cL_{\chi}^* + \frac{L'}{L}(1-w, \bar{\chi^*}) + O( n_K \log(2 + |w|) ) = \cL_{\chi}^* + O(n_K \log(2+|w|))
\end{align*}
using \cref{LogDedekind} since $\Re\{1-w\} = 3/2$. From \eqref{BoundF0}, we therefore find that
\begin{align*}
& \hspace*{-0.5in} \frac{1}{2\pi i} \int_{-\frac{1}{2} -i\infty}^{-\frac{1}{2}+i\infty} \Big( -\frac{L'}{L}(w, \chi^*) \Big) F_0((s-w)\cL) dw \\
&  = \frac{\cL_{\chi}^*}{2\pi i}\int_{-\frac{1}{2} -i\infty}^{-\frac{1}{2}+i\infty} F_0((s-w)\cL) dw + O\Big( \frac{A(f)}{\cL^2}  \int_{-\frac{1}{2} -i\infty}^{-\frac{1}{2}+i\infty}  \frac{n_K \cdot \log(2+|w|)}{|s-w|^2}  dw \Big)
\end{align*}
Since $F_0$ is entire and satisfies \eqref{BoundF0}, we may pull the line of integration in the first integral as far left as we desire, concluding that the first integral vanishes. One can readily verify that integral in the error term is  
\[
\ll \frac{A(f)}{\cL^2} \cdot n_K \log(2+|s|) \ll \frac{A(f) n_K \log \cT}{\cL^2} \ll  \frac{A(f)}{\cL}
\]
by \cref{QuantityRelations}. Combining these bounds into \eqref{LapInv2a} and comparing with \eqref{LapInv1}, we deduce
\begin{equation}
\begin{aligned}
\tfrac{1}{\cL} \cdot \mathcal{W}(s,\chi^*; f) & = -\frac{L'}{L}(s,\chi^*) f(0) \tfrac{1}{\cL} + E_0(\chi)  \cdot F_0((s-1)\cL)  
 -  \sum_{\rho} F_0((s-\rho)\cL) + O\Big( \frac{A(f)}{\cL}   \Big). 
\end{aligned}
\label{LapInv2b}
\end{equation}
We wish to apply \cref{ClassicalEI} giving $\delta = \delta(\epsilon)$, but we must discard zeros in the above sum where $|1+it-\rho| \geq \delta$. By \eqref{BoundF0}, \cite[Lemma 2.1]{LMO}, and \cref{QuantityRelations}, this discard induces an error
\begin{align*}
\ll \sum_{|1+it-\rho|\geq \delta} \frac{A(f)}{\cL^2|s-\rho|^2} 
& \ll_{\delta} \frac{A(f)}{\cL^2} \sum_{\rho} \frac{1}{1+|t-\gamma|^2} \ll_{\delta} \frac{A(f)}{\cL^2}(\cL^*+n_K \log \cT) \ll_{\delta} \frac{A(f) }{\cL}.
\end{align*}
Hence, taking real parts of \eqref{LapInv2b}, applying \cref{ClassicalEI}, and using \eqref{F0Laplace}, we find
\begin{align*}
\frac{\cK(s,\chi^*; f)}{ \cL}  & \leq E_0(\chi) \Re\Big\{ F((s-1)\cL)  -\sum_{|1+it-\rho| < \delta}  F((s-\rho)\cL) \Big\} + f(0) \big( \phi + \epsilon \big) \tfrac{\cL_{\chi}}{\cL}  + O_{\delta, \cC}\big( \cL^{-1} \big).
\end{align*}
Taking $\cL$ sufficiently large depending on $\epsilon$ and the content of $f$, the error term may be made arbitrarily small. Upon choosing a new $\epsilon$, we have established \eqref{ExplicitNP1} in the range
\[
1+2\cL^{-1} \leq \sigma \leq 1 + (\log \cL)^{1/2} \cL^{-1}.
\]
Similar to the discussion in \cite[Section 5, p.22-23]{HBLinnik}, one may show \eqref{ExplicitNP1} holds in the desired extended range by considering $g(t) = e^{\alpha t}f(t)$ for  $0 \leq \alpha \leq (\log \cL)/3x_0$. 
\end{proof}

In analogy with \cref{PolyEI} and \cref{PolyEI-Apply}, we would like to use \cref{ExplicitNP} restricting the sum over zeros $\rho$ to just a few specified zeros. To do so, we require our weight $f$ to satisfy an additional condition which was introduced in \cite[Section 6]{HBLinnik}. \\

\noindent
\textbf{Condition 2} \emph{The function $f$ is non-negative. Moreover, its Laplace transform $F$ satisfies
\[
\Re\{ F(z) \} \geq 0 \text{ for $\Re(z) \geq 0$}.
\]
}

\vspace{-4mm} 
\noindent
Condition 2 implies that, viewed as a real-variable function of $t \in \R$, $F(t)$ is a positive decreasing real-valued function. We may now give a more convenient version of \cref{ExplicitNP} in the following lemma. 
\begin{lem} \label{ExplicitNP-Apply} Let $\epsilon \in (0,1)$ be arbitrary, and let $s = \sigma + it$ with
\[
|\sigma-1| \leq \frac{\log\log \cT}{C_0 \cL}, \qquad |t| \leq 5T_0
\]
where constants $C_0 > 0$ and $T_0 \geq 1$ come from \cref{ZeroGap}. Write $\sigma = 1 -  \lambda/\cL$ and $t = \mu/\cL.$ 

Let $\chi \pmod{\kq}$ be an arbitrary Hecke character and let $\mathcal{Z} := \{ \tilde{\rho}_1, \tilde{\rho}_2, \dots, \tilde{\rho}_J\}$
 be a finite, possibly empty, multiset of zeros of $L(s,\chi)$ (called the \underline{extracted zeros}) containing the multiset
 \[
 \{ \rho_{\chi} :  \sigma < \beta_{\chi} \leq 1, \quad |\gamma_{\chi}| \leq T_0 \}. 
 \]
 Write $\tilde{\rho_j} = \tilde{\beta_j} + i\tilde{\gamma_j} =  \Big(1-\dfrac{\tilde{\lambda_j}}{\cL}\Big) + i\cdot \dfrac{\tilde{\mu_j}}{\cL}$ for $1 \leq j \leq J$ and suppose $f$ satisfies Conditions 1 and 2. Then
 \begin{equation*}
\begin{aligned}
\cL^{-1} \cdot \cK(s,\chi; f) & \leq E_0(\chi) \cdot \Re\{  F(-\lambda + i\mu) \} - \sum_{j=1}^J \Re\{ F(\tilde{\lambda_j}-\lambda - i(\tilde{\mu_j}-\mu) ) \}   + f(0) \phi \tfrac{\cL_{\chi}}{\cL}   + \epsilon
\end{aligned}
\end{equation*}
for $\cL$ sufficiently large depending only on $\epsilon$, the content of $f$, and the number of extracted zeros $J$. 
\end{lem}
\begin{rem*} The dependence of ``sufficiently large" on $J$ is insignificant for our purposes, as we will employ the lemma with $0 \leq J \leq 10$ in all of our applications.
\end{rem*}

\begin{proof} From \cref{ExplicitNP}, it follows that
\begin{equation}
\begin{aligned}
\frac{\cK(s,\chi; f)}{\cL} & \leq f(0)\big( \phi \tfrac{\cL_{\chi}}{\cL} + \epsilon \big)  + E_0(\chi) \cdot \Re\{  F(-\lambda + i\mu) \} \\
& \qquad - \sum_{|1+it-\rho| \leq \delta} \Re\{ F( (s-\rho) \cL) \}
\end{aligned}
\label{EI-Application-0}
\end{equation}
for some $\delta = \delta(\epsilon,\cC)$.  We consider the sum over zeros depending on whether $\rho \in \cZ$ or not.  For any $\rho = \tilde{\rho}_j \in \mathcal{Z}$, if $|1+it-\tilde{\rho}_j| \geq \delta$, then $|\tilde{\mu}_j-\mu| \gg_{\delta} \cL$ or $\tilde{\lambda}_j \gg_{\delta} \cL$. From \eqref{F0Laplace} and \eqref{BoundF0}, it follows that
\[
\Re\{F((s-\tilde{\rho}_j)\cL)\} \ll_{\epsilon, \cC} \cL^{-1}.
\]
implying
\begin{equation}
 \sum_{ \substack{ |1+it-\rho| \leq \delta \\ \rho \in \mathcal{Z}} } \Re\{ F( (s-\rho) \cL) \} = \sum_{j=1}^J \Re\{ F(\tilde{\lambda_j}-\lambda - i(\tilde{\mu_j}-\mu) ) \} +  O_{\epsilon, \cC}( J \cL^{-1}).
 \label{EI-Application-1}
\end{equation}
Next, for all zeros $\rho = \beta + i\gamma \not\in \cZ$ satisfying $|1+it-\rho| \leq \delta$, we claim $\beta \leq \sigma$. Assuming the claim, it follows by Condition 2 that
\begin{equation}
\sum_{ \substack{ |1+it-\rho| \leq \delta \\ \rho \not\in \mathcal{Z}} } \Re\{ F( (s-\rho) \cL) \} \geq 0. 
\label{EI-Application-2}
\end{equation}
To see the claim, assume for a contradiction that $\sigma \leq \beta \leq 1$ for some zero $\rho = \beta + i\gamma$ occurring in \eqref{EI-Application-2}. As $|1+it-\rho| \leq \delta$, it follows that
\[
|\gamma| \leq |t|+\delta \leq 5T_0+1 \leq 6T_0.
\]
From \cref{ZeroGap}, either 
\[
|\gamma| \leq T_0 \qquad \text{ or } \qquad \beta \leq 1-\frac{\log \log \cT}{C_0 \cL}. 
\]
In the latter case, it follows $\beta \leq \sigma$ which is a contradiction, so it must be that $|\gamma| \leq T_0$ and $\sigma \leq \beta \leq 1$. By the assumptions of the lemma, it follows $\rho \in \mathcal{Z}$, which is also a contradiction. This proves the claim.

Therefore, combining \eqref{EI-Application-1} and \eqref{EI-Application-2}, we may conclude 
\[
- \sum_{|1+it-\rho| \leq \delta} \Re\{ F( (s-\rho) \cL) \} \leq -\sum_{j=1}^J \Re\{ F(\tilde{\lambda_j}-\lambda - i(\tilde{\mu_j}-\mu) ) \} +  O_{\epsilon, \cC}( J \cL^{-1}).
\]
Using this bound in \eqref{EI-Application-0} and taking $\cL$ sufficiently large depending on $\epsilon, \cC$ and $J$, we have the desired result upon choosing a new $\epsilon$. 
\end{proof}

We also record a lemma useful for applications of \cref{ExplicitNP-Apply} in \cref{SZ-ZeroRepulsion,CC-ZeroRepulsion}. 

\begin{lem} \label{RepelLemma} Suppose $f$ satisfies Conditions 1 and 2. For $a,b \geq 0$ and $y  \in \R$, we have that
\[
\Re\{ F(-a + iy) - F(iy) - F(b-a + iy) \} \leq 
\begin{cases} F(-a) - F(0) & \text{if } b \geq a, \\ 
F(-a) - F(b-a) & \text{if } b \leq a.
\end{cases} 
\]
\end{lem}
\begin{proof} If $b \geq a$, then by Condition 2, $\Re\{F(b-a+iy)\} \geq 0$ so the LHS of the desired inequality is
\begin{align*}
\leq \Re\{F(-a+iy)-F(iy) \}  & = \int_0^{\infty} f(t) (e^{at}-1) \cos(yt) dt   \leq \int_0^{\infty} f(t) (e^{at}-1) dt  = F(-a) - F(0)
\end{align*}
since $f(t) \geq 0$ and $a \geq 0$. A similar argument holds for $b \leq a$, except we exclude $\Re\{F(iy)\}$ in this case. 
\end{proof}

\section{Numerical Zero Density Estimate}
\label{NewZeroDensityEstimate}
Let us first introduce some notation intended only for this section. \\

\noindent
\emph{Worst Low-lying Zeros of each Character} \\
	Consider the rectangle
	\[
	 \{ s \in \C : 0 \leq \sigma \leq 1, \quad |t| \leq 1\}. 
	\]
	For each character with a zero in this rectangle, index it $\chi^{(k)}$ for $k=1,2,\dots$ with a zero $\rho^{(k)}$ in this rectangle defined by:
	\[
	\Re(\rho^{(k)}) = \max\{ \Re(\rho) : L(\rho, \chi^{(k)}) = 0, |\gamma| \leq 1\},
	\]
	so $\chi^{(j)} \neq \chi^{(k)}$ for $j \neq k$. Write
	\[
	\rho^{(k)} := \beta^{(k)} + i\gamma^{(k)}, \quad \beta^{(k)} = 1 - \frac{\lambda^{(k)}}{\cL}, \quad \gamma^{(k)} = \frac{\mu^{(k)}}{\cL}. 
	\]
	Without loss, we may assume $\lambda^{(1)} \leq \lambda^{(2)} \leq \dots$ and so on. 

	\begin{rem*} Upon comparing with the indexing given in \cref{sec:ZeroFreeGap_and_BadZeros}, we always have the bound $\lambda_k \leq \lambda^{(k)}$ for all $k$ where both quantities exist. 
	\end{rem*}

\noindent
\emph{Low-lying Zero Density} \\
	For $\lambda \geq 0$, consider the rectangle
	\[
	\cS = \cS(\lambda) := \{ s \in \C : 1- \frac{\lambda}{\cL} \leq \sigma \leq 1, \quad |t| \leq 1\}. 
	\]
	Define 
	\begin{align*}
	N = N(\lambda) & := \# \{ \chi \neq \chi_0 \pmod{\kq} \mid L(s,\chi) \text{ has a zero in }\cS(\lambda) \} =  \sum_{\substack{ \lambda^{(k)} \leq \lambda \\ \chi^{(k)} \neq \chi_0} } 1
	\end{align*}
	Below is the main result of this section which gives bounds on $N(\lambda)$ using the smoothed explicit inequality.

\begin{thm} \label{NewZDE}Suppose $f$ satisfies Conditions 1 and 2 and let $\epsilon > 0$.  Assume $\lambda_1 \geq b$ for some $b \geq 0$.  For $\lambda \geq 0$,  if
\[
F(\lambda-b) > \tfrac{1}{\vartheta} f(0)\phi,
\]
and
\[
\Big( F(\lambda-b) - \tfrac{1}{\vartheta}f(0)\phi \Big)^2 > \tfrac{1}{\vartheta} f(0)\phi \Big(f(0)\phi + F(-b)\Big)
\]
then unconditionally, 
\begin{equation}
N(\lambda)\leq 
\frac{\Big( f(0)\phi + F(-b) \Big)\Big( F(-b) - (\tfrac{1}{\vartheta}-1)f(0)\phi \Big)}
{\Big( F(\lambda-b) - \tfrac{1}{\vartheta}f(0)\phi \Big)^2 -\tfrac{1}{\vartheta} f(0)\phi \Big(f(0)\phi + F(-b) \Big)} + \epsilon
\label{ZDE-eqn}
\end{equation}
for $\cL$ sufficiently large depending on $\epsilon$ and $f$. 
\end{thm}
\begin{rem*} Recall $\vartheta \in [\tfrac{3}{4}, 1]$ by the definition of $\cL$ in \cref{sec:ZeroFreeGap_and_BadZeros}. 
\end{rem*}
\begin{rem*} If $\zeta_K(s)$ has a real zero in $\cS(\lambda)$, then one can extract this zero from $\cK(\sigma, \chi_0; f)$ in the argument below and hence improve \eqref{ZDE-eqn} to 
\[
N(\lambda)\leq 
\frac{\Big( f(0)\phi + F(-b) - F(\lambda-b)\Big)\Big( F(-b)- F(\lambda-b) - (\tfrac{1}{\vartheta}-1)f(0)\phi \Big)}
{\Big( F(\lambda-b) - \tfrac{1}{\vartheta}f(0)\phi \Big)^2 -\tfrac{1}{\vartheta} f(0)\phi \Big(f(0)\phi + F(-b) - F(\lambda-b)\Big)} + \epsilon
\]
with naturally modified assumptions. The utility of such a bound is not entirely clear. If the real zero is exceptional, then the Deuring-Heilbronn  phenomenon from \cref{SZ-ZeroRepulsion} would likely be a better substitute. 
\end{rem*}

\begin{proof} We closely follow the arguments in \cite[Section 12]{HBLinnik}. Let $\chi \pmod{\kq}$ denote a non-principal character with a zero $\tilde{\rho} = \tilde{\beta} + i\tilde{\gamma}$ in $\cS(\lambda)$; that is, $b \leq \lambda_1 \leq \tilde{\lambda} \leq \lambda.$
Applying \cref{ExplicitNP-Apply} with $s = \sigma+i\tilde{\gamma}$ where $\sigma = 1- \tfrac{b}{\cL}$ and $\mathcal{Z} = \{ \tilde{\rho}\}$ we find that
\begin{equation}
 \cL^{-1} \cdot \cK(\sigma + i\tilde{\gamma}, \chi; f)  \leq f(0) \phi \tfrac{\cL_{\chi}}{\cL}   - F(\tilde{\lambda}-b)+ \epsilon
\label{ZD0} 
\end{equation}
for $\cL$ sufficiently large depending on $\epsilon$ and the content of $f$.  Since $F$ is decreasing by Condition 2,  it follows that 
$F(\tilde{\lambda}-b) \geq F(\lambda-b).$ Also recalling  that $\tfrac{\cL_{\chi}}{\cL} \leq \vartheta^{-1}$ by \eqref{cL-definition} and \eqref{cL-definition-1}, we see that \eqref{ZD0} implies:
\begin{equation} 
\cL^{-1} \cdot \cK(\sigma + i\tilde{\gamma}, \chi; f) \leq  
 f(0) \tfrac{1}{\vartheta} \phi  - F(\lambda-b) + \epsilon. 
\label{ZD1}
\end{equation}
Summing \eqref{ZD1} over $\chi = \chi^{(j)}$ (which are non-principal by construction) and $\tilde{\gamma} = \gamma^{(j)}$ for $j=1,\dots,N$ where $N = N (\lambda)$, we deduce that
\begin{equation}
\begin{aligned}
\Big( F(\lambda-b) - f(0)\tfrac{1}{\vartheta}\phi - \epsilon \Big) N \cL &  \leq   -\sum_{j \leq N} \cK(\sigma+i\gamma^{(j)}, \chi^{(j)}; f) \\
& = -\sum_{(\kn,\kq) =1} \Lambda(\kn) (\N\kn)^{-\sigma} f\Big( \frac{\log \N\kn}{\cL}\Big) \Re\Big\{ \sum_{j \leq N} \chi^{(j)}(\kn) (\N\kn)^{-i\gamma^{(j)}} \Big\} \\
& \leq \sum_{(\kn,\kq) =1} \Lambda(\kn) (\N\kn)^{-\sigma} f\Big( \frac{\log \N\kn}{\cL}\Big) \Big| \sum_{j \leq N} \chi^{(j)}(\kn) (\N\kn)^{-i\gamma^{(j)}} \Big|. \\
\end{aligned}
\label{CSready}
\end{equation}
The LHS of \eqref{CSready} is positive by assumption so after squaring both sides of \eqref{CSready}, we apply Cauchy-Schwarz to the last expression on the RHS implying
\[
\big(\text{LHS of \eqref{CSready}}\big)^2 \leq \ds S_1 S_2
\]
where
\begin{align*}
S_1 & = \sum_{(\kn,\kq)=1} \Lambda(\kn) (\N\kn)^{-\sigma} f\Big( \frac{\log \N\kn}{\cL}\Big) = \cK(\beta, \chi_0; f), \\
\text{and} \qquad S_2 & = \sum_{(\kn,\kq) =1} \Lambda(\kn) (\N\kn)^{-\sigma} f\Big( \frac{\log \N\kn}{\cL}\Big) \Big| \sum_{j \leq N} \chi^{(j)}(\kn) (\N\kn)^{-i\gamma^{(j)}} \Big|^2   \\
& = \sum_{j,k \leq N} \cK(\sigma+i(\gamma^{(j)}-\gamma^{(k)}), \chi^{(j)}\bar{\chi}^{(k)}; f). 
\end{align*}
The  $1$ term from $S_1$ and the $N$ terms in $S_2$ with $j=k$ give
\[
\cK(\sigma,\chi_0; f) \leq \cL \big( f(0) \phi  + F(-b )  + \epsilon \big)
\]
by \cref{ExplicitNP-Apply}. For the $N^2-N$ terms in $S_2$ with $j \neq k$, apply \cref{ExplicitNP-Apply} extracting no zeros to see that
\[
\cK(\sigma+i(\gamma^{(j)}-\gamma^{(k)}), \chi^{(j)}\bar{\chi}^{(k)}; f) \leq \cL \big( f(0) \tfrac{1}{\vartheta} \phi + \epsilon \big).
\]
 Therefore, from \eqref{CSready}, we conclude
\begin{align*}
&  \Big( F(\lambda-b) - f(0)\tfrac{1}{\vartheta}\phi - \epsilon \Big)^2 N^2 \cL^2  \\
& \quad \leq   \cL \Big[ f(0) \phi + \epsilon  + F(-b )  \Big]  \times \cL \Big[ \Big( f(0) \phi + \epsilon + F(-b )  \Big) N + \big(f(0) \tfrac{1}{\vartheta}\phi + \epsilon \big) (N^2-N) \Big] 
\end{align*}
Dividing both sides by $N \cL^2$, solving the inequality, and choosing a new $\epsilon > 0$ depending on $f$, we find
\[
N \leq 
\frac{\Big( f(0)\phi + F(-b) \Big)\Big( F(-b) - (\tfrac{1}{\vartheta}-1)f(0)\phi \Big)}
{\Big( F(\lambda-b) - \tfrac{1}{\vartheta}f(0)\phi \Big)^2 -\tfrac{1}{\vartheta} f(0)\phi \Big(f(0)\phi + F(-b) \Big)} + \epsilon
\]
provided the denominator is $> 0$ which is one of our hypotheses.
\end{proof}

To demonstrate the utility of \cref{NewZDE}, we produce a table of numerical bounds for $N(\lambda)$. Just as in Heath-Brown's case \cite[Table 13]{HBLinnik}, it turns out that the acquired bounds only hold for certain bounded ranges of $\lambda \in [0, \lambda_b]$ depending on $\lambda_1 \geq b$. However, for small values of $\lambda$, the resulting bounds are expected to be better than an explicit version of classical zero density estimates which has yet to be established.  

We apply \cref{NewZDE} using $\vartheta = \tfrac{3}{4}$. From the definitions of $\cL, N(\lambda)$ and $\cS(\lambda)$, it is immediate that the same bounds hold for all $\vartheta \in [\tfrac{3}{4}, 1]$. Choose the weight $f = f_{\hat{\theta}, \hat{\lambda}}$ from \cite[Lemma 7.1]{HBLinnik} with parameters $\hat{\theta}$ and $\hat{\lambda}$, say, taking
\[
\hat{\theta} =  1.63 + 1.28 b - 4.35 \lambda,  \qquad \hat{\lambda} = \lambda.
\]  
This is roughly optimal based on numerical experimentation and produces \cref{Table-NewZDE}. Only non-trivial bounds are displayed since trivially $N(\lambda) \leq 1$ for $\lambda < \lambda_1$. 
\begin{table}
\hspace*{-0.3in}
\begin{tabular}{l|c|c|c|c|c|c|c|c|c|c} 
 & $\lambda_1 \geq 0$ & $\lambda_1 \geq .0875$ & $\lambda_1 \geq .1$ & $\lambda_1 \geq .1227$ & $\lambda_1 \geq .15$ & $\lambda_1 \geq .20$ & $\lambda_1 \geq .25$ & $\lambda_1 \geq .30$ & $\lambda_1 \geq .35$  \\ 
 \hline 
 $\lambda $ & $N(\lambda)$ & $N(\lambda)$ & $N(\lambda)$ & $N(\lambda)$ & $N(\lambda)$ & $N(\lambda)$ & $N(\lambda)$ & $N(\lambda)$ & $N(\lambda)$ \\ 
 \hline 
 .1 &  2 &  2 &  &  &  &  &  &  &  \\ 
 .125 &  2 &  2 &  2 &  2 &  &  &  &  &  \\ 
 .150 &  3 &  3 &  3 &  3 &  &  &  &  &  \\ 
 .175 &  3 &  3 &  3 &  3 &  3 &  &  &  &  \\ 
 .200 &  4 &  4 &  4 &  3 &  3 &  &  &  &  \\ 
 .225 &  4 &  4 &  4 &  4 &  4 &  4 &  &  &  \\ 
 .250 &  5 &  5 &  5 &  5 &  4 &  4 &  &  &  \\ 
 .275 &  6 &  6 &  5 &  5 &  5 &  5 &  5 &  &  \\ 
 .300 &  7 &  6 &  6 &  6 &  6 &  6 &  5 &  &  \\ 
 .325 &  9 &  8 &  7 &  7 &  7 &  7 &  6 &  6 &  \\ 
 .350 &  11 &  9 &  9 &  9 &  8 &  8 &  7 &  7 &  \\ 
 .375 &  15 &  11 &  11 &  10 &  10 &  9 &  8 &  8 &  7 \\ 
 .400 &  22 &  15 &  14 &  13 &  12 &  11 &  10 &  9 &  8 \\ 
 .425 &  46 &  22 &  20 &  18 &  16 &  14 &  12 &  11 &  10 \\ 
 .450 &  $\infty$ &  41 &  36 &  29 &  24 &  19 &  16 &  13 &  12 \\ 
 .475 &  &  1087 &  207 &  85 &  51 &  30 &  22 &  18 &  15 \\ 
 .500 &  & $\infty$ & $\infty$  & $\infty$ & $\infty$  &  90 &  40 &  27 &  21 \\ 
 .525 &  &  &  &  &  & $\infty$ &  413 &  61 &  34 \\ 
 .550 &  &  &  &  &  &  & $\infty$ & $\infty$ &  127 \\ 
 .575 &  &  &  &  &  &  &  &  &  $\infty$ \\ 
 .600 &  &  &  &  &  &  &  &  &  \\ 
\end{tabular}

\caption{Bounds for $N(\lambda)$}
\label{Table-NewZDE}
\end{table}

\section{Zero Repulsion: $\chi_1$ and $\rho_1$ are real}
\label{SZ-ZeroRepulsion}

Recall the indexing of zeros from \cref{sec:ZeroFreeGap_and_BadZeros}. Throughout this section, we assume $\chi_1$ and $\rho_1$ are real. We wish to quantify the zero repulsion (also called Deuring-Heilbronn phenomenon) of $\rho_1$ with $\rho'$ and $\rho_2$ using the results of \cref{GeneralizedExpIneq,PolynomialEXP} along with various trigonometric identities analogous to the classical one:  $3 + 4\cos \theta + \cos2\theta \geq 0$. We emphasize that $\chi_1$ can be quadratic or possibly principal. 

We will primarily use the smoothed explicit inequality (\cref{ExplicitNP-Apply}) and so we assume that the weight function $f$ continues to satisfy Conditions 1 and 2. For simplicity, henceforth denote $\cK(s, \chi) = \cK(s, \chi; f).$ Suppose characters $\chi, \chi_*$ have zeros $\rho, \rho_*$ respectively. Our starting point is the trigonometric identity
\[
0 \leq \chi_0(\kn)\big( 1 + \Re\{ \chi(\kn) (\N\kn)^{i\gamma} \} \big) \big( 1 + \Re\{ \chi_*(\kn) (\N\kn)^{i\gamma_*} \} \big).
\]
Multiplying by $\Lambda(\kn) f( \cL^{-1}\log\N\kn) (\N\kn)^{-\sigma}$ and summing over $\kn$, it follows that
\begin{equation}
\begin{aligned}
0 & \leq \cK(\sigma, \chi_0)  + \cK(\sigma+i\gamma, \chi) + \cK(\sigma+i\gamma_*, \chi_*)  \\
& \qquad + \frac{1}{2} \cK(\sigma+i\gamma + i\gamma_*, \chi \chi_*)+ \frac{1}{2} \cK(\sigma+i\gamma - i\gamma_*, \chi \bar{\chi_*})
\end{aligned} \qquad \text{ for } \sigma  > 0.
\label{TrigIdentity}
\end{equation}
 In some cases, we will use a simpler trigonometric identity:
 \[
0 \leq  \chi_0(\kn) + \Re\{ \chi(\kn) (\N\kn)^{i\gamma} \} 
 \]
which similarly yields
\begin{equation}
\begin{aligned}
0 & \leq \cK(\sigma, \chi_0)  + \cK(\sigma+i\gamma, \chi) 
\end{aligned} \qquad \text{ for } \sigma  > 0.
\label{TrigIdentity-0}
\end{equation}
\subsection{Bounds for $\lambda'$}
\label{SZ-BoundsLp}

We establish zero repulsion results for $\rho'$ in terms of $\rho_1$, using different methods depending on various ranges of $\lambda_1$.  In this subsection, we intentionally include more details to proofs but in later subsections we shall omit these extra explanations as the arguments will be similar to those found here. 

\begin{lem} \label{SZ-L1LpIdentity} Assume $\chi_1$ and $\rho_1$ are real. Let $\epsilon > 0$ and suppose $f$ satisfies Conditions 1 and 2. Provided $\cL$ is sufficiently large depending on $\epsilon$ and $f$, the following holds:

\begin{enumerate}[(a)]

	\item If $\chi_1$ is quadratic and $\lambda' \leq \lambda_2$, then with $\psi = 4\phi$ it follows that
		\begin{align*}
0 \leq & \quad F(-\lambda')  - F(0) - F(\lambda_1-\lambda') + \Re\{ F(-\lambda'+i\mu') - F(i\mu') - F(\lambda_1-\lambda'+i\mu') \} + f(0)\psi+\epsilon.
\end{align*}

	\item If $\chi_1$ is principal, then with $\psi = 2\phi$ it follows that
	\	\begin{align*}
0 \leq & \quad F(-\lambda')  - F(0) - F(\lambda_1-\lambda')  + \Re\{ F(-\lambda'+i\mu')  - F(i\mu')  - F(\lambda_1-\lambda'+i\mu')\}  + f(0)\psi+\epsilon.
\end{align*}
\end{enumerate}
\end{lem}
\begin{proof} (a) In \eqref{TrigIdentity}, choose $\chi = \chi_* =  \chi_1, \rho = \rho'$ and $\rho_* = \rho_1$ with $\sigma = \beta'$ in \eqref{TrigIdentity} giving
\begin{equation}
\begin{aligned}
0 & \leq \cK(\beta', \chi_0)  + \cK(\beta'+ i\gamma', \chi_1) + \cK(\beta', \chi_1)  + \cK(\beta'+ i\gamma', \chi_0).
\end{aligned}
\label{TI-SZ-Lp}
\end{equation}
Apply \cref{ExplicitNP-Apply} to each $\cK(\ast, \ast)$ term and extract the relevant zeros as follows:
\begin{itemize}
	\item For $\cK(\beta', \chi_0)$ and $\cK(\beta'+i\gamma', \chi_0)$, extract no zeros since by assumption $\lambda' \leq \lambda_2$ yielding
\begin{equation}
\begin{aligned}
\cL^{-1} \cK(\beta', \chi_0)  & \leq f(0)\phi \tfrac{\cL_0}{\cL}   + F(-\lambda')+ \epsilon, \\
\cL^{-1} \cK(\beta'+i\gamma', \chi_0) & \leq f(0)  \phi \tfrac{\cL_0}{\cL}   + \Re\{ F(-\lambda' + i\mu')\}+ \epsilon.
\label{SZa-Chi0}
\end{aligned}
\end{equation}

 \item For $\cK(\beta'+i\gamma', \chi_1)$ and $\cK(\beta', \chi_1)$, extract $\{ \rho_1, \rho'\}$ implying
\begin{equation}
\begin{aligned}
\cL^{-1} \cK(\beta'+i\gamma', \chi_1) &  \leq f(0) \phi \tfrac{\cL_{\chi_1}}{\cL}  - F(0) - \Re\{F(\lambda_1-\lambda' + i\mu')\}+ \epsilon, \\
\cL^{-1} \cK(\beta', \chi_1) &  \leq f(0) \phi \tfrac{\cL_{\chi_1}}{\cL}  - F(\lambda_1-\lambda') - \Re\{F(i\mu')\}+ \epsilon. 
\end{aligned}
\label{SZa-Chi1a}
\end{equation}
\end{itemize}
Using \eqref{SZa-Chi0} and \eqref{SZa-Chi1a} in \eqref{TI-SZ-Lp}  and rescaling $\epsilon$,  the desired inequality follows from  except with  $\psi = \phi \cdot \tfrac{2\cL_0 + 2\cL_{\chi_1}}{\cL}.$ From \cref{QuantityRelations},  $\psi \leq 4\phi$ so we may use $\psi = 4\phi$ instead. 

\noindent (b) Use \eqref{TrigIdentity-0} with $\chi = \chi_0, \sigma = \beta'$ and $\rho = \rho'$, from which we deduce
\[
0 \leq \cK(\beta', \chi_0) + \cK(\beta' + i\gamma', \chi_0) \quad \text{ for } \sigma > 0. 
\]
Similar to (a), for both $\cK(\ast, \chi_0)$ terms, apply \cref{ExplicitNP-Apply}
extracting both zeros $\{ \rho_1,\rho' \}$ yielding
\begin{align*}
\cL^{-1} \cK(\beta', \chi_0)  & \leq  f(0)  \phi \tfrac{\cL_0}{\cL} + F(-\lambda') - F(\lambda_1-\lambda')  - \Re\{ F(i\mu')\} + \epsilon \\
\cL^{-1} \cK(\beta'+i\gamma', \chi_0)  & \leq  f(0)  \phi \tfrac{\cL_0}{\cL}    + \Re\{ F(-\lambda'+i\mu') - F(\lambda_1 - \lambda' + i\mu') \}  - F(0) + \epsilon
\end{align*}
Combined with the previous inequality, this yields the desired result with $\psi = 2\phi \cdot \frac{\cL_0}{\cL}$. By \cref{QuantityRelations}, we may use $\psi = 2\phi$ instead. \end{proof}

\subsubsection{$\lambda_1$ very small} We now obtain a preliminary version of the Deuring-Heilbronn phenomenon for zeros of $L(s,\chi_1)$. 

\begin{lem} \label{SZ-L1Lp_verysmall} Assume $\chi_1$ and $\rho_1$ are real. Let $\epsilon > 0$ and suppose $\cL$ is sufficiently large depending on $\epsilon$.

\begin{enumerate}[(a)]
	\item If $\chi_1$ is quadratic and $\lambda' \leq \lambda_2$, then either $\lambda' < 4e$ or 
	\[
	\lambda' \geq \Big(\frac{1}{2} - \epsilon\Big) \log(\lambda_1^{-1}), 
	\]
	which is non-trivial for $\lambda_1 \leq 3.5 \times 10^{-10}$. 
	\item If $\chi_1$ is principal, then either $\lambda' < 4e$ or
	\[
	\lambda' \geq \big(1 - \epsilon\big) \log(\lambda_1^{-1}),
	\]
	which is non-trivial for $\lambda_1 \leq 1.8 \times 10^{-5}$. 
\end{enumerate}
\end{lem}
\begin{proof} The proof is a close adaptation of \cite[p. 37]{HBLinnik}. 
From \cref{RepelLemma} and \cref{SZ-L1LpIdentity}, we have that
\[
0 \leq 2F(-\lambda') - F(0) - 2F(\lambda_1-\lambda') + f(0)(\psi+\epsilon).
\]
where $\psi$ depends on the cases in \cref{SZ-L1LpIdentity} and we assume $f(0) > 0$. As in \cite[p. 37]{HBLinnik}, choose 
\[
f(t) = \begin{cases} x_0 - t & 0 \leq t \leq x_0 \\
0 & t \geq x_0
\end{cases}
\]
for which Conditions 1 and 2 hold. Then by the same calculations, we see that
\begin{align*}
 2F(-\lambda') - 2F(\lambda_1-\lambda') \leq  \frac{2x_0 \lambda_1 \exp(x_0 \lambda')}{(\lambda')^2}, \qquad F(0) = \tfrac{1}{2} x_0^2, \qquad 
 f(0) = x_0,
\end{align*}
and so from the first inequality, we have that
\[
2x_0\lambda_1(\lambda')^{-2} \exp(x_0\lambda') - \frac{1}{2}x_0^2 + x_0(\psi+\epsilon) \geq 0. 
\]
Choose $x_0 := 2\psi + \frac{1}{\lambda'} + 2\epsilon$ so that the dependence on $f$ is uniform for $\lambda' \geq 1$. With this choice, our inequality above then leads to
\[
\lambda_1 \geq \frac{\lambda'}{4} \exp(-x_0 \lambda') = \frac{\lambda'}{4e} \exp(-(2\psi+2\epsilon) \lambda'). 
\]
When $\lambda' \geq 4e$, we conclude
\[
\lambda' \geq (\tfrac{1}{2\psi}-\epsilon) \log(\lambda_1^{-1}).
\]
The result in each case follows from the value of $\psi$ given in \cref{SZ-L1LpIdentity} and noting $\phi = 1/4$. 
\end{proof}
\subsubsection{$\lambda_1$ small} 

Here we create a  ``numerical version" of \cref{SZ-L1LpIdentity}. 
\begin{lem} \label{SZ-L1Lpsmall} Let $\epsilon > 0$ and for $b \geq 0$, assume $0 < \lambda_1 \leq b$ and retain the assumptions of \cref{SZ-L1LpIdentity}. Suppose, for some $\lambda_b' > 0$, we have 
\begin{equation}
2F(-\lambda_b') - 2F(b-\lambda_b') - F(0) +   f(0) \psi  \leq 0
\label{L1small-numerical}
\end{equation}
where $\psi = 4\phi$ or $2\phi$ if $\chi_1$ is quadratic or principal respectively. Then $\lambda' \geq \lambda_b' - \epsilon$ for $\cL$ sufficiently large depending on $\epsilon, b$ and $f$. 
\end{lem}
\begin{proof}
\cref{SZ-L1LpIdentity} and \cref{RepelLemma} imply that
\[
0 \leq 2F(-\lambda') - 2F(\lambda_1-\lambda') - F(0) + f(0)\psi+\epsilon.
\]
Now, by Conditions 1 and 2, the function
\[
F(-\lambda) - F(b-\lambda) = \int_0^{\infty} f(t) e^{\lambda} (1-e^{-b}) dt
\]
is an increasing function of $\lambda$ and also of $b$. Hence, the previous inequality implies that 
\[
0 \leq 2F(-\lambda') - 2F(b-\lambda') - F(0) + f(0)\psi+\epsilon.
\]
On the other hand, from the increasing behaviour of $F(-\lambda) - F(b-\lambda)$, we may deduce that, if \eqref{L1small-numerical} holds for some $\lambda_b'$, then
\[
0 \leq 2F(-\lambda) - 2F(b-\lambda) - F(0) + f(0) \psi \qquad \text{only if $\lambda \geq \lambda_b'$.}
\]
Comparing with the previous inequality and choosing a new value of $\epsilon$, we conclude $\lambda' \geq \lambda_b' - \epsilon$. See \cite[p.773]{KadNg} for details on this last argument. 
\end{proof}

 In each case, employing \cref{SZ-L1Lpsmall}  for various values of $b$ requires a choice of $f$ depending on $b$ which maximizes the computed value of $\lambda_b'$. Based on numerical experimentation, we choose $f = f_{\lambda}$ from  \cite[Lemma 7.2]{HBLinnik} with parameter $\lambda = \lambda(b)$. This produces \cref{Table-SZ-L1LpSmall-Quadratic,Table-SZ-L1LpSmall-Principal}. Note that the bounds in \cref{Table-SZ-L1LpSmall-Quadratic} are applicable in a later subsection for bounds on $\lambda_2$.

\begin{table}
\captionsetup{justification=centering}
\caption{
Bounds for $\lambda^{\star} = \lambda'$ with $\chi_1$ quadratic, $\rho_1$ real and $\lambda_1$ small; \newline and for $\lambda^{\star} = \lambda_2$ with $\chi_1$ quadratic, $\rho_1$ real,  $\chi_2$ principal and $\lambda_1$ small.
}

\begin{tabular}{l|c|c|c} 
$\lambda_1 \leq$ & $\tfrac{1}{2}\log \lambda_1^{-1} \geq$ & $\lambda^{\star} \geq$ & $\lambda$ \\
 \hline 
 $10^{-10}$ &  11.51 &  10.99 &  .8010 \\ 
 $10^{-9}$ &  10.36 &  9.920 &  .7975 \\ 
 $10^{-8}$ &  9.210 &  8.838 &  .7930 \\ 
 $10^{-7}$ &  8.059 &  7.740 &  .7873 \\ 
$10^{-6}$ &  6.908 &  6.623 &  .7796 \\ 
 $10^{-5}$ &  5.756 &  5.481 &  .7687 \\ 
 $10^{-4}$ &  4.605 &  4.303 &  .7521 \\ 
 .001 &  3.454 &  3.075 &  .7239 \\ 
 .005 &  2.649 &  2.176 &  .6896 \\ 
 .010 &  2.303 &  1.778 &  .6679 \\ 
 .015 &  2.100 &  1.542 &  .6522 \\ 
 .020 &  1.956 &  1.374 &  .6394 \\ 
 .025 &  1.844 &  1.244 &  .6283 \\ 
 \end{tabular}
 \qquad
 \begin{tabular}{l|c|c|c} 
$\lambda_1 \leq$ & $\tfrac{1}{2}\log \lambda_1^{-1} \geq$ & $\lambda^{\star} \geq$ & $\lambda$ \\
 \hline 
 .030 &  1.753 &  1.137 &  .6183 \\ 
 .035 &  1.676 &  1.048 &  .6092 \\ 
 .040 &  1.609 &  .9699 &  .6007 \\ 
 .045 &  1.551 &  .9016 &  .5927 \\ 
 .050 &  1.498 &  .8407 &  .5852 \\ 
 .055 &  1.450 &  .7859 &  .5780 \\ 
 .060 &  1.407 &  .7362 &  .5711 \\ 
 .065 &  1.367 &  .6906 &  .5644 \\ 
 .070 &  1.330 &  .6487 &  .5580 \\ 
 .075 &  1.295 &  .6098 &  .5517 \\ 
 .080 &  1.263 &  .5738 &  .5457 \\ 
 .085 &  1.233 &  .5401 &  .5397 \\ 
\end{tabular}

\label{Table-SZ-L1L2small-Chi2Principal}
\label{Table-SZ-L1LpSmall-Quadratic}
\end{table}

\begin{table}
\begin{tabular}{l|c|c|c} 
$\lambda_1 \leq$ & $\log \lambda_1^{-1} \geq$ & $\lambda' \geq$ & $\lambda$ \\
\hline
 $10^{-5}$ &  11.51 &  11.66 &  1.545 \\ 
 $10^{-4}$ &  9.210 &  9.324 &  1.516 \\ 
 .001 &  6.908 &  6.902 &  1.468 \\ 
 .005 &  5.298 &  5.135 &  1.413 \\ 
 .010 &  4.605 &  4.352 &  1.379 \\ 
 .015 &  4.200 &  3.887 &  1.355 \\ 
 .020 &  3.912 &  3.555 &  1.336 \\ 
 .025 &  3.689 &  3.297 &  1.319 \\ 
 .030 &  3.507 &  3.084 &  1.304 \\ 
 .035 &  3.352 &  2.905 &  1.291 \\ 
 .040 &  3.219 &  2.749 &  1.279 \\ 
 .045 &  3.101 &  2.611 &  1.267 \\ 
 .050 &  2.996 &  2.488 &  1.257 \\ 
 .055 &  2.900 &  2.377 &  1.246 \\ 
 .060 &  2.813 &  2.275 &  1.237 \\ 
 .065 &  2.733 &  2.181 &  1.227 \\ 
 .070 &  2.659 &  2.095 &  1.218 \\ 
 .075 &  2.590 &  2.015 &  1.210 \\ 
 .080 &  2.526 &  1.940 &  1.201 
 \end{tabular}
 \qquad
 \begin{tabular}{l|c|c|c} 
$\lambda_1 \leq$ & $\log \lambda_1^{-1} \geq$ & $\lambda' \geq$ & $\lambda$ \\
\hline
 .085 &  2.465 &  1.869 &  1.193 \\ 
 .0875 &  2.436 &  1.836 &  1.189 \\ 
 .090 &  2.408 &  1.803 &  1.185 \\ 
 .095 &  2.354 &  1.741 &  1.178 \\ 
 .100 &  2.303 &  1.681 &  1.170 \\ 
 .105 &  2.254 &  1.625 &  1.163 \\ 
 .110 &  2.207 &  1.572 &  1.156 \\ 
 .115 &  2.163 &  1.521 &  1.149 \\ 
 .120 &  2.120 &  1.472 &  1.142 \\ 
 .125 &  2.079 &  1.426 &  1.135 \\ 
 .130 &  2.040 &  1.381 &  1.129 \\ 
 .135 &  2.002 &  1.338 &  1.122 \\ 
 .140 &  1.966 &  1.297 &  1.116 \\ 
 .145 &  1.931 &  1.258 &  1.110 \\ 
 .150 &  1.897 &  1.220 &  1.103 \\ 
 .155 &  1.864 &  1.183 &  1.097 \\ 
 .160 &  1.833 &  1.148 &  1.091 \\ 
 .165 &  1.802 &  1.113 &  1.085 \\ 
 .170 &  1.772 &  1.080 &  1.079 
\end{tabular}

\caption{Bounds for $\lambda'$ with $\chi_1$ principal, $\rho_1$ real and $\lambda_1$ small.}
\label{Table-SZ-L1LpSmall-Principal}
\end{table}

\subsubsection{$\lambda_1$ medium}

 As a first attempt, we use techniques similar to before. 

\begin{lem} \label{SZ-L1Lpmed-0} Assume $\chi_1$ and $\rho_1$ are real.   Provided $\cL$ is sufficiently large, it follows that if $\rho'$ is real then
\[
\lambda' \geq 
\begin{cases} 
0.6069 &\text{ if } \chi_1 \text{ is quadratic and } \lambda' \leq \lambda_2, \\ 
1.2138 & \text{ if } \chi_1 \text{ is principal}, 
\end{cases}
\]
and if $\rho'$ is complex then
\[
\lambda' \geq
\begin{cases} 
0.1722 & \text{ if } \chi_1 \text{ is quadratic and } \lambda' \leq \lambda_2,
\\ 0.3444 & \text{ if } \chi_1 \text{ is principal}. 
\end{cases}
\]
\end{lem}
\begin{proof} If $\rho'$ is real, then $\mu' = 0$. From \cref{SZ-L1LpIdentity} it follows that
\[
0 \leq F(-\lambda') - F(0) - F(\lambda_1-\lambda') + \tfrac{1}{2} f(0)\psi + \epsilon
\]
where $\epsilon, f, \psi$ are specified in \cref{SZ-L1LpIdentity}. Since $F$ is decreasing by Condition 2, 
\[
0 \leq F(-\lambda') - 2F(0) + \tfrac{1}{2} f(0)\psi + \epsilon.
\]
We select the function from \cite[Lemma 7.5]{HBLinnik} corresponding to $k=2$. Hence,
\[
\frac{1}{\lambda'} \cos^2 \theta 	\leq \tfrac{1}{2}\psi + \epsilon.
\]
For $k =2$, we find $\theta = 0.9873...$ and so $\lambda' \geq \frac{0.6069}{\psi}$ for an appropriate choice of $\epsilon$. If $\rho'$ is complex, then we follow a similar argument  selecting $f$ from \cite[Lemma 7.5]{HBLinnik} corresponding to $k=\tfrac{3}{2}$ (i.e. $\theta = 1.2729...$).
\end{proof}

For $\rho'$ complex, a method based on \cref{PolynomialEXP} leads to better bounds than \cref{SZ-L1Lpmed-0}. 
\begin{lem} \label{SZ-L1Lpmedium} Assume $\chi_1$ and $\rho_1$ is real and also suppose $\rho'$ is complex.  Let  $\lambda > 0$ and $J > 0$. If $\cL$ is sufficiently large depending on $\epsilon, \lambda$ and $J$ then
\begin{align*}
0 & \leq (J^2 + \tfrac{1}{2}) \big( P_4(1) - P_4\big( \frac{\lambda}{\lambda+\lambda_1} \big) \big) - 2J \cdot P_4\big( \frac{\lambda}{\lambda+\lambda'} \big)  +
\begin{cases}   
2 \phi  (J+1)^2 \lambda  + \epsilon& \text{if $\chi_1$ is quadratic}, \\
\phi  (J+1)^2 \lambda + \epsilon & \text{if $\chi_1$ is principal},
\end{cases}
\end{align*}
provided
\begin{equation}
\frac{J_0}{(\lambda+\lambda')^4} + \frac{1}{(\lambda+\lambda_1)^4} > \frac{1}{\lambda^4} \qquad \text{with  $J_0 = \min\{ \tfrac{J}{2} + \tfrac{1}{2J}, 4J\}$.}
\label{SZ-L1Lpmedium-Condition}
\end{equation}
\end{lem}
\begin{rem*} Recall $P_4(X)$ is a fixed polynomial throughout the paper and is defined by \eqref{def:P_4}. 
\end{rem*}
\begin{proof} For an admissible polynomial $P(X) = \sum_{k=1}^d a_k X^k$, we begin with the inequality
\begin{align*}
0 & \leq \chi_0(\kn) (1+\chi_1(\kn)) \big( J + \Re\{ \chi_1(\kn) (\N\kn)^{-i\gamma'} \} \big)^2 \\
& = (J^2+\tfrac{1}{2})\big( \chi_0(\kn) + \chi_1(\kn) \big) + 2J \cdot \big( \Re\{ \chi_0(\kn) (\N\kn)^{-i\gamma'} \} + \Re\{ \chi_1(\kn) (\N\kn)^{-i\gamma'} \} \big)  \\
& \qquad + \tfrac{1}{2} \cdot \big( \Re\{ \chi_0(\kn) (\N\kn)^{-2i\gamma'} \} + \Re\{ \chi_1(\kn) (\N\kn)^{-2i\gamma'} \} \big). 
\end{align*}
To introduce $\cP(s,\chi) = \cP(s, \chi; P)$, we multiply the above inequality by 
\[
\frac{\Lambda(\kn)}{(\N\kn)^\sigma} \Big( \sum_{k=1}^d a_k \frac{\big( (\sigma-1) \log \N\kn\big)^{k-1}}{(k-1)!} \Big)
\]
with $\sigma = 1 + \tfrac{\lambda}{\cL}$ and sum over ideals $\kn$ yielding
\begin{equation}
\label{SZ-L1Lpmedium-Identity}
\begin{aligned}
0 & \leq (J^2+\tfrac{1}{2})\big( \cP(\sigma, \chi_0) + \cP(\sigma, \chi_1) \big)   + 2J \cdot \big( \cP(\sigma+i\gamma', \chi_0)  +  \cP(\sigma+i\gamma', \chi_1)  \big) \\
& \qquad + \tfrac{1}{2} \cdot \big(  \cP(\sigma+2i\gamma', \chi_0) +  \cP(\sigma+2i\gamma', \chi_1)\big). 
\end{aligned}
\end{equation}
Taking $P(X) = P_4(X)$ so $a_1 = 1$, we consider the two cases depending on $\chi_1$. \\

\noindent
\underline{\emph{(a) $\chi_1$ is quadratic:}} Apply \cref{PolyEI-Apply} to each $\cP(\ast, \ast)$ term in \eqref{SZ-L1Lpmedium-Identity} extracting the pole from $\chi_0$-terms and the zeros $\rho_1,\rho'$ (and possibly $\bar{\rho'}$) from the $\chi_1$-terms. Each of these applications yields the following:
\begin{align*}
\cL^{-1} \cdot \cP(\sigma, \chi_0) & \leq \phi  \tfrac{\cL_0}{\cL} + \epsilon + \frac{P_4(1)}{\lambda}, \\
\cL^{-1} \cdot \cP(\sigma, \chi_1) & \leq  \phi \tfrac{\cL_{\chi_1}}{\cL} + \epsilon - \frac{1}{\lambda} \cdot \Big( P_4\big( \frac{\lambda}{\lambda+\lambda_1} \big)  + \Re\big\{ P_4\big( \frac{\lambda}{\lambda+\lambda' + i\mu'} \big) \big\}  \Big), \\
\cL^{-1} \cdot \cP(\sigma+i\gamma', \chi_0) & \leq  \phi  \tfrac{\cL_0}{\cL} + \epsilon + \frac{1}{\lambda} \cdot  \Re\Big\{ P_4\big( \frac{\lambda}{\lambda +i\mu'} \big) \Big\}, \\
\cL^{-1} \cdot \cP(\sigma+i\gamma', \chi_1) & \leq \phi \tfrac{\cL_{\chi_1}}{\cL} + \epsilon - \frac{1}{\lambda} \cdot \Big(   P_4\big( \frac{\lambda}{\lambda+\lambda'} \big) + \Re\big\{ P_4\big( \frac{\lambda}{\lambda+\lambda_1+i\mu'} \big) + P_4\big( \frac{\lambda}{\lambda+\lambda'+2i\mu'} \big) \big\} \Big), \\
\cL^{-1} \cdot \cP(\sigma+2i\gamma', \chi_0) & \leq  \phi  \tfrac{\cL_0}{\cL} + \epsilon + \frac{1}{\lambda} \cdot  \Re\Big\{ P_4\big( \frac{\lambda}{\lambda +2i\mu'} \big) \Big\}, \\
\cL^{-1} \cdot \cP(\sigma+2i\gamma', \chi_1) & \leq \phi \tfrac{\cL_{\chi_1}}{\cL} + \epsilon  - \frac{1}{\lambda} \cdot \Re\big\{ P_4\big( \frac{\lambda}{\lambda+\lambda_1+2i\mu'} \big)   +  P_4\big( \frac{\lambda}{\lambda+\lambda' + i\mu'} \big)  \big\},
\end{align*}
provided $\cL$ is sufficiently large depending on $\epsilon$ and $\lambda$. For the term $\cP(\sigma+i\gamma',\chi_1)$ we extracted all 3 zeros of $\chi_1$, i.e. that $\mu' \neq 0$. Substituting into \eqref{SZ-L1Lpmedium-Identity} and noting $\frac{\cL_0+\cL_{\chi_1}}{\cL} \leq 2$ by \cref{QuantityRelations}, we find 
\begin{equation}
\label{SZ-L1Lpmedium-Identity-a}
\begin{aligned}
0 & \leq (J^2 + \tfrac{1}{2}) \big( P_4(1) - P_4\big( \frac{\lambda}{\lambda+\lambda_1} \big) \big) - 2J P_4\big( \frac{\lambda}{\lambda+\lambda'} \big)  - A - B  + 2\phi (J+1)^2 \lambda + \epsilon
\end{aligned}
\end{equation}
where
\begin{align*}
A & = \Re\big\{ (J^2 + 1)  \cdot  P_4\big( \frac{\lambda}{\lambda+\lambda' + i\mu'} \big) \big\}  +  2J \cdot P_4\big( \frac{\lambda}{\lambda+\lambda_1+i\mu'} \big) - 2J \cdot P_4\big( \frac{\lambda}{\lambda +i\mu'} \big) \Big\}, \\
B & =  \Re\big\{ 2J \cdot  P_4\big( \frac{\lambda}{\lambda+\lambda' + 2i\mu'} \big)   +  \frac{1}{2} \cdot  P_4\big( \frac{\lambda}{\lambda+\lambda_1+2i\mu'} \big)  - \frac{1}{2} \cdot P_4\big( \frac{\lambda}{\lambda +2i\mu'} \big) \Big\}.
\end{align*}
From \cref{PmLemma}, we see that $A,B \geq 0$ provided
\begin{align*}
\frac{J^2+1}{(\lambda+\lambda')^4} + \frac{2J}{(\lambda+\lambda_1)^4} > \frac{2J}{\lambda^4}  \qquad \text{and} \qquad \frac{2J}{(\lambda+\lambda')^4} + \frac{1/2}{(\lambda+\lambda_1)^4} > \frac{1/2}{\lambda^4}.
\end{align*}
Assumption \eqref{SZ-L1Lpmedium-Condition}  implies both of these inequalities.  \\

\noindent
\underline{\emph{(b) $\chi_1$ is principal:}} Then \eqref{SZ-L1Lpmedium-Identity} becomes
\[
0 \leq (2J^2+ 1) \cP(\sigma, \chi_0) + 4J \cdot  \cP(\sigma+i\gamma', \chi_0)  + \cP(\sigma+2i\gamma', \chi_0). 
\]

We similarly apply \cref{PolyEI-Apply} to each term above extracting the pole and zeros $\rho_1,\rho'$ (and possibly $\bar{\rho'}$). Each of these applications yields the following:
\begin{align*}
\cL^{-1} \cdot \cP(\sigma, \chi_0) & \leq  \phi \tfrac{\cL_0}{\cL} + \epsilon + \frac{1}{\lambda} \Big(  P_4(1) - P_4\big( \frac{\lambda}{\lambda+\lambda_1} \big)  + \Re\Big\{ P_4\big( \frac{\lambda}{\lambda+\lambda' + i\mu'} \big) \Big\}  \Big), \\
\cL^{-1} \cdot \cP(\sigma+i\gamma', \chi_0) & \leq  \phi \tfrac{\cL_0}{\cL} + \epsilon + \frac{1}{\lambda} \Big( - P_4\big( \frac{\lambda}{\lambda+\lambda'} \big)    +   \Re\Big\{ P_4\big( \frac{\lambda}{\lambda +i\mu'} \big) - P_4\big( \frac{\lambda}{\lambda+\lambda_1+i\mu'} \big) \\
& \hspace{2.4in} - P_4\big( \frac{\lambda}{\lambda+\lambda'+2i\mu'} \big) \Big\} \Big),  & \\
\cL^{-1} \cdot \cP(\sigma+2i\gamma', \chi_0) & \leq \phi \tfrac{\cL_0}{\cL} + \epsilon  + \frac{1}{\lambda} \cdot  \Re\Big\{ P_4\big( \frac{\lambda}{\lambda +2i\mu'} \big) - P_4\big( \frac{\lambda}{\lambda+\lambda_1+2i\mu'} \big)   - P_4\big( \frac{\lambda}{\lambda+\lambda' + i\mu'} \big) \Big\}.
\end{align*}
Substituting these into the previous inequality, noting $\frac{\cL_0}{\cL} \leq 1$ and dividing by 2, we obtain  \eqref{SZ-L1Lpmedium-Identity-a} except with $2\phi$ replaced by $\phi$. Following the same argument, we obtain the desired result. 
\end{proof}

Again, we exhibit a ``numerical version" of \cref{SZ-L1Lpmedium}.
\begin{cor} \label{SZ-L1Lpmedium-Numerical} Assume $\chi_1$ and $\rho_1$ is real and suppose $\rho'$ is complex. Let $\epsilon > 0$. Suppose $0 < \lambda_1 \leq b, \lambda > 0, J > 0$ and that there exists $\lambda_b' \in [0,\infty)$ satisfying
\begin{align*}
 (J^2 + \tfrac{1}{2}) \big( P_4(1) - P_4\big( \frac{\lambda}{\lambda+ b} \big) \big) - 2J \cdot P_4\big( \frac{\lambda}{\lambda+\lambda_b'} \big)  + \psi (J+1)^2 \lambda \leq 0
\end{align*}
where $\psi = 2\phi$ or $\phi$ if $\chi_1$ is quadratic or principal respectively. Then it follows that $\lambda' \geq \lambda_b' - \epsilon$ for $\cL$ sufficiently large depending on $\epsilon, \lambda$ and $J$ provided
\[
\frac{J_0}{(\lambda+\lambda'_b)^4} + \frac{1}{(\lambda+b)^4} > \frac{1}{\lambda^4} \qquad \text{where $J_0 = \min\{ \tfrac{J}{2} + \tfrac{1}{2J}, 4J\}$. }
\] 
\end{cor}
\begin{proof} 
From \cref{SZ-L1Lpmedium}, 
\[
0 \leq (J^2 + \tfrac{1}{2}) \big( P_4(1) - P_4\big( \frac{\lambda}{\lambda+ \lambda_1} \big) \big) - 2J \cdot P_4\big( \frac{\lambda}{\lambda+\lambda'} \big)  + (2-E_0(\chi_1)) \cdot \phi (J+1)^2 \lambda + \epsilon.
\]
Since $P_4$ has non-negative coefficients and $P_4(0) = 0$, the above expression is \emph{increasing} with $\lambda_1$ and $\lambda'$. From this observation, the desired result follows. 
\end{proof}

\cref{SZ-L1Lpmedium-Numerical} gives lower bounds for $\lambda'$ for certain ranges of $\lambda_1$. For each range $0 < \lambda_1 \leq b$, we choose $\lambda = \lambda(b) > 0, J = J(b) > 0$ to produce an optimal lower bound $\lambda_b'$ for $\lambda'$. This produces \cref{Table-SZ-L1Lpmedium-Quadratic,Table-SZ-L1Lpmedium-Principal}.
\begin{table}

\begin{tabular}{c|c|c|c|c} 
 $\lambda_1 \leq$ & $\tfrac{1}{2}\log(1/\lambda_1)$ & $\lambda^{\star} \geq $ & $\lambda$ & $J$ \\ 
 \hline 
 .09 &  1.204 &  .5261 &  1.239 &  .8837 \\ 
 .10 &  1.151 &  .5063 &  1.265 &  .8793 \\ 
 .11 &  1.104 &  .4880 &  1.289 &  .8752 \\ 
 .12 &  1.060 &  .4709 &  1.310 &  .8714 \\ 
 .1227 &  1.049 &  .4665 &  1.316 &  .8704 \\ 
 .13 &  1.020 &  .4549 &  1.330 &  .8677 \\ 
 .14 &  .9831 &  .4398 &  1.348 &  .8642 \\ 
 .15 &  .9486 &  .4257 &  1.364 &  .8608 \\ 
 .16 &  .9163 &  .4122 &  1.379 &  .8575 \\ 
 .17 &  .8860 &  .3995 &  1.393 &  .8544 \\ 
 .18 &  .8574 &  .3874 &  1.405 &  .8513 \\ 
 \end{tabular}
 \qquad
 \begin{tabular}{c|c|c|c|c} 
 $\lambda_1 \leq$ & $\tfrac{1}{2}\log(1/\lambda_1)$ & $\lambda^{\star} \geq $ & $\lambda$ & $J$ \\ 
 \hline 
  .19 &  .8304 &  .3759 &  1.417 &  .8483 \\ 
 .20 &  .8047 &  .3649 &  1.428 &  .8454 \\ 
 .21 &  .7803 &  .3544 &  1.438 &  .8426 \\ 
 .22 &  .7571 &  .3443 &  1.447 &  .8398 \\ 
 .23 &  .7348 &  .3347 &  1.455 &  .8370 \\ 
 .24 &  .7136 &  .3254 &  1.463 &  .8343 \\ 
 .25 &  .6931 &  .3165 &  1.471 &  .8316 \\ 
 .26 &  .6735 &  .3080 &  1.477 &  .8289 \\ 
 .27 &  .6547 &  .2998 &  1.483 &  .8263 \\ 
 .28 &  .6365 &  .2918 &  1.489 &  .8237 \\ 
 .2866 &  .6248 &  .2868 &  1.493 &  .8220 \\ 
\end{tabular}

\captionsetup{justification=centering}
\caption{Bounds for $\lambda^{\star} = \lambda'$ with $\chi_1$ quadratic, $\rho_1$ real and $\lambda_1$ medium; 
\newline and for $\lambda^{\star} = \lambda_2$ with $\chi_1$ quadratic, $\rho_1$ real, $\chi_2$ principal, and $\rho_2$ complex.}
\label{Table-SZ-L1L2medium-QuadPrincip}
\label{Table-SZ-L1Lpmedium-Quadratic}
\end{table}

\begin{table}

\begin{tabular}{c|c|c|c|c} 
 $\lambda_1 \leq $ & $\log(1/\lambda_1) \geq $ & $\lambda' \geq $ & $\lambda$ & $J$ \\ 
 \hline 
 .18 &  1.715 &  1.052 &  2.478 &  .8837 \\ 
 .19 &  1.661 &  1.032 &  2.505 &  .8815 \\ 
 .20 &  1.609 &  1.013 &  2.530 &  .8793 \\ 
 .21 &  1.561 &  .9939 &  2.555 &  .8772 \\ 
 .22 &  1.514 &  .9759 &  2.578 &  .8752 \\ 
 .23 &  1.470 &  .9586 &  2.600 &  .8733 \\ 
 .24 &  1.427 &  .9418 &  2.621 &  .8714 \\ 
 .25 &  1.386 &  .9255 &  2.641 &  .8695 \\ 
 .26 &  1.347 &  .9098 &  2.660 &  .8677 \\ 
 .27 &  1.309 &  .8945 &  2.678 &  .8659 \\ 
 .28 &  1.273 &  .8797 &  2.695 &  .8642 \\ 
 .29 &  1.238 &  .8653 &  2.712 &  .8625 \\ 
 .30 &  1.204 &  .8513 &  2.728 &  .8608 \\ 
 .31 &  1.171 &  .8377 &  2.743 &  .8592 \\ 
 .32 &  1.139 &  .8245 &  2.758 &  .8575 \\ 
 .33 &  1.109 &  .8116 &  2.772 &  .8560 \\ 
 .34 &  1.079 &  .7990 &  2.785 &  .8544 \\ 
 .35 &  1.050 &  .7867 &  2.798 &  .8528 \\ 
 .36 &  1.022 &  .7748 &  2.811 &  .8513 \\ 
  .37 &  .9943 &  .7631 &  2.822 &  .8498 \\ 
 .38 &  .9676 &  .7517 &  2.834 &  .8483 \\ 
 \end{tabular}
 \qquad
 \begin{tabular}{c|c|c|c|c} 
 $\lambda_1 \leq $ & $\log(1/\lambda_1) \geq $ & $\lambda' \geq $ & $\lambda$ & $J$ \\ 
 \hline 
  .39 &  .9416 &  .7406 &  2.845 &  .8469 \\ 
 .40 &  .9163 &  .7297 &  2.855 &  .8454 \\ 
 .41 &  .8916 &  .7191 &  2.866 &  .8440 \\ 
 .42 &  .8675 &  .7087 &  2.875 &  .8426 \\ 
 .43 &  .8440 &  .6985 &  2.885 &  .8412 \\ 
 .44 &  .8210 &  .6886 &  2.894 &  .8398 \\ 
 .45 &  .7985 &  .6788 &  2.903 &  .8384 \\ 
 .46 &  .7765 &  .6693 &  2.911 &  .8370 \\ 
 .47 &  .7550 &  .6600 &  2.919 &  .8356 \\ 
 .48 &  .7340 &  .6508 &  2.927 &  .8343 \\ 
 .49 &  .7133 &  .6418 &  2.934 &  .8329 \\ 
 .50 &  .6931 &  .6330 &  2.941 &  .8316 \\ 
 .51 &  .6733 &  .6244 &  2.948 &  .8303 \\ 
 .52 &  .6539 &  .6159 &  2.955 &  .8289 \\ 
 .53 &  .6349 &  .6076 &  2.961 &  .8276 \\ 
 .54 &  .6162 &  .5995 &  2.967 &  .8263 \\ 
 .55 &  .5978 &  .5915 &  2.973 &  .8250 \\ 
 .56 &  .5798 &  .5837 &  2.978 &  .8237 \\ 
 .57 &  .5621 &  .5760 &  2.984 &  .8224 \\ 
 .5733 &  .5563 &  .5735 &  2.985 &  .8220 \\ 
\end{tabular}

\caption{Bounds for $\lambda'$ with $\chi_1$ principal, $\rho_1$ real and $\lambda_1$ medium.}
\label{Table-SZ-L1Lpmedium-Principal}
\end{table}

\subsubsection{Summary of bounds on $\lambda'$}
We collect the results of the previous subsections for each range of $\lambda_1$ into a single result for ease of use. 

\begin{prop}  \label{SZ-L1Lp-Summary} Assume $\chi_1$ and $\rho_1$ are real. Suppose $\cL$ is sufficiently large depending on $\epsilon  > 0$. Then:
\begin{enumerate}[(a)]
	\item Suppose $\chi_1$ is quadratic and $\lambda' \leq \lambda_2$. Then
	\[
	\lambda' \geq 
	\begin{cases}
	(\tfrac{1}{2}-\epsilon) \log \lambda_1^{-1} & \text{if $\lambda_1 \leq 10^{-10}$}  \\
	0.2103 \log \lambda_1^{-1} & \text{if $\lambda_1 \leq 0.1227$}
	\end{cases}
	\]
	and if $\lambda_1 > 0.1227$ then the bounds in \cref{Table-SZ-L1Lpmedium-Quadratic} apply and $\lambda' \geq 0.2866.$
	\item Suppose $\chi_1$ is principal. If $\lambda_1 \leq 0.0875$, then 
	\[
	\lambda' \geq 
	\begin{cases}
	(1-\epsilon) \log \lambda_1^{-1} & \text{if $\lambda_1 \leq 10^{-5}$}  \\
	0.7399 \log \lambda_1^{-1} & \text{if $\lambda_1 \leq 0.0875$}
	\end{cases}
	\]
	and if $\lambda_1 > 0.0875$ then the bounds in \cref{Table-SZ-L1LpSmall-Principal,Table-SZ-L1Lpmedium-Principal} apply and 	$\lambda' \geq 0.5733. $
\end{enumerate}
\end{prop}
\begin{rem*} The constants $0.1227$ and $0.0875$ come a posteriori from the corresponding zero-free regions established in \cref{sec:ZeroFreeRegion}. 
\end{rem*}
\begin{proof} (a) Suppose $\lambda_1 \leq 10^{-10}$. From \cref{Table-SZ-L1LpSmall-Quadratic}, we see that $\lambda' \geq 10.99 > 4e$ and so the desired bound follows from \cref{SZ-L1Lp_verysmall}.  Suppose $\lambda_1 \leq 0.1227$. One compares \cref{SZ-L1Lpmed-0,Table-SZ-L1Lpmedium-Quadratic} and finds that the latter gives weaker bounds.  Thus, we only consider  \cref{Table-SZ-L1LpSmall-Quadratic,Table-SZ-L1Lpmedium-Quadratic} for this range of $\lambda_1$. For the subinterval $\lambda_1 \in [0.12, 0.1227]$, it follows that
\[
\lambda' \geq 0.4663 \geq  \frac{0.4663}{\log 1/0.12} \log \lambda_1^{-1} \geq 0.2200 \log \lambda_1^{-1}.
\]
Repeat this process for each subinterval $[10^{-10}, 10^{-9}], [10^{-9}, 10^{-8}], \dots, [0.85,0.9], \dots, [0.12, 0.1227]$ to obtain the desired bound. For $\lambda_1 > 0.1227$, one again compares \cref{SZ-L1Lpmed-0,Table-SZ-L1Lpmedium-Quadratic} and finds that the latter gives weaker bounds. For (b), we argue analogous to (a) except we only use \cref{Table-SZ-L1LpSmall-Principal} for $\lambda_1 \leq 0.0875$. 
\end{proof}

\subsection{Bounds for $\lambda_2$} 
\label{SZ-BoundsL2}
We follow the same general approach as $\lambda'$ with natural modifications. Throughout, we shall assume $\lambda_2 \leq \lambda'$; otherwise, we may use the bounds from \cref{SZ-BoundsLp} on $\lambda'$. 
\begin{lem} \label{SZ-L1L2Identity} Assume $\chi_1$ and $\rho_1$ are real and also that $\lambda_2 \leq \lambda'$. Suppose $f$ satisfies Conditions 1 and 2. For $\epsilon > 0$, provided $\cL$ is sufficiently large depending on $\epsilon$ and $f$,  the following holds: 

\begin{enumerate}[(a)]
	\item If $\chi_1,\chi_2$ are non-principal, then with $\psi = 4\phi$ it follows that
\[
0 \leq F(-\lambda_2) - F(0) - F(\lambda_1-\lambda_2) + f(0)\psi+\epsilon.
\]
	\item If $\chi_1$ is principal, then $\chi_2$ is necessarily non-principal and with $\psi = 2\phi$ it follows that
	\begin{align*}
0 \leq & \quad F(-\lambda_2)  - F(0) - F(\lambda_1-\lambda_2) + f(0)\psi+\epsilon.
\end{align*}

	\item If $\chi_2$ is principal, then $\chi_1$ is necessarily non-principal and with $\psi = 4\phi$ it follows that
	\	\begin{align*}
0 \leq & \quad F(-\lambda_2)  - F(0) - F(\lambda_1-\lambda_2)  + \Re\{ F(-\lambda_2 + i\mu_2) - F(i\mu_2) - F(\lambda_1-\lambda_2+i\mu_2)  \} + f(0)\psi + \epsilon.
\end{align*}
\end{enumerate}
\end{lem}
\begin{proof} In \eqref{TrigIdentity}, set $(\chi, \rho) = (\chi_1, \rho_1)$ and $(\chi_*, \rho_*)  = (\chi_2, \rho_2)$ and $\sigma = \beta_2$, which gives
\begin{equation}
\begin{aligned}
0 & \leq \cK(\beta_2, \chi_0)  + \cK(\beta_2, \chi_1) + \cK(\beta_2 + i\gamma_2, \chi_2) + \tfrac{1}{2} \cK(\beta_2+ i\gamma_2, \chi_1\chi_2) + \tfrac{1}{2}\cK(\beta_2- i\gamma_2, \chi_1\bar{\chi_2}).
\end{aligned}
\label{TI-SZ-L2}
\end{equation}
The arguments involved are entirely analogous to \cref{SZ-L1LpIdentity} so we omit the details here. For all cases, one applies \cref{ExplicitNP-Apply} to each $\cK(\ast, \ast)$ term, extracting $\rho_1$ or $\rho_2$ whenever possible. We remark that $\chi_1\chi_2$ and $\chi_1\bar{\chi_2}$ are always non-principal by construction (see \cref{sec:ZeroFreeGap_and_BadZeros}).
\end{proof}

\subsubsection{$\lambda_1$ very small}

We include the final result here without proof for the sake of brevity. 

\begin{lem} \label{SZ-L1L2_verysmall} Assume $\chi_1$ and $\rho_1$ are real and $\lambda_2 \leq \lambda'$.  Suppose $\cL$ is sufficiently large depending on $\epsilon > 0$.

\begin{enumerate}[(a)]
	\item If $\chi_1,\chi_2$ are non-principal, then either $\lambda_2 < 2e$ or 
	\[
	\lambda_2 \geq \Big(\frac{1}{2} - \epsilon\Big) \log(\lambda_1^{-1}), 
	\]
	which is non-trivial for $\lambda_1 \leq 1.8 \times 10^{-5}$. 
	\item If $\chi_1$ is principal, then $\chi_2$ is necessarily non-principal and either $\lambda_2 < 2e$ or
	\[
	\lambda_2 \geq \big(1 - \epsilon\big) \log(\lambda_1^{-1}),
	\]
	which is non-trivial for $\lambda_1 \leq 4.3 \times 10^{-3}$. 
	\item If $\chi_2$ is principal, then $\chi_1$ is necessarily non-principal and either $\lambda_2 < 4e$ or
	\[
	\lambda_2 \geq \Big(\frac{1}{2} - \epsilon\Big) \log(\lambda_1^{-1}),
	\]
	which is non-trivial for $\lambda_1 \leq 3.5 \times 10^{-10}$. 
\end{enumerate}
\end{lem}
\begin{proof} Analogous to \cref{SZ-L1Lp_verysmall} using \cref{SZ-L1L2Identity} in place of \cref{SZ-L1LpIdentity}. We omit the details for brevity.
\end{proof}
\subsubsection{$\lambda_1$ small} 
\begin{lem} \label{SZ-L1L2small} Assume $\chi_1$ and $\rho_1$ are real and also that $\lambda_2 \leq \lambda'$. Suppose $f$ satisfies Conditions 1 and 2.  Let $\epsilon > 0$ and assume $0 <  \lambda_1 \leq b$ for some $b > 0$. Suppose, for some $\tilde{\lambda}_b > 0$, we have 
\[
\begin{array}{ll}
 F(-\tilde{\lambda}_b) - F(b-\tilde{\lambda}_b) - F(0) +   4\phi f(0)   \leq 0 & \text{if $\chi_1, \chi_2$ are non-principal,} \\\\
 F(-\tilde{\lambda}_b) - F(b-\tilde{\lambda}_b) - F(0) +   2\phi f(0)   \leq 0 & \text{if $\chi_1$ is principal,} \\\\
 2F(-\tilde{\lambda}_b) - 2F(b-\tilde{\lambda}_b) - F(0) +   4\phi f(0)  \leq 0 & \text{if $\chi_2$ is principal} \\
\end{array}
\]
Then, according to the above cases, $\lambda_2 \geq \tilde{\lambda}_b - \epsilon$ provided $\cL$ is sufficiently large depending on $\epsilon, b$ and $f$.
\end{lem}
\begin{proof} Analogous to \cref{SZ-L1Lpsmall} using \cref{SZ-L1L2Identity} in place of \cref{SZ-L1LpIdentity}. Hence, we omit the proof.  
\end{proof}

As before,  \cref{SZ-L1L2small} requires a choice of $f$ depending on $b$ which maximizes the computed value of $\tilde{\lambda}_b$. Based on numerical experimentation, we choose $f = f_{\lambda}$ from  \cite[Lemma 7.2]{HBLinnik} with parameter $\lambda = \lambda(b)$ for all cases. This produces \cref{Table-SZ-L1L2small-Chi1Chi2NonPrincipal,Table-SZ-L1L2small-Chi2Principal,Table-SZ-L1L2small-Chi1Principal}.

\begin{table}
\begin{tabular}{l|c|c|c} 
$\lambda_1 \leq$ & $\tfrac{1}{2}\log \lambda_1^{-1} \geq$ & $\lambda_2 \geq$ & $\lambda$ \\
 \hline 
 $10^{-5}$ &  5.756 &  5.828 &  .7725 \\ 
 $10^{-4}$ &  4.605 &  4.662 &  .7579 \\ 
 .001 &  3.454 &  3.451 &  .7342 \\ 
 .005 &  2.649 &  2.569 &  .7065 \\ 
 .010 &  2.303 &  2.178 &  .6896 \\ 
 .015 &  2.100 &  1.947 &  .6776 \\ 
 .020 &  1.956 &  1.783 &  .6679 \\ 
 .025 &  1.844 &  1.654 &  .6596 \\ 
 .030 &  1.753 &  1.550 &  .6522 \\ 
 .035 &  1.676 &  1.461 &  .6455 \\ 
 .040 &  1.609 &  1.384 &  .6394 \\ 
 .045 &  1.551 &  1.317 &  .6337 \\ 
 .050 &  1.498 &  1.256 &  .6283 \\ 
 .055 &  1.450 &  1.202 &  .6232 \\ 
 .060 &  1.407 &  1.152 &  .6183 \\ 
 .065 &  1.367 &  1.107 &  .6137 \\ 
 .070 &  1.330 &  1.065 &  .6092 \\ 
 .075 &  1.295 &  1.026 &  .6049 \\ 
 .080 &  1.263 &  .9895 &  .6007 \\ 
 .085 &  1.233 &  .9555 &  .5967 \\ 
 .090 &  1.204 &  .9236 &  .5928 \\ 
 .095 &  1.177 &  .8935 &  .5890 \\ 
 .100 &  1.151 &  .8652 &  .5853 \\ 
 .105 &  1.127 &  .8383 &  .5816 \\ 
 .110 &  1.104 &  .8127 &  .5781 \\ 
 .115 &  1.081 &  .7884 &  .5746 \\ 
 .120 &  1.060 &  .7653 &  .5712 \\ 
 .125 &  1.040 &  .7432 &  .5679 \\ 
 .130 &  1.020 &  .7221 &  .5646 \\ 
 .135 &  1.001 &  .7019 &  .5613 \\ 
 .140 &  .9831 &  .6825 &  .5582 \\ 
 .145 &  .9655 &  .6639 &  .5550 \\ 
 .150 &  .9486 &  .6460 &  .5520 \\

\end{tabular}
 \qquad
 \begin{tabular}{l|c|c|c} 
 $\lambda_1 \leq$ & $\tfrac{1}{2}\log \lambda_1^{-1} \geq$ & $\lambda_2 \geq$ & $\lambda$ \\
 \hline 
 .155 &  .9322 &  .6288 &  .5489 \\ 
 .160 &  .9163 &  .6122 &  .5459 \\ 
 .165 &  .9009 &  .5962 &  .5429 \\ 
 .170 &  .8860 &  .5808 &  .5400 \\ 
 .175 &  .8715 &  .5659 &  .5371 \\ 
 .180 &  .8574 &  .5515 &  .5342 \\ 
 .185 &  .8437 &  .5376 &  .5314 \\ 
 .190 &  .8304 &  .5242 &  .5286 \\ 
 .195 &  .8174 &  .5111 &  .5258 \\ 
 .200 &  .8047 &  .4985 &  .5231 \\ 
 .205 &  .7924 &  .4863 &  .5203 \\ 
 .210 &  .7803 &  .4744 &  .5176 \\ 
 .215 &  .7686 &  .4629 &  .5150 \\ 
 .220 &  .7571 &  .4517 &  .5123 \\ 
 .225 &  .7458 &  .4408 &  .5097 \\ 
 .230 &  .7348 &  .4302 &  .5070 \\ 
 .235 &  .7241 &  .4200 &  .5044 \\ 
 .240 &  .7136 &  .4100 &  .5018 \\ 
 .245 &  .7032 &  .4002 &  .4993 \\ 
 .250 &  .6931 &  .3908 &  .4967 \\ 
 .255 &  .6832 &  .3816 &  .4942 \\ 
 .260 &  .6735 &  .3726 &  .4916 \\ 
 .265 &  .6640 &  .3638 &  .4891 \\ 
 .270 &  .6547 &  .3553 &  .4866 \\ 
 .275 &  .6455 &  .3470 &  .4841 \\ 
 .280 &  .6365 &  .3389 &  .4817 \\ 
 .285 &  .6276 &  .3310 &  .4792 \\ 
 .290 &  .6189 &  .3233 &  .4768 \\ 
 .295 &  .6104 &  .3158 &  .4743 \\ 
 .300 &  .6020 &  .3084 &  .4719 \\ 
\end{tabular}
\caption{Bounds for $\lambda_2$ with $\chi_1$ quadratic, $\rho_1$ real,  $\chi_2$ non-principal and $\lambda_1$ small.}
\label{Table-SZ-L1L2small-Chi1Chi2NonPrincipal}
\end{table}

\begin{table}
\begin{tabular}{l|c|c|c} 
$\lambda_1 \leq$ & $\log \lambda_1^{-1} \geq$ & $\lambda_2 \geq$  & $\lambda$ \\
 \hline 
 .004 &  5.521 &  6.150 &  1.448 \\ 
 .006 &  5.116 &  5.705 &  1.434 \\ 
 .008 &  4.828 &  5.386 &  1.422 \\ 
 .010 &  4.605 &  5.137 &  1.413 \\ 
 .015 &  4.200 &  4.682 &  1.394 \\ 
 .020 &  3.912 &  4.357 &  1.379 \\ 
 .025 &  3.689 &  4.103 &  1.366 \\ 
 .03 &  3.507 &  3.895 &  1.355 \\ 
 .04 &  3.219 &  3.565 &  1.336 \\ 
 .05 &  2.996 &  3.309 &  1.319 \\ 
 .06 &  2.813 &  3.099 &  1.304 \\ 
 .07 &  2.659 &  2.922 &  1.291 \\ 
 .08 &  2.526 &  2.769 &  1.279 \\ 
 .0875 &  2.436 &  2.666 &  1.270 \\ 
 .10 &  2.303 &  2.513 &  1.257 \\ 
 .12 &  2.120 &  2.304 &  1.237 \\ 
 .14 &  1.966 &  2.130 &  1.218 \\ 
 .16 &  1.833 &  1.979 &  1.201 \\ 
 .18 &  1.715 &  1.847 &  1.186 \\ 
  .20 &  1.609 &  1.730 &  1.171 \\ 
 .22 &  1.514 &  1.625 &  1.156 \\ 
 \end{tabular}
 \qquad
 \begin{tabular}{l|c|c|c} 
$\lambda_1 \leq$ & $\log \lambda_1^{-1} \geq$ & $\lambda_2 \geq$  & $\lambda$ \\
 \hline 
  .24 &  1.427 &  1.531 &  1.142 \\ 
 .26 &  1.347 &  1.444 &  1.129 \\ 
 .28 &  1.273 &  1.365 &  1.116 \\ 
 .30 &  1.204 &  1.292 &  1.104 \\ 
 .32 &  1.139 &  1.224 &  1.092 \\ 
 .34 &  1.079 &  1.162 &  1.080 \\ 
 .36 &  1.022 &  1.103 &  1.068 \\ 
 .38 &  .9676 &  1.048 &  1.057 \\ 
 .40 &  .9163 &  .9970 &  1.046 \\ 
 .42 &  .8675 &  .9488 &  1.035 \\ 
 .44 &  .8210 &  .9033 &  1.025 \\ 
 .46 &  .7765 &  .8605 &  1.014 \\ 
 .48 &  .7340 &  .8199 &  1.004 \\ 
 .50 &  .6931 &  .7816 &  .9934 \\ 
 .52 &  .6539 &  .7452 &  .9833 \\ 
 .54 &  .6162 &  .7106 &  .9733 \\ 
 .56 &  .5798 &  .6778 &  .9633 \\ 
 .58 &  .5447 &  .6466 &  .9535 \\ 
 .60 &  .5108 &  .6168 &  .9438 \\ 
 .6068 &  .4996 &  .6070 &  .9405 \\ 
\end{tabular}

\caption{Bounds for $\lambda_2$ with $\chi_1$ principal, $\rho_1$ real and $\lambda_1$ small.}
\label{Table-SZ-L1L2small-Chi1Principal}
\end{table}

\subsubsection{$\lambda_1$ medium}

We first deal with the case when $\rho_2$ is real and $\chi_2$ is principal, i.e. $\mu_2 = 0$.  

\begin{lem} \label{SZ-L1L2med-0} Assume $\chi_1$ and $\rho_1$ are real. Suppose $\cL$ is sufficiently large. If $\rho_2$ is real, then
\[
\lambda_2 \geq 
\begin{cases} 
0.3034 & \text{if $\chi_1, \chi_2$ are non-principal},  \\
0.6069 & \text{otherwise}.
\end{cases}
\]
If $\rho_2$ is complex, then
\[
\lambda_2 \geq 
\begin{cases} 
0.3034 & \text{if $\chi_1, \chi_2$ are non-principal},  \\
0.6069 & \text{if $\chi_1$ is principal}, \\
0.1722 & \text{if $\chi_2$ is principal}.
\end{cases}
\]

\end{lem}
\begin{proof} Analogous to \cref{SZ-L1Lpmed-0} using \cref{SZ-L1L2Identity} in place of \cref{SZ-L1LpIdentity}. The arguments lead to selecting $f$ from \cite[Lemma 7.5]{HBLinnik} corresponding to $k=2$ (i.e. $\theta = 0.9873...$) when $\rho_2$ is real or $\chi_2$ is non-principal, and to $k=3/2$   (i.e. $\theta = 1.2729...$) when $\rho_2$ is complex and $\chi_2$ is principal. 
\end{proof}

For $\chi_2$ principal and $\rho_2$ complex, the ``polynomial method" of \cref{PolynomialEXP} yields better bounds. 
\begin{lem} \label{SZ-L1L2medium} Assume $\chi_1$ is quadratic and $\rho_1$ is real. Further suppose $\chi_2$ is principal and $\rho_2$ is complex.  Let  $\lambda > 0$ and $J > 0$. If $\cL$ is sufficiently large depending on $\epsilon, \lambda$ and $J$, then
\begin{align*}
0 & \leq (J^2 + \tfrac{1}{2}) \big( P_4(1) - P_4\big( \frac{\lambda}{\lambda+\lambda_1} \big) \big) - 2J P_4\big( \frac{\lambda}{\lambda+\lambda_2} \big)   +  2 \phi (J+1)^2 \lambda  + \epsilon 
\end{align*}
provided
\begin{equation}
\frac{J_0}{(\lambda+\lambda_2)^4} + \frac{1}{(\lambda+\lambda_1)^4} > \frac{1}{\lambda^4}. 
\label{SZ-L1L2medium-Condition}
\end{equation}
with $J_0 = \min\{ \tfrac{J}{2} + \tfrac{1}{2J}, 4J \}$. 
\end{lem}
\begin{proof} This is analogous to \cref{SZ-L1Lpmedium} so we give a brief outline here. We begin with the inequality
\begin{align*}
0 & \leq \chi_0(\kn) (1+\chi_1(\kn)) \big( J + \Re\{ (\N\kn)^{-i\gamma_2} \} \big)^2 \\
& = (J^2+\tfrac{1}{2})\big( \chi_0(\kn) + \chi_1(\kn) \big) + 2J \cdot \big( \Re\{ (\N\kn)^{-i\gamma_2} \} + \Re\{ \chi_1(\kn) (\N\kn)^{-i\gamma_2} \} \big)  \\
& \qquad + \tfrac{1}{2} \cdot \big( \Re\{ (\N\kn)^{-2i\gamma_2} \} + \Re\{ \chi_1(\kn) (\N\kn)^{-2i\gamma_2} \} \big). 
\end{align*}
We introduce $\cP(s,\chi) = \cP(s, \chi; P_4)$ in the usual way with $\sigma = 1 + \tfrac{\lambda}{\cL}$, yielding
\begin{equation}
\label{SZ-L1L2medium-Identity}
\begin{aligned}
0 & \leq (J^2+\tfrac{1}{2})\big( \cP(\sigma, \chi_0) + \cP(\sigma, \chi_1) \big) + 2J \cdot \big( \cP(\sigma+i\gamma_2, \chi_0)  +  \cP(\sigma+i\gamma_2, \chi_1)  \big)  \\
& \qquad + \tfrac{1}{2} \cdot \big(  \cP(\sigma+2i\gamma_2, \chi_0) +  \cP(\sigma+2i\gamma_2, \chi_1)\big). 
\end{aligned}
\end{equation}
Next, apply \cref{PolyEI-Apply} to each $\cP(\ast, \ast)$ term in \eqref{SZ-L1L2medium-Identity} extracting the zero $\rho_2$ from $\chi_0$-terms and the zero $\rho_1$ from the $\chi_1$-terms. One also extracts both zeros $\{\rho_2, \bar{\rho_2}\}$ from $\cP(\sigma+i\gamma_2, \chi_0)$. Noting $\frac{\cL_0+\cL_{\chi_1}}{\cL} \leq 2$ by \cref{QuantityRelations} and choosing a new $\epsilon$, these applications yield the following:
\begin{equation}
\label{SZ-L1L2medium-Identity-a}
\begin{aligned}
0 & \leq (J^2 + \tfrac{1}{2}) \big( P_4(1) - P_4\big( \frac{\lambda}{\lambda+\lambda_1} \big) \big) - 2J P_4\big( \frac{\lambda}{\lambda+\lambda_2} \big)  - A - B + 2 \phi(J+1)^2 \lambda + \epsilon
\end{aligned}
\end{equation}
provided $\cL$ is sufficiently large depending on $\epsilon$ and $\lambda$ and where
\begin{align*}
A & = (J^2 + 1)  \Re\big\{ P_4\big( \frac{\lambda}{\lambda+\lambda_2 + i\mu_2} \big) \big\}  +  2J \cdot  \Re\big\{ P_4\big( \frac{\lambda}{\lambda+\lambda_1+i\mu_2} \big) \big\} - 2J \cdot \Re\Big\{ P_4\big( \frac{\lambda}{\lambda +i\mu_2} \big) \Big\}, \\
B & = 2J \cdot  \Re\big\{ P_4\big( \frac{\lambda}{\lambda+\lambda_2 + 2i\mu_2} \big) \big\}  +  \frac{1}{2} \cdot  \Re\big\{ P_4\big( \frac{\lambda}{\lambda+\lambda_1+2i\mu_2} \big) \big\} - \frac{1}{2} \cdot \Re\Big\{ P_4\big( \frac{\lambda}{\lambda +2i\mu_2} \big) \Big\}. 
\end{align*}
Assumption \eqref{SZ-L1L2medium-Condition} implies $A,B \geq 0$ by \cref{PmLemma} yielding the desired result from \eqref{SZ-L1L2medium-Identity-a}. \end{proof}

With an appropriate numerical version of \cref{SZ-L1L2medium}, analogous to \cref{SZ-L1Lpmedium-Numerical}, we obtain lower bounds for $\lambda_2$ for $\lambda_1 \in [0,b]$ and fixed $b > 0$. Optimally choosing $\lambda = \lambda(b) > 0, J = J(b) > 0$  produces \cref{Table-SZ-L1L2medium-QuadPrincip} again. 

\subsubsection{Summary of bounds on $\lambda_2$}
We collect the estimates of the previous subsections for each range of $\lambda_1$ into a  one result for ease of use. 

\begin{prop} \label{SZ-L1L2-Summary} Assume $\chi_1$ and $\rho_1$ are real. Suppose $\cL$ is sufficiently large depending on $\epsilon > 0$:
\begin{enumerate}[(a)]
	\item Suppose $\chi_1$ is quadratic and $\lambda_2 \leq \lambda'$. Then
	\[
	\lambda_2 \geq 
	\begin{cases}
	(\tfrac{1}{2}-\epsilon) \log \lambda_1^{-1} & \text{if $\lambda_1 \leq 10^{-10}$}  \\
	0.2103 \log \lambda_1^{-1} & \text{if $\lambda_1 \leq 0.1227$}
	\end{cases}
	\]
	and if $\lambda_1 > 0.1227$ then the bounds in \cref{Table-SZ-L1L2medium-QuadPrincip} apply and $\lambda_2 \geq 0.2866.$
	\item Suppose $\chi_1$ is principal. Then 
	\[
	\lambda_2 \geq (1-\epsilon) \log \lambda_1^{-1} \quad \text{if $\lambda_1 \leq 0.0875$}.
	\]
	and if $\lambda_1 > 0.0875$ then the bounds in \cref{Table-SZ-L1L2small-Chi1Principal} apply and $\lambda_2 \geq 0.6069.$
\end{enumerate}
\end{prop}
\begin{rem*} After comparing \cref{SZ-L1L2-Summary,SZ-L1Lp-Summary} in the case when $\chi_1$ is quadratic, we realize that the additional assumptions $\lambda'\leq \lambda_2$ or $\lambda_2 \leq \lambda'$ are superfluous. 
\end{rem*}
\begin{proof} (a) First, suppose $\chi_2$ is non-principal. For $\lambda_1 \leq 10^{-5}$, we see from \cref{Table-SZ-L1L2small-Chi1Chi2NonPrincipal} that $\lambda_2 \geq 5.828 > 2e$ so the desired bound follows form \cref{SZ-L1L2_verysmall}.  For $10^{-5} \leq \lambda_1 \leq 0.1227$, consider \cref{Table-SZ-L1L2small-Chi1Chi2NonPrincipal}. Apply the same process as in \cref{SZ-L1Lp-Summary} to each subinterval $[10^{-5}, 10^{-4}],  \dots, [0.12, 0.125]$ to obtain 
\[
\lambda_2 \geq 0.3506 \log \lambda_1^{-1}. 
\]

Now, suppose $\chi_2$ is principal. For $\lambda_1 \leq 10^{-10}$, we see from \cref{Table-SZ-L1L2small-Chi2Principal} that $\lambda_2 \geq 10.99 > 4e$ so the desired bound follows from \cref{SZ-L1L2_verysmall}. For $10^{-10} \leq \lambda_1 \leq 0.1227$, consider \cref{Table-SZ-L1L2small-Chi2Principal,Table-SZ-L1L2medium-QuadPrincip}. Apply the same process as in \cref{SZ-L1Lp-Summary} to  each subinterval $[10^{-10}, 10^{-9}],  \dots, [0.85,0.9],\dots [0.12, 0.1227]$ and obtain
\[
\lambda_2 \geq 0.2103 \log \lambda_1^{-1}.
\]
Upon comparing the two cases, the latter gives weaker results in the range $\lambda_1 \leq 0.1227$. For $\lambda_1 > 0.1227$, we compare \cref{SZ-L1L2med-0,Table-SZ-L1L2small-Chi1Chi2NonPrincipal,Table-SZ-L1L2medium-QuadPrincip} and see that  last one gives the weakest bounds.

\noindent
(b) Similar to (a) except we use \cref{Table-SZ-L1L2small-Chi1Principal} in conjunction with \cref{SZ-L1L2_verysmall}. The range $\lambda_1 \leq 0.004$ gives the bound $\lambda_2 \geq (1-\epsilon)\log \lambda_1^{-1}$. The range $0.004 \leq \lambda_1 \leq 0.0875$ turns out to give a better bound but we opt to write a bound uniform for $\lambda_1 \leq 0.0875$. For $\lambda_1 > 0.0875$, we use \cref{SZ-L1L2med-0,Table-SZ-L1L2small-Chi1Principal}. 
\end{proof}

\section{Zero Repulsion: $\chi_1$ or $\rho_1$ is complex}
\label{CC-ZeroRepulsion}
When $\chi_1$ or $\rho_1$ is complex, the effect of zero repulsion is lesser than when $\chi_1$ and $\rho_1$ are real. Nonetheless we will follow the same general outline as the previous section, but with modified trigonometric identities and more frequently using the ``polynomial method" of \cref{PolynomialEXP}. Also, whether $\chi_1$ is principal naturally affects our arguments in a significant manner so for clarity we further subdivide our results on this condition.

\subsection{Bounds for $\lambda'$} \label{CC-BoundsLp}

\subsubsection{$\chi_1$ non-principal} 
\begin{lem} \label{CC-L1LpIdentity} Assume $\chi_1$ or $\rho_1$ is complex with $\chi_1$ non-principal. Let  $\lambda > 0, J \geq \tfrac{1}{4}$. If $\cL$ is sufficiently large depending on $\epsilon, \lambda$ and $J$ then
\[
0 \leq  (J^2+\tfrac{1}{2}) P_4(1) -  (J^2+\tfrac{1}{2}) \cdot P_4\big( \frac{\lambda}{\lambda+\lambda'} \big) - 2J \cdot P_4\big( \frac{\lambda}{\lambda+\lambda_1}\big)  + 2(J+1)^2 \phi \lambda + \epsilon
\]
provided
\begin{equation}
	\frac{ J_0}{(\lambda+\lambda_1)^4} + \frac{1}{(\lambda+\lambda')^4} > \frac{1}{\lambda^4} \qquad \text{with $J_0 = \min\{ J+\tfrac{3}{4J}, 4J \}.$ }
 \label{CC-L1LpIdentity-Condition}
\end{equation}
\end{lem}
\begin{proof} For simplicity, denote $\cP(s,\chi) = \cP(s,\chi; P_4)$. Our starting point is the  trigonometric identity
\[
0 \leq  \chi_0(\kn) \big(1 + \Re\{ \chi_1(\kn) (\N\kn)^{i\gamma'} \} \big) \big( J+ \Re\{ \chi_1(\kn) (\N\kn)^{i\gamma_1} \} \big)^2.
\]
In the usual way, it follows that
\begin{equation}
\begin{aligned}
0 & \leq(J^2+\tfrac{1}{2}) \big\{ \cP(\sigma, \chi_0)  + \cP(\sigma+i\gamma', \chi_1) \big\} + \\
& \qquad + J \cP(\sigma + i(\gamma_1+\gamma'), \chi_1^2) + 2J \cP(\sigma+i\gamma_1, \chi_1) + J \cP(\sigma+i(\gamma_1-\gamma'), \chi_0) \\
& \qquad + \tfrac{1}{4} \cP(\sigma + i(2\gamma_1+\gamma'), \chi_1^3) + \tfrac{1}{2} \cP(\sigma+2i\gamma_1, \chi_1^2) + \tfrac{1}{4} \cP(\sigma+i(2\gamma_1-\gamma'), \chi_1) \\
\end{aligned} 
\label{TrigIdentity-CC-L1Lp}
\end{equation}
where $\sigma = 1 + \tfrac{\lambda}{\cL}$. To each term $\cP( \, \cdot \, , \chi_1^r)$ above, we apply \cref{PolyEI-Apply} extracting zeros depending on the order of $\chi_1$ and the value of $r$.   We divide our argument into cases. 
\begin{enumerate}[(i)]
	\item $(\ord \chi_1 \geq 4)$ \emph{Extract $\{\rho_1, \rho'\}$ from $\cP( \, \cdot \, , \chi_1^r)$ when $r=1$.} From \eqref{TrigIdentity-CC-L1Lp}, we deduce
	\begin{equation}
	\begin{aligned}
	0 \leq & (J^2+\tfrac{1}{2}) P_4(1) -  (J^2+\tfrac{1}{2}) P_4\big( \frac{\lambda}{\lambda+\lambda'} \big) - 2J P_4\big( \frac{\lambda}{\lambda+\lambda_1}\big)  -A + \lambda \psi + \epsilon  
	\end{aligned}
	\label{CC-L1Lp-Case1}
	\end{equation}
	where  $\psi = (J^2 + 3J + \tfrac{3}{2}) \phi \tfrac{\cL_{\chi_1}}{\cL} + (J^2 + J + \tfrac{1}{2}) \phi \tfrac{\cL_0}{\cL},$ and
	\[
	A =  \Re\Big\{  (J^2+\tfrac{3}{4}) P_4\big( \frac{\lambda}{\lambda+\lambda_1+it_1} \big) + 2J \cdot P_4\big( \frac{\lambda}{\lambda+\lambda' + it_1}\big) - J \cdot P_4\big( \frac{\lambda}{\lambda+it_1}\big)\Big\}
	\]
	with $t_1 = \mu' - \mu_1$. One can easily verify that $J^2 + 3J+\tfrac{3}{2} \leq 3 \cdot(J^2 + J +\tfrac{1}{2})$ and so by \cref{QuantityRelations}, we may  more simply take $\psi =  2(J+1)^2 \phi$ in \eqref{CC-L1Lp-Case1}. By \cref{PmLemma}, assumption \eqref{CC-L1LpIdentity-Condition} implies $A \geq 0$, completing the proof of case (i). 
	\\

	\item $(\ord \chi_1 = 3)$ \emph{Extract $\{\rho_1, \rho'\}$ or $\{\bar{\rho_1}, \bar{\rho'}\}$ from $\cP( \, \cdot \, , \chi_1^r)$ when $r=1$ or $2$ respectively.} Then by \eqref{TrigIdentity-CC-L1Lp}, 
	\begin{equation}
	\begin{aligned}
	0 \leq & (J^2+\tfrac{1}{2}) P_4(1) -  (J^2+\tfrac{1}{2}) P_4\big( \frac{\lambda}{\lambda+\lambda'} \big) - 2J P_4\big( \frac{\lambda}{\lambda+\lambda_1}\big) - A - B +  \lambda \psi + \epsilon 
	\end{aligned}
	\label{CC-L1Lp-Case2}
	\end{equation}
	where $\psi = (J^2 + 3J + \tfrac{5}{4}) \phi \tfrac{\cL_{\chi_1}}{\cL} + (J^2 + J + \tfrac{3}{4}) \phi \tfrac{\cL_0}{\cL}$, the quantity $A$ is as defined in case (i), and
	\begin{align*}
	B = \Re\Big\{ J \cdot P_4\big( \frac{\lambda}{\lambda+\lambda_1+it_2} \big) + \tfrac{1}{2} \cdot P_4\big( \frac{\lambda}{\lambda+\lambda' + it_2}\big)- \tfrac{1}{4} \cdot P_4\big( \frac{\lambda}{\lambda+it_2}\big)\Big\}
	\end{align*}
	with $t_2 = \mu' + 2\mu_1$. Again, one can check that $J^2 + 3J+\tfrac{5}{4} \leq 3 \cdot(J^2 + J +\tfrac{3}{4})$ and so by \cref{QuantityRelations}, we may take $\psi =  2(J+1)^2 \phi$ in \eqref{CC-L1Lp-Case2}. Similar to (i), \cref{PmLemma} and assumption \eqref{CC-L1LpIdentity-Condition} imply $A, B \geq 0$. 
	
	\item $(\ord \chi_1 = 2)$ \emph{Extract $\{\rho_1, \bar{\rho_1}, \rho'\}$ from $\cP( \, \cdot \, , \chi_1^r)$ when $r=1$ or $3$.}
	Again, apply \cref{PolyEI-Apply} to the terms in \eqref{TrigIdentity-CC-L1Lp} except with a slightly more careful analysis. We outline these modifications here. 
		\begin{itemize}
			\item Write  $2J \cdot \cP(\sigma+i\gamma_1,\chi_1) = J \cdot \cP(\sigma+i\gamma_1,\chi_1) + J \cdot \cP(\sigma+i\gamma_1, \chi_1).$  Extract $\{\rho_ 1, \bar{\rho_1}, \rho'\}$ from the first term and extract  $\{\rho_1, \bar{\rho_1}, \bar{\rho'}\}$ from the second term.
			\item For $\tfrac{1}{4} \cP(\sigma + i(2\gamma_1+\gamma'), \chi_1)$ and  $\tfrac{1}{4} \cP(\sigma + i(2\gamma_1-\gamma'), \chi_1)$, extract $\{\rho_1,  \rho'\}$ and $\{\rho_1,  \bar{\rho'}\}$ respectively.
		\end{itemize}

	 With these modifications, \eqref{TrigIdentity-CC-L1Lp} overall yields
	\begin{equation}
	\begin{aligned}
	0 \leq & (J^2+\tfrac{1}{2}) P_4(1) -  (J^2+\tfrac{1}{2}) P_4\big( \frac{\lambda}{\lambda+\lambda'} \big) - 2J P_4\big( \frac{\lambda}{\lambda+\lambda_1}\big) + \lambda (\psi + \epsilon)  \\
	& \quad - \Re\Big\{  (J^2+\tfrac{3}{4}) P_4\big( \frac{\lambda}{\lambda+\lambda_1+it_1} \big) + J \cdot P_4\big( \frac{\lambda}{\lambda+\lambda' + it_1}\big) - J \cdot P_4\big( \frac{\lambda}{\lambda+it_1}\big)\Big\} \\
	& \quad - \Re\Big\{ (J^2 + \tfrac{3}{4}) \cdot P_4\big( \frac{\lambda}{\lambda+\lambda_1+it_3} \big) + J \cdot P_4\big( \frac{\lambda}{\lambda+\lambda' + it_3}\big)- J \cdot P_4\big( \frac{\lambda}{\lambda+it_3}\big)\Big\} \\
	& \quad - \Re\Big\{ 2J \cdot P_4\big( \frac{\lambda}{\lambda+\lambda_1+it_4} \big) + \tfrac{1}{2} \cdot P_4\big( \frac{\lambda}{\lambda+\lambda' + it_4}\big)- \tfrac{1}{2} \cdot P_4\big( \frac{\lambda}{\lambda+it_4}\big)\Big\}
	\end{aligned}
	\label{CC-L1Lp-Case3}
	\end{equation}
	where $t_1 = \mu' - \mu_1; t_3 = \mu' + \mu_1; t_4 = 2\mu_1;$ and $\psi = (J^2 + 2J+1) \phi \tfrac{\cL_{\chi_1}}{\cL} + (J^2 + 2J + 1) \phi \tfrac{\cL_0}{\cL}.$ Trivially $J^2 + 2J+1 \leq 3 \cdot(J^2 + 2J+1)$ and so by \cref{QuantityRelations}, we may  more simply take $\psi =  2(J+1)^2 \phi$. The three terms $\Re\{ \dots \}$ in \eqref{CC-L1Lp-Case3} are all $\geq 0$ by \cref{PmLemma} and \eqref{CC-L1LpIdentity-Condition} and hence can be ignored.
\end{enumerate}
This completes the proof in all cases. 
\end{proof}

A suitable numerical version of \cref{CC-L1LpIdentity} produces \cref{Table-Lp-CC_Chi1NonPrincipal}.

\begin{table}
\begin{tabular}{l|c|c|c} 
 $\lambda_1 \leq$ & $\lambda' \geq$& $\lambda$ & $J$ \\ 
 \hline 
 .1227 &  .7391 &  1.097 &  .7788 \\ 
 .125 &  .7266 &  1.104 &  .7797 \\ 
 .130 &  .7007 &  1.120 &  .7817 \\ 
 .135 &  .6766 &  1.135 &  .7836 \\ 
 .140 &  .6540 &  1.149 &  .7854 \\ 
 .145 &  .6328 &  1.162 &  .7872 \\ 
 .150 &  .6128 &  1.174 &  .7889 \\ 
 .155 &  .5939 &  1.185 &  .7906 \\ 
 .160 &  .5759 &  1.195 &  .7923 \\ 
 .165 &  .5589 &  1.204 &  .7939 \\ 
 .170 &  .5427 &  1.213 &  .7955 \\ 
 .175 &  .5272 &  1.221 &  .7971 \\ 
 .180 &  .5124 &  1.229 &  .7986 \\ 
 .185 &  .4982 &  1.236 &  .8001 \\ 
 .190 &  .4846 &  1.242 &  .8016 \\ 
 .195 &  .4715 &  1.249 &  .8030 \\ 
 .200 &  .4590 &  1.254 &  .8045 \\ 
 .205 &  .4469 &  1.259 &  .8059 \\ 
 \end{tabular}
 \qquad
 \begin{tabular}{l|c|c|c} 
 $\lambda_1 \leq$ & $\lambda' \geq$& $\lambda$ & $J$ \\ 
 \hline 
 .210 &  .4353 &  1.264 &  .8073 \\ 
 .215 &  .4241 &  1.269 &  .8087 \\ 
 .220 &  .4132 &  1.273 &  .8100 \\ 
 .225 &  .4027 &  1.276 &  .8114 \\ 
 .230 &  .3926 &  1.280 &  .8127 \\ 
 .235 &  .3828 &  1.283 &  .8140 \\ 
 .240 &  .3733 &  1.285 &  .8153 \\ 
 .245 &  .3641 &  1.288 &  .8166 \\ 
 .250 &  .3552 &  1.290 &  .8179 \\ 
 .255 &  .3465 &  1.292 &  .8191 \\ 
 .260 &  .3381 &  1.294 &  .8204 \\ 
 .265 &  .3300 &  1.295 &  .8216 \\ 
 .270 &  .3220 &  1.296 &  .8229 \\ 
 .275 &  .3143 &  1.297 &  .8241 \\ 
 .280 &  .3068 &  1.298 &  .8253 \\ 
 .285 &  .2995 &  1.299 &  .8265 \\ 
 .290 &  .2924 &  1.299 &  .8277 \\ 
 .2909 &  .2911 &  1.299 &  .8279 \\ 
\end{tabular}

\caption{Bounds for $\lambda'$ with $\chi_1$ or $\rho_1$ complex and $\chi_1$ non-principal}
\label{Table-Lp-CC_Chi1NonPrincipal}
\end{table}

\subsubsection{$\chi_1$ principal}
\begin{lem} \label{CC-L1LpIdentity-0-real} 
Assume $\chi_1$ is principal, $\rho_1$ is complex, and $\rho'$ is real. 
Suppose $f$ satisfies Conditions 1 and 2. For $\epsilon > 0$, provided $\cL$ is sufficiently large depending on $\epsilon$ and $f$,  the following holds: 
\[
0 \leq 2F(-\lambda') -  2F(\lambda_1-\lambda')  - F(0) + 2\phi f(0)+\epsilon.
\]
\end{lem}
\begin{proof} This is analogous to \cref{SZ-L1LpIdentity}. To be brief, use \eqref{TrigIdentity-0} with $(\chi,\gamma) = (\chi_0, \gamma_1)$ and $\sigma = \beta'$ and apply \cref{ExplicitNP-Apply} extracting $\{\rho', \rho_1, \bar{\rho_1}\}$ from $\cK(\beta',\chi_0)$ and $\{\rho', \rho_1\}$ from $\cK(\beta'+i\gamma_1,\chi_0)$.
\end{proof}

\begin{table}
\begin{tabular}{l|c|c} 
$\lambda_1 \leq$ & $\lambda' \geq$ & $\lambda$ \\ 
 \hline 
 .0875 &  1.836 &  1.189 \\ 
 .09 &  1.803 &  1.185 \\ 
 .10 &  1.681 &  1.170 \\ 
 .11 &  1.572 &  1.156 \\ 
 .12 &  1.472 &  1.142 \\ 
 .13 &  1.381 &  1.129 \\ 
 .14 &  1.297 &  1.116 \\ 
 .15 &  1.220 &  1.103 \\ 
 .16 &  1.148 &  1.091 \\ 
 .17 &  1.080 &  1.079 \\ 
 .18 &  1.017 &  1.068 \\ 
 .19 &  .9578 &  1.056 \\ 
 .20 &  .9020 &  1.045 \\ 
 .21 &  .8493 &  1.034 \\ 
 \end{tabular}
 \qquad
 \begin{tabular}{l|c|c} 
$\lambda_1 \leq$ & $\lambda' \geq$ & $\lambda$ \\ 
 \hline
 .22 &  .7994 &  1.023 \\ 
 .23 &  .7522 &  1.013 \\ 
 .24 &  .7073 &  1.002 \\ 
 .25 &  .6646 &  .9917 \\ 
 .26 &  .6239 &  .9813 \\ 
 .27 &  .5851 &  .9711 \\ 
 .28 &  .5480 &  .9609 \\ 
 .29 &  .5126 &  .9508 \\ 
 .30 &  .4787 &  .9407 \\ 
 .31 &  .4462 &  .9307 \\ 
 .32 &  .4150 &  .9208 \\ 
 .33 &  .3851 &  .9108 \\ 
 .34 &  .3565 &  .9009 \\ 
 .3443 &  .3445 &  .8966 \\ 
\end{tabular}

\caption{Bounds for $\lambda'$ with $\chi_1$ principal, $\rho_1$ complex and $\rho'$ real}
\label{Table-Lp-CC_Chi1Principal_RhopReal}
\end{table}

A numerical version of \cref{CC-L1LpIdentity-0-real} yields bounds for $\lambda'$  with $f = f_{\lambda}$ taken from \cite[Lemma 7.2]{HBLinnik}, producing \cref{Table-Lp-CC_Chi1Principal_RhopReal}. The remaining case consists of $\chi_1$ principal  with both $\rho_1$ and $\rho'$ complex.

\begin{lem} \label{CC-L1LpIdentity_Principal} Assume $\chi_1$ is principal, $\rho_1$ is complex and $\rho'$ is complex. Let  $\lambda > 0$ and $J  > 0$. If $\cL$ is sufficiently large depending on $\epsilon, \lambda$ and $J$ then
\[
0 \leq  (J^2+\tfrac{1}{2}) P_4(1) -  (J^2+\tfrac{1}{2}) \cdot P_4\big( \frac{\lambda}{\lambda+\lambda_1} \big) - 2J \cdot P_4\big( \frac{\lambda}{\lambda+\lambda'}\big)  + 2(J+1)^2 \phi \lambda + \epsilon
\]
provided both of the following hold:
\begin{equation}
\begin{aligned}
	\frac{1}{(\lambda+\lambda_1)^4} + \frac{J_0}{(\lambda+\lambda')^4} > \frac{1}{\lambda^4} \qquad \text{ and } 
	\qquad	\frac{2}{(\lambda+\lambda_1)^4} + \frac{J_1}{(\lambda+\lambda')^4} > \frac{1}{\lambda^4}, \\
\end{aligned}
 \label{CC-L1LpIdentity_Principal-Condition}
\end{equation}
where $J_0 = \min\{ J+\tfrac{3}{4J}, 4J \}$ and $J_1 = 4J/(J^2+1)$.  
\end{lem}
\begin{proof} Analogous to \cref{CC-L1LpIdentity} but we exchange the roles of $\rho_1$ and $\rho'$ using the trigonometric identity
\[
0 \leq  \chi_0(\kn) \big(1 + \Re\{ (\N\kn)^{i\gamma_1} \} \big) \big( J+ \Re\{ (\N\kn)^{i\gamma'} \} \big)^2.
\]
Writing  $\cP(s) = \cP(s, \chi_0 ; P_4)$, it follows in the usual way that
\begin{equation}
\begin{aligned}
0 & \leq(J^2+\tfrac{1}{2}) \big\{ \cP(\sigma)  + \cP(\sigma+i\gamma_1) \big\}  + J \cP(\sigma + i(\gamma'+\gamma_1)) + 2J \cP(\sigma+i\gamma') + J \cP(\sigma+i(\gamma'-\gamma_1)) \\
& \qquad + \tfrac{1}{4} \cP(\sigma + i(2\gamma'+\gamma_1)) + \tfrac{1}{2} \cP(\sigma+2i\gamma') + \tfrac{1}{4} \cP(\sigma+i(2\gamma'-\gamma_1)) \\
\end{aligned} 
\label{TrigIdentity-CC-L1Lp_Principal}
\end{equation}
where $\sigma = 1 + \tfrac{\lambda}{\cL}$. Next, apply \cref{PolyEI-Apply} to each term according to the following outline:
\begin{itemize}
	\item $\cP(\sigma)$ and $\cP(\sigma+i\gamma')$ extract all 4 zeros $\{\rho_1, \bar{\rho_1}, \rho', \bar{\rho'} \}$.
	\item $\cP(\sigma+i\gamma_1)$ and $\cP(\sigma + i(\gamma' + \gamma_1)) $ extract only $\{\rho_1, \rho', \bar{\rho'} \}$.
	\item $\cP(\sigma + i(\gamma' - \gamma_1))$ extract only $\{ \bar{\rho_1}, \rho', \bar{\rho'} \}$.
	\item $ \cP(\sigma + i(2\gamma'+\gamma_1))$ and $\cP(\sigma + i(2\gamma'-\gamma_1))$  extract $\{\rho_1, \rho'\}$ and $\{\bar{\rho_1}, \rho'\}$ respectively.
	\item $\cP(\sigma + 2i\gamma')$  extract only $\{ \rho_1, \bar{\rho_1}, \rho'\}$. 
\end{itemize}
When necessary, we utilize that $P_4(\bar{X}) = \bar{P_4(X)}$. Then overall we obtain:
\begin{equation}
\begin{aligned}
0 \leq & (J^2 + \tfrac{1}{2})  P_4(1) -  (J^2+\tfrac{1}{2}) P_4\big( \frac{\lambda}{\lambda+\lambda_1} \big) - 2J \cdot P_4\big( \frac{\lambda}{\lambda+\lambda'}\big) + 2\phi (J+1)^2 \lambda + \epsilon \\
& \quad - \sum_{r=1}^{7} \Re\Big\{  A_r \cdot P_4\big( \frac{\lambda}{\lambda+\lambda_1+it_r} \big) + B_r \cdot P_4\big( \frac{\lambda}{\lambda+\lambda' + it_r}\big) - C_r \cdot P_4\big( \frac{\lambda}{\lambda+it_r}\big)\Big\} 
\end{aligned}
\label{CC-L1Lp-Principal_LastStep}
\end{equation}
where
\[
\begin{array}{c|c|c|c|c|c|c|c|c|c|c|c|c|c|}
r & 1 & 2 & 3 & 4 & 5 & 6 & 7 \\
\hline
t_r & \mu_1 & \mu' &  \mu' + \mu_1 & \mu'-\mu_1 & 2\mu'  & 2\mu' + \mu_1 &  2\mu' - \mu_1 \\
\hline
A_r & 2J^2 + 1 & 2J & 2J & 2J &  1/2 & 1/2 &  1/2 \\
\hline
B_r & 2J & 2J^2+3/2 & J^2 + 3/4 & J^2+3/4 & 2J & J & J
\\
\hline 
C_r &  J^2 + 1/2 & 2J & J & J & 1/2 & 1/4 & 1/4
\\
\hline 
\end{array}
\]
It suffices to show the sum over $r$ in \eqref{CC-L1Lp-Principal_LastStep} is non-negative. By \cref{PmLemma}, the sum is $\geq 0$ provided
\[
\frac{A_r}{(\lambda+\lambda_1)^4} + \frac{B_r}{(\lambda+\lambda')^4} > \frac{C_r}{ \lambda^4} \qquad \text{for $r=1,2,\dots,7.$}
\]
After inspection, the most stringent conditions are $r=1, 2$ and $5$, which are implied by assumption \eqref{CC-L1LpIdentity_Principal-Condition}. 
\end{proof}
This produces \cref{Table-Lp-CC_Chi1Principal_RhopComplex} in the usual fashion.

\begin{table}
\begin{tabular}{l|c|c|c} 
 $\lambda_1 \leq$ & $\lambda' \geq$ & $\lambda$ & $J$ \\ 
 \hline 
 .0875 &  .5330 &  1.155 &  .8815 \\ 
 .090 &  .5278 &  1.161 &  .8804 \\ 
 .095 &  .5179 &  1.171 &  .8782 \\ 
 .100 &  .5083 &  1.181 &  .8762 \\ 
 .105 &  .4991 &  1.190 &  .8742 \\ 
 .110 &  .4902 &  1.198 &  .8723 \\ 
 .115 &  .4817 &  1.206 &  .8704 \\ 
 .120 &  .4734 &  1.213 &  .8686 \\ 
 .125 &  .4654 &  1.220 &  .8669 \\ 
 .130 &  .4577 &  1.226 &  .8652 \\ 
 .135 &  .4502 &  1.232 &  .8636 \\ 
 .140 &  .4429 &  1.238 &  .8620 \\ 
 .145 &  .4359 &  1.243 &  .8605 \\ 
 .150 &  .4290 &  1.248 &  .8590 \\ 
 .155 &  .4223 &  1.252 &  .8576 \\ 
 .160 &  .4159 &  1.257 &  .8562 \\ 
 .165 &  .4096 &  1.261 &  .8548 \\ 
 .170 &  .4034 &  1.265 &  .8534 \\ 
 .175 &  .3974 &  1.268 &  .8521 \\ 
 .180 &  .3916 &  1.271 &  .8509 \\ 
 .185 &  .3859 &  1.274 &  .8496 \\ 
  .190 &  .3804 &  1.277 &  .8484 \\ 
 \end{tabular}
 \qquad
 \begin{tabular}{l|c|c|c} 
 $\lambda_1 \leq$ & $\lambda' \geq$ & $\lambda$ & $J$ \\ 
 \hline 
 .195 &  .3749 &  1.280 &  .8472 \\ 
 .200 &  .3696 &  1.282 &  .8460 \\ 
 .205 &  .3645 &  1.284 &  .8449 \\ 
 .210 &  .3594 &  1.286 &  .8437 \\ 
 .215 &  .3545 &  1.288 &  .8426 \\ 
 .220 &  .3497 &  1.290 &  .8415 \\ 
 .225 &  .3449 &  1.291 &  .8405 \\ 
 .230 &  .3403 &  1.293 &  .8394 \\ 
 .235 &  .3358 &  1.294 &  .8384 \\ 
 .240 &  .3314 &  1.295 &  .8374 \\ 
 .245 &  .3270 &  1.296 &  .8364 \\ 
 .250 &  .3228 &  1.297 &  .8354 \\ 
 .255 &  .3186 &  1.297 &  .8344 \\ 
 .260 &  .3145 &  1.298 &  .8335 \\ 
 .265 &  .3106 &  1.298 &  .8326 \\ 
 .270 &  .3066 &  1.299 &  .8317 \\ 
 .275 &  .3028 &  1.299 &  .8308 \\ 
 .280 &  .2990 &  1.299 &  .8299 \\ 
 .285 &  .2953 &  1.299 &  .8290 \\ 
 .290 &  .2917 &  1.299 &  .8281 \\ 
 .2909 &  .2911 &  1.299 &  .8280 \\ 
\end{tabular}
\caption{Bounds for $\lambda'$ with $\chi_1$ principal, $\rho_1$ complex and $\rho'$ complex}
\label{Table-Lp-CC_Chi1Principal_RhopComplex}
\end{table}

\subsubsection{Summary of bounds} We collect the results in the subsection into a single proposition for the reader's convenience. 

\begin{prop} \label{Bounds-Lp-CC}
Assume $\chi_1$ or $\rho_1$ is complex. Provided $\cL$ is sufficiently large, we have the following:
\begin{enumerate}[(a)]
	\item If $\chi_1$ is non-principal then $\lambda' \geq 0.2909$ and the bounds for $\lambda'$ in \cref{Table-Lp-CC_Chi1NonPrincipal} apply.
	\item If $\chi_1$ is principal then  $\lambda' \geq 0.2909$ and the bounds for $\lambda'$ in \cref{Table-Lp-CC_Chi1Principal_RhopComplex} apply.
\end{enumerate}
\end{prop}
\begin{proof} If $\chi_1$ is non-principal, then the only bounds available come from \cref{Table-Lp-CC_Chi1NonPrincipal}. If $\chi_1$ is principal, then upon comparing \cref{Table-Lp-CC_Chi1Principal_RhopReal,Table-Lp-CC_Chi1Principal_RhopComplex}, one finds that the latter gives weaker bounds.
\end{proof}

\subsection{Bounds for $\lambda_2$}
Before dividing into cases, we begin with the following lemma analogous to \cref{CC-L1LpIdentity}. 
\begin{lem} \label{CC-L1L2Identity} Assume $\chi_1$ or $\rho_1$ is complex. Suppose $f$ satisfies Conditions 1 and 2. For $\epsilon > 0$, provided $\cL$ is sufficiently large depending on $\epsilon$ and $f$,  the following holds: 

\begin{enumerate}[(a)]
	\item If $\chi_1,\chi_2$ are non-principal, then
\[
0 \leq F(-\lambda_1) - F(0) - F(\lambda_2-\lambda_1) + 4\phi f(0)+\epsilon.
\]
	\item If $\chi_1$ is principal, then $\rho_1$ is complex, $\chi_2$ is non-principal and  
	\begin{align*}
0 \leq & \quad F(-\lambda_1)  - F(0) - F(\lambda_2-\lambda_1) + \Re\{ F(-\lambda_1 + i\mu_1) - F(i\mu_1) - F(\lambda_2-\lambda_1+i\mu_1)  \} + 4\phi f(0)+\epsilon.
\end{align*}
	\item If $\chi_2$ is principal, then $\chi_1$ is non-principal and
	\	\begin{align*}
0 \leq & \quad F(-\lambda_1)  - F(0) - F(\lambda_2-\lambda_1) + \Re\{ F(-\lambda_1 + i\mu_2) - F(i\mu_2) - F(\lambda_2-\lambda_1+i\mu_2)  \} + 4\phi f(0) +\epsilon.
\end{align*}
\end{enumerate}
\end{lem}
\begin{proof} The arguments involved are very similar to \cref{SZ-L1LpIdentity} and \cref{SZ-L1L2Identity} so we omit most of the details. Briefly, use \eqref{TrigIdentity} by setting $(\chi, \rho) = (\chi_1, \rho_1)$ and $(\chi_*, \rho_*)  = (\chi_2, \rho_2)$ and $\sigma = \beta_1$, which gives
\begin{equation}
\begin{aligned}
0 & \leq \cK(\beta_1, \chi_0)  + \cK(\beta_1+i\gamma_1, \chi_1) + \cK(\beta_1 + i\gamma_2, \chi_2)  \\
& \qquad + \tfrac{1}{2} \cK(\beta_1+ i(\gamma_1+\gamma_2), \chi_1\chi_2) + \tfrac{1}{2}\cK(\beta_1+ i(\gamma_1-\gamma_2), \chi_1\bar{\chi_2}).
\end{aligned}
\label{TI-CC-L2}
\end{equation}
Apply \cref{ExplicitNP-Apply} to each $\cK(\ast, \ast)$ term, extracting zeros $\rho_1$ or $\rho_2$ whenever possible, depending on the cases. Recall $\chi_1\chi_2$ and $\chi_1\bar{\chi_2}$ are always non-principal by construction (see \cref{sec:ZeroFreeGap_and_BadZeros}).  
\end{proof}
\subsubsection{$\chi_1$ and $\chi_2$ non-principal}

A numerical version of \cref{CC-L1L2Identity} suffices here. 

\begin{lem} \label{CC-L1L2small_Chi1Chi2NonPrincipal} Assume $\chi_1$ or $\rho_1$ is complex with $\chi_1, \chi_2$ non-principal. Let $\epsilon > 0$ and for $b > 0$, assume $0 < \lambda_1 \leq b$. Suppose, for some $\tilde{\lambda}_b > 0$, we have 
\[
\begin{array}{ll}
 F(-b) - F(0)- F(\tilde{\lambda}_b-b) +   4 \phi f(0) \leq 0 & 
\end{array}
\]
Then $\lambda_2 \geq \tilde{\lambda}_b - \epsilon$ provided $\cL$ is sufficiently large depending on $\epsilon$ and $f$.
\end{lem}
\begin{proof} Analogous to \cref{SZ-L1L2small} using \cref{CC-L1L2Identity} in place of \cref{SZ-L1L2Identity}. Hence, we omit the proof.  
\end{proof}

This produces \cref{Table-L2-CC_Chi1Chi2NonPrincipal} by taking $f = f_{\lambda}$ from \cite[Lemma 7.2]{HBLinnik} with parameter $\lambda = \lambda(b)$.  
\begin{table}
\begin{tabular}{l|c|c} 
 $\lambda_1 \leq$ & $\lambda_2 \geq$ & $\lambda$ \\ 
 \hline 
 .1227 &  .4890 &  .3837 \\ 
 .13 &  .4779 &  .3888 \\ 
 .135 &  .4706 &  .3922 \\ 
 .140 &  .4635 &  .3955 \\ 
 .145 &  .4566 &  .3986 \\ 
 .150 &  .4499 &  .4017 \\ 
 .155 &  .4433 &  .4047 \\ 
 .160 &  .4370 &  .4077 \\ 
 .165 &  .4308 &  .4105 \\ 
 .170 &  .4247 &  .4133 \\ 
 .175 &  .4188 &  .4160 \\ 
 .180 &  .4131 &  .4187 \\ 
 .185 &  .4075 &  .4213 \\ 
 .190 &  .4020 &  .4238 \\ 
 .195 &  .3966 &  .4263 \\ 
 .200 &  .3914 &  .4287 \\ 
 .205 &  .3862 &  .4311 \\ 
  .210 &  .3812 &  .4334 \\ 
   .215 &  .3763 &  .4357 \\ 
 \end{tabular}
 \qquad
 \begin{tabular}{l|c|c} 
 $\lambda_1 \leq$ & $\lambda_2 \geq$ & $\lambda$ \\ 
 \hline 
 .220 &  .3715 &  .4380 \\ 
 .225 &  .3668 &  .4402 \\ 
 .230 &  .3622 &  .4423 \\ 
 .235 &  .3576 &  .4444 \\ 
 .240 &  .3532 &  .4465 \\ 
 .245 &  .3488 &  .4486 \\ 
 .250 &  .3446 &  .4506 \\ 
 .255 &  .3404 &  .4526 \\ 
 .260 &  .3363 &  .4545 \\ 
 .265 &  .3322 &  .4564 \\ 
 .270 &  .3283 &  .4583 \\ 
 .275 &  .3244 &  .4602 \\ 
 .280 &  .3205 &  .4620 \\ 
 .285 &  .3168 &  .4638 \\ 
 .290 &  .3131 &  .4656 \\ 
 .295 &  .3094 &  .4673 \\ 
 .300 &  .3059 &  .4690 \\ 
 .3034 &  .3035 &  .4702 \\ 
\end{tabular}
\caption{Bounds for $\lambda_2$ with $\chi_1$ or $\rho_1$ complex and $\chi_1, \chi_2$ non-principal}
\label{Table-L2-CC_Chi1Chi2NonPrincipal}
\end{table}

\subsubsection{$\chi_1$ principal or $\chi_2$ is principal}

When $\chi_2$ is principal and $\rho_2$ is real,  a numerical version of \cref{CC-L1L2Identity} suffices.

\begin{lem} \label{CC-L1L2small} Assume $\chi_1$ or $\rho_1$ is complex. Further assume $\chi_2$ is principal and $\rho_2$ is real. Let $\epsilon > 0$ and for $b > 0$, assume $0 < \lambda_1 \leq b$. Suppose, for some $\tilde{\lambda}_b > 0$, we have 
\[
\begin{array}{ll}
F(-b) - F(0)- F(\tilde{\lambda}_b-b) +   2\phi f(0) \leq 0 & 
\end{array}
\]
Then $\lambda_2 \geq \tilde{\lambda}_b - \epsilon$ provided $\cL$ is sufficiently large depending on $\epsilon$ and $f$.
\end{lem}
This produces \cref{Table-L2-CC_Chi2Principal_Rho2Real} by taking $f$ from \cite[Lemma 7.2]{HBLinnik} with parameter $\lambda = \lambda(b)$. Now, when $\chi_1$ is principal or when $\chi_2$ is principal and $\rho_2$ is complex, we employ the ``polynomial method".   

\begin{table}
\begin{tabular}{l|c|c} 
 $\lambda_1 \leq$ & $\lambda_2 \geq$ & $\lambda$ \\ 
 \hline 
 .1227 &  1.221 &  .6530 \\ 
 .13 &  1.203 &  .6620 \\ 
 .15 &  1.155 &  .6846 \\ 
 .17 &  1.112 &  .7049 \\ 
 .19 &  1.073 &  .7234 \\ 
 .21 &  1.037 &  .7403 \\ 
 .23 &  1.003 &  .7560 \\ 
 .25 &  .9710 &  .7707 \\ 
 .27 &  .9412 &  .7844 \\ 
 .29 &  .9132 &  .7973 \\ 
 .31 &  .8867 &  .8095 \\ 
 .33 &  .8615 &  .8210 \\ 
 .35 &  .8377 &  .8320 \\ 
 \end{tabular}
 \qquad
 \begin{tabular}{l|c|c} 
 $\lambda_1 \leq$ & $\lambda_2 \geq$ & $\lambda$ \\ 
 \hline 
 .37 &  .8149 &  .8425 \\ 
 .39 &  .7932 &  .8526 \\ 
 .41 &  .7725 &  .8622 \\ 
 .43 &  .7526 &  .8714 \\ 
 .45 &  .7335 &  .8803 \\ 
 .47 &  .7152 &  .8889 \\ 
 .49 &  .6977 &  .8971 \\ 
 .51 &  .6807 &  .9051 \\ 
 .53 &  .6644 &  .9128 \\ 
 .55 &  .6487 &  .9203 \\ 
 .57 &  .6336 &  .9276 \\ 
 .59 &  .6189 &  .9346 \\ 
 .6068 &  .6070 &  .9404 \\ 
\end{tabular}

\caption{Bounds for $\lambda_2$ with $\chi_1$ or $\rho_1$ complex and $\chi_2$ principal and $\rho_2$ real}
\label{Table-L2-CC_Chi2Principal_Rho2Real}
\end{table}

\begin{lem} \label{CC-L1L2Identity-0} Suppose $\chi_j$ is principal and $\rho_j$ is complex, and let $\chi_k \neq \chi_j$. Let  $\epsilon, \lambda, J  > 0$.  If $\cL$ is sufficiently large depending on $\epsilon, \lambda$ and $J$, then 
\begin{align*}
0 & \leq (J^2 + \tfrac{1}{2}) \big\{ P_4(1) - P_4\big( \frac{\lambda}{\lambda+\lambda_k} \big) \big\} - 2J P_4\big( \frac{\lambda}{\lambda+\lambda_j} \big) +  2\phi (J+1)^2 \lambda    + \epsilon
\end{align*}
provided 
\begin{equation}
\frac{J_0}{(\lambda+\lambda_j)^4} + \frac{1}{(\lambda+\lambda_k)^4} > \frac{1}{ \lambda^4} \qquad \text{with $J_0 = \min\{ J + \tfrac{3}{4J}, 4J \}$.}
 \label{CC-L1L2-Principal-Condition}
\end{equation}
\end{lem}

\begin{proof} Write $\cP(s,\chi) = \cP(s, \chi; P_4)$. We begin with the inequality
\begin{align*}
0 & \leq \chi_0(\kn) \big(1+\Re\{ \chi_k(\kn) (\N\kn)^{-i\gamma_k}\} \big) \big( J + \Re\{ (\N\kn)^{-i\gamma_j} \} \big)^2 
\end{align*}
It follows in the usual fashion that
\begin{equation}
\begin{aligned}
0 & \leq(J^2+\tfrac{1}{2}) \big\{ \cP(\sigma, \chi_0)  + \cP(\sigma+i\gamma_k, \chi_k) \big\} + \\
& \qquad + J \cP(\sigma + i(\gamma_j+\gamma_k), \chi_k) + 2J \cP(\sigma+i\gamma_j, \chi_0) + J \cP(\sigma+i(\gamma_j-\gamma_k), \bar{\chi_k}) \\
& \qquad + \tfrac{1}{4} \cP(\sigma + i(2\gamma_j+\gamma_k), \chi_k) + \tfrac{1}{2} \cP(\sigma+2i\gamma_j, \chi_0) + \tfrac{1}{4} \cP(\sigma+i(2\gamma_j-\gamma_k), \bar{\chi_k}) \\
\end{aligned} 
\label{TrigIdentity-CC-L1L2}
\end{equation}
where $\sigma = 1 + \tfrac{\lambda}{\cL}$. Next,apply \cref{PolyEI-Apply} to each $\cP(\ast, \ast)$ term in \eqref{TrigIdentity-CC-L1L2} extracting $\{\rho_j, \bar{\rho_j}\}$  from $\chi_0$-terms, $\rho_k$ from the $\chi_k$-terms, and $\bar{\rho_k}$ from $\bar{\chi_k}$-terms. When necessary, we also use that $P_4(\bar{X}) = \bar{P_4(X)}$. Then  overall
\begin{equation}
\begin{aligned}
0 \leq & (J^2+\tfrac{1}{2}) P_4(1) -  (J^2+\tfrac{1}{2}) P_4\big( \frac{\lambda}{\lambda+\lambda_k} \big) - 2J P_4\big( \frac{\lambda}{\lambda+\lambda_j}\big) +  \psi  \lambda + \epsilon - A - B
\end{aligned}
\label{CC-L1L2-Principal}
\end{equation}
where 
\begin{align*}
A & = \Re\Big\{  (2J^2+\tfrac{3}{2}) P_4\big( \frac{\lambda}{\lambda+\lambda_j+i\mu_j} \big) + 2J \cdot P_4\big( \frac{\lambda}{\lambda+\lambda_k + i\mu_j}\big) - 2J \cdot P_4\big( \frac{\lambda}{\lambda+i\mu_j}\big)\Big\},  \\
B & = \Re\Big\{  2J \cdot P_4\big( \frac{\lambda}{\lambda+\lambda_j+2i\mu_j} \big) + \tfrac{1}{2} \cdot P_4\big( \frac{\lambda}{\lambda+\lambda_k + 2i\mu_j}\big) - \tfrac{1}{2} \cdot P_4\big( \frac{\lambda}{\lambda+2i\mu_j}\big)\Big\},
\end{align*}
and $\psi = (J^2 + 2J+1) \phi \tfrac{\cL_{\chi_k}}{\cL} + (J^2 + 2J + 1) \phi \tfrac{\cL_0}{\cL}.$ Trivially $J^2 + 2J+1 \leq 3 \cdot(J^2 + 2J+1)$ and so by \cref{QuantityRelations}, we may  more simply take $\psi =  2(J+1)^2 \phi$ in \eqref{CC-L1L2-Principal}. From \cref{PmLemma} and \eqref{CC-L1L2-Principal-Condition}, it follows $A,B \geq 0$.
\end{proof}

We record a numerical version of \cref{CC-L1L2Identity-0} without proof. 
\begin{cor} \label{CC-L1L2-Principmedium-Numerical} Suppose $\chi_1$ or $\rho_1$ is complex. For $b > 0$, assume $0 < \lambda_1 \leq b$ and let $\lambda, J > 0$. Denote $J_0 := \min\{ J + \tfrac{3}{4J}, 4J \}$. Assume one of the following holds: 

\begin{enumerate}[(a)]
	\item  $\chi_1$ is principal, $\rho_1$ is complex. Further there exists $\tilde{\lambda_b}\in [0,\infty)$ satisfying
\begin{align*}
0 & = (J^2 + \tfrac{1}{2}) \big( P_4(1) - P_4\big( \frac{\lambda}{\lambda+ \tilde{\lambda_b} } \big) \big) - 2J \cdot P_4\big( \frac{\lambda}{\lambda+b} \big)  + 2\phi (J+1)^2 \lambda + \epsilon.
\end{align*}
and
\[
\frac{J_0}{(\lambda+b)^4} + \frac{1}{(\lambda+\tilde{\lambda_b})^4} > \frac{1}{ \lambda^4}.
\]
	\item $\chi_2$ is principal, $\rho_2$ is complex. Further there exists $\tilde{\lambda_b}\in [0,\infty)$ satisfying
	\begin{align*}
0 & = (J^2 + \tfrac{1}{2}) \big( P_4(1) - P_4\big( \frac{\lambda}{\lambda+ b} \big) \big) - 2J \cdot P_4\big( \frac{\lambda}{\lambda+\tilde{\lambda_b}} \big)  + 2\phi (J+1)^2 \lambda + \epsilon
\end{align*}
and
\[
 \frac{1}{(\lambda+b)^4} + \frac{J_0}{(\lambda+\tilde{\lambda_b})^4} > \frac{1}{ \lambda^4}.
\]
\end{enumerate}
Then, in either case, $\lambda_2 \geq \tilde{\lambda_b} - \epsilon$ for $\cL$ sufficiently large depending on $\epsilon, b, \lambda$ and $J$. 
\end{cor}

This produces \cref{Table-L2-CC_Chi1Principal,Table-L2-CC_Chi2Principal_Rho2Complex}.

\begin{table}
\begin{tabular}{l|c|c|c} 
 $\lambda_1 \leq$ & $\lambda_2 \geq$ & $\lambda$ & $J$ \\ 
 \hline 
 .0875 &  1.017 &  .9321 &  .7627 \\ 
 .090 &  .9892 &  .9474 &  .7640 \\ 
 .095 &  .9385 &  .9760 &  .7666 \\ 
 .100 &  .8937 &  1.002 &  .7690 \\ 
 .105 &  .8537 &  1.026 &  .7713 \\ 
 .110 &  .8175 &  1.048 &  .7735 \\ 
 .115 &  .7846 &  1.069 &  .7757 \\ 
 .120 &  .7544 &  1.087 &  .7777 \\ 
 .125 &  .7266 &  1.104 &  .7797 \\ 
 .130 &  .7007 &  1.120 &  .7817 \\ 
 .135 &  .6766 &  1.135 &  .7836 \\ 
 .140 &  .6540 &  1.149 &  .7854 \\ 
 .145 &  .6328 &  1.162 &  .7872 \\ 
 .150 &  .6128 &  1.174 &  .7889 \\ 
 .155 &  .5939 &  1.185 &  .7906 \\ 
 .160 &  .5759 &  1.195 &  .7923 \\ 
 .165 &  .5589 &  1.204 &  .7939 \\ 
 .170 &  .5427 &  1.213 &  .7955 \\ 
 .175 &  .5272 &  1.221 &  .7971 \\ 
 .180 &  .5124 &  1.229 &  .7986 \\ 
 .185 &  .4982 &  1.236 &  .8001 \\ 
  .190 &  .4846 &  1.242 &  .8016 \\ 
 \end{tabular}
 \qquad
 \begin{tabular}{l|c|c|c} 
 $\lambda_1 \leq$ & $\lambda_2 \geq$ & $\lambda$ & $J$ \\ 
 \hline 
 .195 &  .4715 &  1.249 &  .8030 \\ 
 .200 &  .4590 &  1.254 &  .8045 \\ 
 .205 &  .4469 &  1.259 &  .8059 \\ 
 .210 &  .4353 &  1.264 &  .8073 \\ 
 .215 &  .4241 &  1.269 &  .8087 \\ 
 .220 &  .4132 &  1.273 &  .8100 \\ 
 .225 &  .4027 &  1.276 &  .8114 \\ 
 .230 &  .3926 &  1.280 &  .8127 \\ 
 .235 &  .3828 &  1.283 &  .8140 \\ 
 .240 &  .3733 &  1.285 &  .8153 \\ 
 .245 &  .3641 &  1.288 &  .8166 \\ 
 .250 &  .3552 &  1.290 &  .8179 \\ 
 .255 &  .3465 &  1.292 &  .8191 \\ 
 .260 &  .3381 &  1.294 &  .8204 \\ 
 .265 &  .3300 &  1.295 &  .8216 \\ 
 .270 &  .3220 &  1.296 &  .8229 \\ 
 .275 &  .3143 &  1.297 &  .8241 \\ 
 .280 &  .3068 &  1.298 &  .8253 \\ 
 .285 &  .2995 &  1.299 &  .8265 \\ 
 .290 &  .2924 &  1.299 &  .8277 \\ 
 .2909 &  .2911 &  1.299 &  .8279 \\ 
\end{tabular}
\caption{Bounds for $\lambda_2$ with $\chi_1$ principal and $\rho_1$ complex}
\label{Table-L2-CC_Chi1Principal}
\end{table}

\begin{table}
\begin{tabular}{l|c|c|c} 
 $\lambda_1 \leq$ & $\lambda_2 \geq$ & $\lambda$ & $J$ \\ 
 \hline 
 .1227 &  .4691 &  1.217 &  .8677 \\ 
 .125 &  .4654 &  1.220 &  .8669 \\ 
 .130 &  .4577 &  1.226 &  .8652 \\ 
 .135 &  .4502 &  1.232 &  .8636 \\ 
 .140 &  .4429 &  1.238 &  .8620 \\ 
 .145 &  .4359 &  1.243 &  .8605 \\ 
 .150 &  .4290 &  1.248 &  .8590 \\ 
 .155 &  .4223 &  1.252 &  .8576 \\ 
 .160 &  .4159 &  1.257 &  .8562 \\ 
 .165 &  .4096 &  1.261 &  .8548 \\ 
 .170 &  .4034 &  1.265 &  .8534 \\ 
 .175 &  .3974 &  1.268 &  .8521 \\ 
 .180 &  .3916 &  1.271 &  .8509 \\ 
 .185 &  .3859 &  1.274 &  .8496 \\ 
 .190 &  .3804 &  1.277 &  .8484 \\ 
 .195 &  .3749 &  1.280 &  .8472 \\ 
 .200 &  .3696 &  1.282 &  .8460 \\ 
 .205 &  .3645 &  1.284 &  .8449 \\ 
 .210 &  .3594 &  1.286 &  .8437 \\ 
 \end{tabular}
 \qquad
 \begin{tabular}{l|c|c|c} 
 $\lambda_1 \leq$ & $\lambda_2 \geq$ & $\lambda$ & $J$ \\ 
 \hline 
 .215 &  .3545 &  1.288 &  .8426 \\ 
 .220 &  .3497 &  1.290 &  .8415 \\ 
 .225 &  .3449 &  1.291 &  .8405 \\ 
 .230 &  .3403 &  1.293 &  .8394 \\ 
 .235 &  .3358 &  1.294 &  .8384 \\ 
 .240 &  .3314 &  1.295 &  .8374 \\ 
 .245 &  .3270 &  1.296 &  .8364 \\ 
 .250 &  .3228 &  1.297 &  .8354 \\ 
 .255 &  .3186 &  1.297 &  .8344 \\ 
 .260 &  .3145 &  1.298 &  .8335 \\ 
 .265 &  .3106 &  1.298 &  .8326 \\ 
 .270 &  .3066 &  1.299 &  .8317 \\ 
 .275 &  .3028 &  1.299 &  .8308 \\ 
 .280 &  .2990 &  1.299 &  .8299 \\ 
 .285 &  .2953 &  1.299 &  .8290 \\ 
 .290 &  .2917 &  1.299 &  .8281 \\ 
 .2909 &  .2911 &  1.299 &  .8280 \\ 
\end{tabular}
\caption{Bounds for $\lambda_2$ with  $\chi_1$ or $\rho_1$ complex and  $\chi_2$ principal and $\rho_2$ complex}
\label{Table-L2-CC_Chi2Principal_Rho2Complex}
\end{table}

\subsubsection{Summary of bounds} We collect the results in the subsection into a single proposition for the reader's convenience. 

\begin{prop} \label{Bounds-L2-CC}
Assume $\chi_1$ or $\rho_1$ is complex. Provided $\cL$ is sufficiently large, the following holds:

\begin{enumerate}[(a)]
	\item If $\chi_1$ is non-principal, then  $\lambda_2 \geq 0.2909$ and the bounds for $\lambda_2$ in \cref{Table-L2-CC_Chi2Principal_Rho2Complex} apply. 
	\item If $\chi_1$ is principal, then $\lambda_2 \geq 0.2909$ and the bounds for $\lambda_2$ in \cref{Table-L2-CC_Chi1Principal} apply. 
\end{enumerate}
\end{prop}
\begin{proof} If $\chi_1$ is non-principal then one compares \cref{Table-L2-CC_Chi1Chi2NonPrincipal}, \cref{Table-L2-CC_Chi2Principal_Rho2Real} and \cref{Table-L2-CC_Chi2Principal_Rho2Complex} and finds that the last one gives the weakest bounds. If $\chi_1$ is principal, then the only bounds available come from \cref{Table-L2-CC_Chi1Principal}. 
\end{proof}

\section{Zero-Free Region}
\label{sec:ZeroFreeRegion}
\noindent
\emph{Proof of \cref{MainTheorem-ZFR}}:   If $\chi_1$ and $\rho_1$ are both real, then \cref{MainTheorem-ZFR} is implied by \cref{SZ-L1Lp-Summary,SZ-L1L2-Summary}. 
Thus, it remains to consider when $\chi_1$ or $\rho_1$ is complex, dividing our cases according to the order of $\chi_1$.
\noindent
\addtocontents{toc}{\SkipTocEntry}
\subsection*{$\chi_1$ has order $\geq 5$}
We begin with the inequality 
\begin{equation}
0  \leq \chi_0(\kn) \big(3+10 \cdot\Re\{ \chi_1(\kn) (\N\kn)^{-i\gamma_1}\} \big)^2 \big( 9 + 10 \cdot \Re\{\chi_1(\kn) (\N\kn)^{-i\gamma_1} \} \big)^2
\label{ZFR-TrigIdentity-NonPrincipal}
\end{equation}
which was also used in \cite[Section 9]{HBLinnik}. This will also be roughly optimal for our purposes. We shall use the smoothed explicit inequality with a weight $f$ satisfying Conditions 1 and 2. By the usual arguments, we expand out the above identity, multiply by the appropriate factor and sum over $\kn$. Overall this yields
\begin{equation}
\begin{aligned}
0 & \leq 14379 \cdot \cK(\sigma, \chi_0)  +  24480  \cdot \cK(\sigma+i\gamma_1, \chi_1) +  14900 \cdot \cK(\sigma+2i\gamma_1, \chi_1^2) \\
& \qquad +  6000 \cdot \cK(\sigma + 3i\gamma_1, \chi_1^3) + 1250 \cdot \cK(\sigma + 4i\gamma_1, \chi_1^4) \\
\end{aligned} 
\label{TrigIdentity-ZFR-Smooth}
\end{equation}
where $\cK(s,\chi) = \cK(s,\chi; f)$ and $\sigma = 1-\frac{\lambda^{\star}}{\cL}$ with constant $\lambda^{\star}$ satisfying
\[
\lambda_1 \leq \lambda^{\star} \leq \min\{ \lambda', \lambda_2\}. 
\]
Now, apply \cref{ExplicitNP-Apply} to each term in \eqref{TrigIdentity-ZFR-Smooth} and consider cases depending on $\ord \chi_1$. For $\cK(\sigma+ni\gamma_1,\chi_1^n)$:

\begin{itemize}
	\item  ($\ord \chi_1 \geq 6)$ Extract $\{\rho_1\}$ if $n=1$ only. 
	\item $(\ord \chi_1 = 5)$ Set $\lambda^{\star} = \lambda_1$ and extract $\{\rho_1\}$ if $n=1$ only. 
\end{itemize}
It follows that
\begin{equation} 
0 \leq 14379 \cdot F(-\lambda^{\star})  - 24480 \cdot F\big( \lambda_1 - \lambda^{\star} \big) + B f(0) \phi + \epsilon
\label{ZFR-Weil-Identity}
\end{equation}
where $B = 14379 \cdot \tfrac{\cL_0}{\cL} + 46630 \cdot \tfrac{\cL_{\chi_1}}{\cL}.$ From \cref{QuantityRelations}, $B \leq 57516 + 3493 \tfrac{\cL_{\chi_1}}{\cL} \leq 62174$ so \eqref{ZFR-Weil-Identity} reduces to
\begin{equation} 
0 \leq 14379 \cdot F(-\lambda^{\star})  - 24480 \cdot F\big( \lambda_1 - \lambda^{\star} \big) + 62174 \phi f(0) + \epsilon.
\label{ZFR-Weil-Identity2}
\end{equation}
We now consider cases. 
\begin{itemize}
	\item $(\ord \chi_1 \geq 6)$ Without loss, we may assume $\lambda_1 \leq 0.180$. From \cref{Bounds-Lp-CC,Bounds-L2-CC}, we may take $\lambda^{\star} = 0.3916$. Choose $f$ according to \cite[Lemma 7.1]{HBLinnik} with parameters $\theta = 1$ and $\lambda = 0.243$. Then \eqref{ZFR-Weil-Identity2} implies $\lambda_1 \geq 0.1764$. 

	\item $(\ord \chi_1 = 5)$ Since $\lambda^{\star} = \lambda_1$ in this case, \eqref{ZFR-Weil-Identity2} becomes
	\[
	0 \leq 14379 \cdot F(-\lambda_1)  - 24480 \cdot F(0)+ 62174 f(0) \phi + \epsilon.
	\]
	We choose $f$ according to \cite[Lemma 7.5]{HBLinnik} with $k = 24480/14379$ giving $\theta = 1.1580...$ and
	\[
	\lambda_1^{-1} \cos^2 \theta \leq \frac{1}{4} \cdot \frac{62174}{14379} + \epsilon
	\]
	whence $\lambda_1 \geq 0.1489$. 
\end{itemize}

\begin{rem*} To bound the quantity $B$ for $\ord \chi_1 \geq 5$, the proof above uses that $\vartheta \geq \tfrac{3}{4}$ leading to some minor loss in the lower bound for $\lambda_1$. If one uses $\vartheta = 1$, say, then one can slightly improve this lower bound. 
\end{rem*}

\noindent
\addtocontents{toc}{\SkipTocEntry}
\subsection*{$\chi_1$ has order $2, 3$ or $4$} We use the same identity \eqref{ZFR-TrigIdentity-NonPrincipal} but instead will apply the ``polynomial method" with $P_4(X)$. In the usual way, it follows from \eqref{ZFR-TrigIdentity-NonPrincipal} that
\begin{equation}
\begin{aligned}
0 & \leq 14379 \cdot \cP(\sigma, \chi_0)  +  24480  \cdot \cP(\sigma+i\gamma_1, \chi_1) +  14900 \cdot \cP(\sigma+2i\gamma_1, \chi_1^2) \\
& \qquad +  6000 \cdot \cP(\sigma + 3i\gamma_1, \chi_1^3) + 1250 \cdot \cP(\sigma + 4i\gamma_1, \chi_1^4) \\
\end{aligned} 
\label{TrigIdentity-ZFR-Poly}
\end{equation}
where $\sigma = 1 + \tfrac{\lambda}{\cL}$ with $\lambda > 0$. The above identity will be roughly optimal for our purposes. Now, we apply \cref{PolyEI-Apply} to each term above and consider cases depending on $\ord \chi_1$. For each term  $\cP(\sigma+ni\gamma_1,\chi_1^n)$:

\begin{itemize}
	\item $(\ord \chi_1 = 4)$ Extract $\{\rho_1\}$ if $n=1$ and $\{ \bar{\rho_1}\}$ if $n=3$. 	
	\item $(\ord \chi_1 = 3)$ Extract $\{\rho_1\}$ if $n=1$ or $4$ and $\{ \bar{\rho_1}\}$ if $n=2$.
	\item $(\ord \chi_1 = 2)$ Extract $\{\rho_1, \bar{\rho_1} \}$ if $n=1$ or $3$ since $\rho_1$ is necessarily complex.
\end{itemize}
It then follows that
\begin{equation} 
0 \leq 14379 \cdot P_4(1)  - 24480 \cdot P_4\big( \frac{\lambda}{\lambda+\lambda_1} \big) + A_{\chi_1} + B_{\chi_1}\phi \lambda + \epsilon
\label{ZFR-Poly-Identity}
\end{equation}
with
\[
A_{\chi_1} = \begin{cases}
\vspace*{2mm}
\ds \Re\{ 1250 \cdot P_4\big( \frac{\lambda}{\lambda+4i\mu_1} \big)  - 6000 \cdot P_4\big( \frac{\lambda}{\lambda+\lambda_1 + 4i\mu_1} \big) \}, & \ord \chi_1 = 4, \\
\vspace*{2mm}
\ds  \Re\{ 6000 \cdot P_4\big( \frac{\lambda}{\lambda+3i\mu_1} \big)   - 16150 \cdot P_4\big( \frac{\lambda}{\lambda+\lambda_1 + 3i\mu_1} \big)  \}, & \ord \chi_1 = 3, \\
\vspace*{2mm}
\ds \Re\{ 14900 \cdot P_4\big( \frac{\lambda}{\lambda+2i\mu_1} \big)  - 30480  \cdot P_4\big( \frac{\lambda}{\lambda+\lambda_1 + 2i\mu_1} \big)  \} & \ord \chi_1 = 2. \\ 
\vspace*{2mm}
\ds \qquad + \Re\{ 1250 \cdot P_4\big( \frac{\lambda}{\lambda+4i\mu_1} \big)   - 6000 \cdot P_4\big( \frac{\lambda}{\lambda+\lambda_1 + 4i\mu_1} \big)  \},
\end{cases}
\]
and
\[
B_{\chi_1} = \begin{cases}
\vspace*{2mm}
15629 \cdot \frac{\cL_0}{\cL} + 45380 \cdot \frac{\cL_{\chi_1}}{\cL} & \text{if } \ord \chi_1 = 4, \\
\vspace*{2mm}
20379 \cdot \frac{\cL_0}{\cL} + 40630 \cdot \frac{\cL_{\chi_1}}{\cL} & \text{if } \ord \chi_1 = 3, \\
\vspace*{2mm}
30529 \cdot \frac{\cL_0}{\cL} + 30480 \cdot \frac{\cL_{\chi_1}}{\cL} & \text{if } \ord \chi_1 = 2. 
\end{cases}
\]
By \cref{QuantityRelations}, we observe $B_{\chi_1} \leq 61009$. Furthermore, applying \cref{PmLemma} to  $A_{\chi_1}$, it follows that $A_{\chi_1} \leq 0$ in all cases provided
\begin{equation}
\frac{14900}{\lambda^4} -  \frac{30480}{(\lambda+\lambda_1)^4} \leq 0.
\label{ZFR-Poly-Condition} 
\end{equation}
Thus, \eqref{ZFR-Poly-Identity}  implies
\[
0 \leq 14379 \cdot P_4(1)  - 24480 \cdot P_4\big( \frac{\lambda}{\lambda+\lambda_1} \big) + 61009 \phi \lambda+ \epsilon
\]
provided \eqref{ZFR-Poly-Condition} holds. Taking $\lambda = 0.9421$ yields $\lambda_1 \geq 0.1227$.

\addtocontents{toc}{\SkipTocEntry}
\subsection*{$\chi_1$ is principal} Recall in this case we assume $\rho_1$ is complex.  We begin with a slightly different inequality:
\begin{align*}
0 & \leq \chi_0(\kn) \big(0+10 \cdot\Re\{  (\N\kn)^{-i\gamma_1}\} \big)^2 \big( 7 + 10 \cdot \Re\{ (\N\kn)^{-i\gamma_1} \} \big)^2.
\end{align*}
Again using the ``polynomial method" with $P_4(X)$, it similarly follows that
\begin{equation}
\begin{aligned}
0 & \leq 620 \cdot \cP(\sigma, \chi_0)  +  1050  \cdot \cP(\sigma+i\gamma_1, \chi_0) +  745 \cdot \cP(\sigma+2i\gamma_1, \chi_0) \\
& \qquad +  350 \cdot \cP(\sigma + 3i\gamma_1, \chi_0) + 125 \cdot \cP(\sigma + 4i\gamma_1, \chi_0) \\
\end{aligned} 
\label{TrigIdentity-ZFR-Poly-P}
\end{equation}
where $\sigma = 1 + \tfrac{\lambda}{\cL}$ with $\lambda > 0$. Apply \cref{PolyEI-Apply} to each term above, extracting $\{ \rho_1, \bar{\rho_1}\}$ since $\rho_1$ is necessarily complex.   It then follows that
\begin{equation} 
0 \leq 620 \cdot P_4(1)  - 1050 \cdot P_4\big( \frac{\lambda}{\lambda+\lambda_1} \big) + A_0 + 2890\phi \lambda+ \epsilon
\label{ZFR-Poly-Identity-P}
\end{equation}
since $\cL_0 \leq \cL$, and where
\begin{align*}
A_0 & =   \ds \Re\{ 1050 \cdot  P_4\big( \frac{\lambda}{\lambda+i\mu_1} \big)  - 1365 \cdot P_4\big( \frac{\lambda}{\lambda+\lambda_1 + i\mu_1} \big)  \} \\
 \vspace*{2mm}
&  \ds \qquad + \Re\{ 745 \cdot P_4\big( \frac{\lambda}{\lambda+2i\mu_1} \big)   - 1400 \cdot P_4\big( \frac{\lambda}{\lambda+\lambda_1 + 2i\mu_1} \big)  \} \\
&  \ds \qquad + \Re\{ 350 \cdot P_4\big( \frac{\lambda}{\lambda+3i\mu_1} \big)   - 870 \cdot P_4\big( \frac{\lambda}{\lambda+\lambda_1 + 3i\mu_1} \big)  \}  \\
&  \ds \qquad + \Re\{ 125 \cdot P_4\big( \frac{\lambda}{\lambda+4i\mu_1} \big)  - 350 \cdot P_4\big( \frac{\lambda}{\lambda+\lambda_1 + 4i\mu_1} \big)  \}.
\end{align*}
Applying \cref{PmLemma} to each term of $A_0$, it follows that $A_0 \leq 0$ provided

\begin{equation}
\frac{1050}{\lambda^4} -  \frac{1365}{(\lambda+\lambda_1)^4} \leq 0.
\label{ZFR-Poly-Condition-2} 
\end{equation}
Thus, \eqref{ZFR-Poly-Identity-P}  implies
\[
0 \leq 620 \cdot P_4(1)  - 1050 \cdot P_4\big( \frac{\lambda}{\lambda+\lambda_1} \big) + 2890\phi \lambda + \epsilon
\]
provided \eqref{ZFR-Poly-Condition-2} is satisfied. Taking $\lambda = 1.291$ yields $\lambda_1 \geq 0.0875$. This completes the proof of \cref{MainTheorem-ZFR}. \hfill \qed

 
\bibliographystyle{alpha}
\bibliography{biblio} 

\end{document}